\documentclass[hidelinks, 10pt]{article}

\usepackage[UKenglish]{babel}
\usepackage[utf8]{inputenc}
\usepackage[T1]{fontenc}
\usepackage{erewhon}
\usepackage{amsmath}
\usepackage{amsthm}
\usepackage{amsfonts}
\usepackage{amssymb}
\usepackage{mathtools}
\usepackage{eucal}  % for a nicer mathcal
\usepackage{dsfont} % for identity
\usepackage{authblk} % For authors
\usepackage{fullpage}
\usepackage{microtype}

\usepackage{fix-cm}
\usepackage[dvipsnames]{xcolor}
\usepackage{hyperref}
\usepackage{newpxmath}
\usepackage{booktabs}
\usepackage{enumitem} % to customize description lists
\usepackage{comment}

\usepackage{adjustbox}
\usepackage{quiver}

\tikzset{node distance=2cm, auto}
\tikzcdset{row sep/normal=2.7em,column sep/normal=3.5em}

\allowdisplaybreaks

\setlist[description]{font=\normalfont\bfseries\:\!}

\renewcommand\labelenumi{(\alph{enumi})}
\renewcommand\theenumi\labelenumi

\hypersetup{
    colorlinks=true,
    linkcolor=gray,
    filecolor=gray,      
    urlcolor=gray,
    citecolor=gray
    }
\numberwithin{equation}{section}

\newtheorem{theorem}{Theorem}[section]
\newtheorem{proposition}[theorem]{Proposition}
\newtheorem{lemma}[theorem]{Lemma}
\newtheorem{corollary}[theorem]{Corollary}

\theoremstyle{definition}
\newtheorem{definition}[theorem]{Definition}
\newtheorem{example}[theorem]{Example}

\newtheorem{remark}[theorem]{Remark}

\newenvironment{content}[1][blank]{\noindent{\large\bfseries\scshape\MakeLowercase{#1}}.\thinspace\noindent}{\vskip .5cm}
\newenvironment{acknowledgements}{\begin{content}[Acknowledgements]}{\end{content}}

\renewenvironment{abstract}{\begin{content}[Abstract]}{\par\noindent\rule{\textwidth}{1pt}\end{content}}

\newcommand{\CategoryFont}[1]{\mathsf{\uppercase{#1}}}
\newcommand{\FunctorFont}[1]{\mathsf{#1}}

\renewcommand{\|}{\:|\:}

\renewcommand{\bar}[1]{\mkern 1.5mu\overline{\mkern-1.5mu#1\mkern-1.5mu}\mkern 1.5mu}

\newcommand{\x}{\otimes}

\renewcommand{\,}{,\dots,}

\newcommand{\1}{\mathds{1}}

\renewcommand{\=}{\mathrel{\mathop:}=}
\renewcommand{\epsilon}{\varepsilon}
\renewcommand{\o}{\circ}
\renewcommand{\b}{\bullet}
\newcommand{\<}{\langle}
\renewcommand{\>}{\rangle}
\newcommand{\op}{\mathsf{op}}

\renewcommand{\*}{\ast}

\newcommand{\Inc}{\mathsf{Inc}}

\newcommand{\mathdash}{\relbar\mkern-5mu\relbar\mkern-5mu\relbar}

\newcommand{\id}{\mathsf{id}}

\newcommand{\CC}{{\mathbf{K}}}
\newcommand{\DD}{{\mathbf{J}}}
\newcommand{\X}{\mathbb{X}}

\newcommand{\End}{\CategoryFont{End}}

\newcommand{\Cat}{\CategoryFont{Cat}}
\newcommand{\Mnd}{\CategoryFont{Mnd}}

\newcommand{\Fib}{\CategoryFont{Fib}}
\newcommand{\Indx}{\CategoryFont{Indx}}
\newcommand{\RestrCat}{\CategoryFont{RCat}}
\newcommand{\sRestrCat}{\CategoryFont{sRCat}}

\newcommand{\TngCat}{\CategoryFont{TngCat}}
\newcommand{\cTngCat}{\CategoryFont{cTngCat}}
\newcommand{\TngMnd}{\CategoryFont{TngMnd}}

\newcommand{\T}{\mathrm{T}}
\newcommand{\TT}{\mathbb{T}}
\newcommand{\Tng}{\CategoryFont{Tng}}

\newcommand{\cTng}{\CategoryFont{CTng}}
\newcommand{\Cart}{\CategoryFont{Cart}}

\newcommand{\q}{\mathbf{q}}
\newcommand{\p}{\mathbf{p}}
\newcommand{\TngFib}{\CategoryFont{TngFib}}
\newcommand{\TngIndx}{\CategoryFont{TngIndx}}
\newcommand{\TngRestrCat}{\CategoryFont{TngRCat}}

\newcommand{\U}{\CategoryFont{U}}

\newcommand{\Univ}[1]{\mathcal{\uppercase{#1}}}
\newcommand{\UnivQ}{\Univ{Q}}
\newcommand{\UnivZ}{\Univ{Z}}
\newcommand{\UnivS}{\Univ{S}}

\newcommand{\Base}{\FunctorFont{Base}}
\newcommand{\Tot}{\FunctorFont{Tot}}
\newcommand{\pt}{{\mathsf{pt}}}

\renewcommand{\H}{\mathsf{H}}
\newcommand{\Curv}{\rho}
\newcommand{\Tors}{\tau}
\newcommand{\Riem}{\mathsf{R}}
\newcommand{\TorsTens}{\mathsf{W}}

\newcommand{\Dsply}{\mathscr{D}}
\newcommand{\Split}{\FunctorFont{Split}}
\newcommand{\Parall}{\FunctorFont{Prll}}

\newcommand{\Equif}{\CategoryFont{Equif}}

\newcommand{\IND}{\mathscr{I}}
\newcommand{\I}{\mathbb{I}}
\newcommand{\E}{\mathbb{E}}

\newcommand{\VF}{\CategoryFont{VF}}

\newcommand{\DO}{\CategoryFont{DO}}
\newcommand{\DB}{\CategoryFont{DB}}

\newcommand{\LS}{\CategoryFont{LS}}
\newcommand{\HAC}{\CategoryFont{HAC}}
\newcommand{\HLC}{\CategoryFont{HLC}}
\newcommand{\VAC}{\CategoryFont{VAC}}
\newcommand{\VLC}{\CategoryFont{VLC}}
\newcommand{\AC}{\CategoryFont{AC}}
\newcommand{\LC}{\CategoryFont{LC}}
\newcommand{\LSC}{\CategoryFont{LSC}}

\title{The formal theory of tangentads\\ {\Large PART II}}
\date{}
\author{\uppercase{Marcello Lanfranchi}}
\affil{\normalsize\textit{Macquarie University, School of Mathematical and Physical Sciences}}

\makeatother
\makeatletter
\begin{document}

\maketitle

%%%%%%%%%%%%%%%%%%%%%%%%%%% ABSTRACT %%%%%%%%%%%%%%%%%%%%%%%%%%%%%%%%%%%%%%
\begin{abstract}\noindent
Tangent category theory is a well-established categorical framework for differential geometry. A long list of fundamental geometric constructions, such as the tangent bundle functor, vector fields, Euclidean spaces, and vector bundles have been successfully generalized and internalized within tangent categories. Over the past decade, the theory has also been extended in several directions, yielding concepts such as tangent monads, tangent fibrations, tangent restriction categories, and reverse tangent categories. It is natural to wonder how these new flavours of the theory interact with the geometric constructions. How does a tangent monad or a tangent fibration lift to the tangent category of differential bundles of a tangent category? What is the correct notion of connections for a tangent restriction category? In previous work, we introduced tangentads, a unifying framework that generalizes many tangent-like notions, and developed a formal theory of vector fields for tangentads. In this paper, we extend this formal theory to three further fundamental constructions. These are differential objects, which generalize Euclidean spaces, differential bundles, which represent vector bundles in tangent category theory, and connections on differential bundles, which are the analogue of Koszul connections. These notions are introduced in the general theory of tangentads via appropriate universal properties. We then extend some of the main results of tangent category theory, including the equivalence between differential objects and differential bundles over the terminal object, and show that connections admit well-defined notions of covariant derivative, curvature, and torsion. Finally, we construct connections using PIE limits and apply our framework to several concrete instances of tangentads.
\end{abstract}

%__________________________________________________________________________
%__________________________________________________________________________

\tableofcontents

\section{Introduction}
\label{section:introduction}
In recent years, tangent category theory, a well-established categorical framework for differential geometry, has been extended in new directions. In~\cite{cockett:differential-bundles}, the notion of tangent fibrations was introduced to study differential bundles. Tangent monads were studied in~\cite{cockett:tangent-monads} as the tangent-categorical analogue of monads, while tangent restriction categories were introduced in~\cite{cockett:tangent-cats} to take into account the partiality of morphisms.

\par It is natural to ask whether these new contexts admit extensions of the constructions and results of tangent category theory. Furthermore, one might hope to develop such extensions in a way that not only encompasses the existing flavours of tangent category theory, but can also be applied in future investigations in other settings.

\par In previous papers~\cite{lanfranchi:tangentads-II}, we developed a formal theory of vector fields by generalizing several results from the theory of vector fields to tangent fibrations, tangent monads, and tangent restriction categories. Our approach was, in fact, more general. We introduced a notion of vector fields for any tangentad, which is a $2$-categorical formalization of a tangent category. We proved that vector fields in this general context carry both a Lie algebra structure and a $2$-monad structure. We also showed how to construct vector fields in the presence of PIE limits, which are a special class of $2$-limits.

\par In this paper, we develop a formal theory for three other important constructions of tangent category theory, that is differential objects, differential bundles, and connections on differential bundles. Differential objects are the analogue in tangent categories of Euclidean spaces, that is, spaces whose tangent bundle is isomorphic to two copies of the space itself. Differential bundles are the analogue of vector bundles in differential geometry, and connections on differential bundles correspond to Koszul connections on vector bundles.

\par We begin by recalling the definition of tangentads, introduced in previous work~\cite{lanfranchi:tangentads-I}, which provide an abstraction of tangent categories in the general context of $2$-category theory. In particular, tangent fibrations, tangent monads, and tangent restriction categories are all examples of tangentads in suitable $2$-categories. We also recall the definition of Cartesian tangentads, which are Cartesian objects in the $2$-category of tangentads.

\par After briefly recalling the main definitions and results concerning differential objects in the context of tangent categories, we present the universal property satisfied by this construction and use it to characterize differential objects for arbitrary Cartesian tangentads. We prove that this construction is $2$-functorial and compute it for several examples of tangentads.

\par We then carry out analogous developments for differential bundles and for connections: we recall the main definitions and results in the context of tangent categories; we present the universal properties enjoyed by these constructions, and extend to the formal context some results. We prove that, under mild conditions, a tangentad admits the construction of differential objects whenever it admits the construction of differential bundles, generalizing the important theorem that establishes that differential bundles over the terminal object are equivalent to differential objects. We construct a notion of curvature and torsion for connections, and construct connections using PIE limits.

\vskip .5cm

\begin{acknowledgements}
The author would like to thank JS Lemay, Richard Garner, and Steve Lack for useful discussions and suggestions on this work. This material is based upon work supported by the AFOSR under award number FA9550-24-1-0008.
\end{acknowledgements}

%__________________________________________________________________________
%__________________________________________________________________________

\section{Recollection on tangentads}
\label{section:tangentads}
Tangentads were initially introduced in~\cite{lanfranchi:grothendieck-tangent-cats} with the name of tangent objects, to establish a Grothendieck construction in the context of tangent categories. In~\cite{lanfranchi:tangentads-I}, this notion was revisited and proved to be suitable to capture a long list of existing concepts in tangent category theory, such as tangent fibrations, tangent monads, tangent split restriction categories, reverse tangent categories, strong display tangent categories, and infinitesimal objects. In this section, we briefly review the definition of a tangentad in a $2$-category.

%__________________________________________________________________________
\subsection{The definition of a tangentad}
\label{subsection:definition-tangentad}
The notion of tangentads was inspired by Leung's approach to tangent categories~\cite{leung:weil-algebras}. In this section, we recall this definition in a more explicit form.

\begin{definition}[{\cite[Definition~4.3]{lanfranchi:grothendieck-tangent-cats}}]
\label{definition:tangentad}
A tangentad in a $2$-category $\CC$ consists of the following data:
\begin{description}
\item[Base object] An object $\X$ of $\CC$;

\item[Tangent bundle $1$-morphism] A $1$-endomorphism $\T\colon\X\to\X$;

\item[Projection]  A $2$-morphism $p\colon\T\Rightarrow\id_\X$, which admits all pointwise $n$-fold pullbacks
\begin{equation*}
% https://q.uiver.app/#q=WzAsNCxbMCwwLCJcXFRfbiJdLFswLDEsIlxcVCJdLFsxLDAsIlxcVCJdLFsxLDEsIlxcaWRfXFxYIl0sWzEsMywicCIsMl0sWzIsMywicCJdLFswLDEsIlxccGlfMSIsMl0sWzAsMiwiXFxwaV9uIl0sWzAsMywiIiwxLHsic3R5bGUiOnsibmFtZSI6ImNvcm5lciJ9fV0sWzIsMSwiXFxkb3RzIiwzLHsib2Zmc2V0IjozLCJzdHlsZSI6eyJib2R5Ijp7Im5hbWUiOiJub25lIn0sImhlYWQiOnsibmFtZSI6Im5vbmUifX19XV0=
\begin{tikzcd}
	{\T_n} & \T \\
	\T & {\id_\X}
	\arrow["{\pi_n}", from=1-1, to=1-2]
	\arrow["{\pi_1}"', from=1-1, to=2-1]
	\arrow["\lrcorner"{anchor=center, pos=0.125}, draw=none, from=1-1, to=2-2]
	\arrow["\dots"{marking, allow upside down}, shift right=3, draw=none, from=1-2, to=2-1]
	\arrow["p", from=1-2, to=2-2]
	\arrow["p"', from=2-1, to=2-2]
\end{tikzcd}
\end{equation*}
which are preserved  by all iterates $\T^n$ of $\T$;

\item[Zero morphism] A $2$-morphism $z\colon\id_\X\Rightarrow\T$;

\item[Sum morphism] A $2$-morphism $s\colon\T_2\Rightarrow\T$;
\end{description}
such that, $\p\=(p,z,s)$ is a $\bar\T$-additive bundle of $\End(\X)$, where $\bar\T$ sends an endomorphism $F\colon\X\to\X$ to $\T\o F$;
\begin{description}
\item[Vertical lift] A $2$-morphism $l\colon\T\Rightarrow\T^2$ such that:
\begin{align*}
&(z,l)\colon\p\to\T\p
\end{align*}
is an additive bundle morphism;

\item[Canonical flip] A $2$-morphism $c\colon\T^2\Rightarrow\T^2$ such that:
\begin{align*}
&(\id_\X,c)\colon\T\p\to\p_\T
\end{align*}
is an additive bundle morphism;
\end{description}
subject to the following conditions:
\begin{enumerate}
\item The vertical lift is coassocative, and compatible with the canonical flip:
\begin{equation*}
% https://q.uiver.app/#q=WzAsNCxbMCwwLCJcXFQiXSxbMSwwLCJcXFReMiJdLFsxLDEsIlxcVF4zIl0sWzAsMSwiXFxUXjIiXSxbMCwxLCJsIl0sWzEsMiwiXFxUIGwiXSxbMCwzLCJsIiwyXSxbMywyLCJsXFxUIiwyXV0=
\begin{tikzcd}
{\T} & {\T^2} \\
{\T^2} & {\T^3}
\arrow["l", from=1-1, to=1-2]
\arrow["l"', from=1-1, to=2-1]
\arrow["{\T l}", from=1-2, to=2-2]
\arrow["{l_\T}"', from=2-1, to=2-2]
\end{tikzcd}\hfill\quad
% https://q.uiver.app/#q=WzAsMyxbMCwwLCJcXFQiXSxbMSwwLCJcXFReMiJdLFsxLDEsIlxcVF4yIl0sWzAsMSwibCJdLFsxLDIsImMiXSxbMCwyLCJsIiwyXV0=
\begin{tikzcd}
{\T} & {\T^2} \\
& {\T^2}
\arrow["l", from=1-1, to=1-2]
\arrow["l"', from=1-1, to=2-2]
\arrow["c", from=1-2, to=2-2]
\end{tikzcd}\hfill\quad
% https://q.uiver.app/#q=WzAsNSxbMSwwLCJcXFReMyJdLFsyLDAsIlxcVF4zIl0sWzIsMSwiXFxUXjMiXSxbMCwwLCJcXFReMiJdLFswLDEsIlxcVF4yIl0sWzAsMSwiXFxUIGMiXSxbMSwyLCJjX1xcVCJdLFszLDAsImxfXFxUIl0sWzMsNCwiYyIsMl0sWzQsMiwiXFxUIGwiLDJdXQ==
\begin{tikzcd}
{\T^2} & {\T^3} & {\T^3} \\
{\T^2} && {\T^3}
\arrow["{l_\T}", from=1-1, to=1-2]
\arrow["c"', from=1-1, to=2-1]
\arrow["{\T c}", from=1-2, to=1-3]
\arrow["{c_\T}", from=1-3, to=2-3]
\arrow["{\T l}"', from=2-1, to=2-3]
\end{tikzcd}
\end{equation*}

\item The canonical flip is a symmetric braiding:
\begin{equation*}
% https://q.uiver.app/#q=WzAsMyxbMCwwLCJcXFReMiJdLFsxLDAsIlxcVF4yIl0sWzEsMSwiXFxUXjIiXSxbMCwxLCJjIl0sWzEsMiwiYyJdLFswLDIsIiIsMix7ImxldmVsIjoyLCJzdHlsZSI6eyJoZWFkIjp7Im5hbWUiOiJub25lIn19fV1d
\begin{tikzcd}
{\T^2} & {\T^2} \\
& {\T^2}
\arrow["c", from=1-1, to=1-2]
\arrow[equals, from=1-1, to=2-2]
\arrow["c", from=1-2, to=2-2]
\end{tikzcd}\hfill\quad
% https://q.uiver.app/#q=WzAsNixbMCwwLCJcXFReMyJdLFsxLDAsIlxcVF4zIl0sWzEsMSwiXFxUXjMiXSxbMCwxLCJcXFReMyJdLFsyLDAsIlxcVF4zIl0sWzIsMSwiXFxUXjMiXSxbMCwxLCJcXFQgYyJdLFswLDMsImNfXFxUIiwyXSxbMywyLCJcXFQgYyIsMl0sWzEsNCwiY19cXFQiXSxbNCw1LCJcXFQgYyJdLFsyLDUsImNfXFxUIiwyXV0=
\begin{tikzcd}
{\T^3} & {\T^3} & {\T^3} \\
{\T^3} & {\T^3} & {\T^3}
\arrow["{\T c}", from=1-1, to=1-2]
\arrow["{c_\T}"', from=1-1, to=2-1]
\arrow["{c_\T}", from=1-2, to=1-3]
\arrow["{\T c}", from=1-3, to=2-3]
\arrow["{\T c}"', from=2-1, to=2-2]
\arrow["{c_\T}"', from=2-2, to=2-3]
\end{tikzcd}
\end{equation*}

\item The tangent bundle is locally linear, that is, the following is a pointwise pullback diagram:
\begin{equation*}
% https://q.uiver.app/#q=WzAsNSxbMiwwLCJcXFReMk0iXSxbMiwxLCJcXFQgTSJdLFswLDEsIk0iXSxbMCwwLCJcXFRfMk0iXSxbMSwwLCJcXFRcXFRfMk0iXSxbMCwxLCJcXFQgcCJdLFsyLDEsInoiLDJdLFszLDIsIlxccGlfMXAiLDJdLFszLDQsImxcXHRpbWVzX3p6X1xcVCJdLFs0LDAsIlxcVCBzIl1d
\begin{tikzcd}
{\T_2} & {\T\o\T_2} & {\T^2} \\
\id_\X && {\T}
\arrow["{z_\T\times l}", from=1-1, to=1-2]
\arrow["{\pi_1p}"', from=1-1, to=2-1]
\arrow["{\T s}", from=1-2, to=1-3]
\arrow["{\T p}", from=1-3, to=2-3]
\arrow["z"', from=2-1, to=2-3]
\end{tikzcd}
\end{equation*}
\end{enumerate}
Furthermore, a tangentad admits negatives when it is equipped with:
\begin{description}
\item[Negation] A $2$-morphism $n\colon\T\Rightarrow\T$ such that:
\begin{equation*}
% https://q.uiver.app/#q=WzAsNCxbMCwwLCJcXFQiXSxbMSwwLCJcXFRfMiJdLFsxLDEsIlxcVCJdLFswLDEsIlxcaWRfXFxYIl0sWzAsMSwiXFw8bixcXGlkX1xcVFxcPiJdLFsxLDIsInMiXSxbMCwzLCJwIiwyXSxbMywyLCJ6IiwyXV0=
\begin{tikzcd}
\T & {\T_2} \\
{\id_\X} & \T
\arrow["{\<n,\id_\T\>}", from=1-1, to=1-2]
\arrow["p"', from=1-1, to=2-1]
\arrow["s", from=1-2, to=2-2]
\arrow["z"', from=2-1, to=2-2]
\end{tikzcd}
\end{equation*}
\end{description}
\end{definition}

We suggest the reader consult~\cite{lanfranchi:tangentads-I} for more details on this definition and for more examples.

\begin{example}
\label{example:tangent-categories}
Tangent categories, initially introduced in~\cite{rosicky:tangent-cats}, are the archetypal example of tangentads. In particular, a tangent category is precisely a tangentad in the $2$-category $\Cat$ of categories.
\end{example}

\begin{example}
\label{example:tangent-monads}
Tangent monads, introduced in~\cite[Definition~19]{cockett:tangent-monads}, are monads in the $2$-category of tangent categories. Concretely, a tangent monad on a tangent category $(\X,\TT)$ consists of a monad $S\colon\X\to\X$ equipped with a natural transformation $\alpha_M\colon S\T M\to\T SM$, for $M\in\X$, such that $(S,\alpha)\colon(\X,\TT)\to(\X,\TT)$ is a lax tangent morphism. Furthermore, the unit $\eta_M\colon M\to SM$ and the multiplication $\mu_M\colon S^2M\to SM$ of the monad $S$ are tangent natural transformations, that is, they are compatible with $\alpha$.
\par More generally, a tangent monad internal to a $2$-category $\CC$ is monad in the $2$-category $\Tng(\CC)$ of tangentads of $\CC$.
\par \cite[Proposition~4.2]{lanfranchi:tangentads-I} shows that, for a $2$-category $\CC$, the $2$-category $\Mnd(\Tng(\CC))$ of monads in the $2$-category of tangentads of $\CC$ is isomorphic to the $2$-category $\Tng(\Mnd(\CC))$ of tangentads in the $2$-category of monads of $\CC$. In particular, tangent monads are tangentads in the $2$-category $\Mnd\=\Mnd(\Cat)$ of monads.
\end{example}

\begin{example}
\label{example:tangent-fibrations}
Tangent fibrations, introduced in~\cite[Definition~5.2]{cockett:differential-bundles}, are (cloven) fibrations $\Pi\colon(\X',\TT')\to(\X,\TT)$ between two tangent categories which preserve the tangent structures strictly, that is, $\Pi$ strictly preserves the tangent structures, and the tangent bundle functors define a Cartesian morphism of fibrations.
\par \cite[Proposition~5.1]{lanfranchi:grothendieck-tangent-cats} proves that tangent fibrations are tangentads in the $2$-category $\Fib$. Concretely, the objects of $\Fib$ are triples $(\X,\X';\Pi)$ formed by two categories $\X$ and $\X'$ and a fibration $\Pi\colon\X'\to\X$. $1$-morphisms of $\Fib$ $(F,F')\colon(\X_\o,\X_\o';\Pi_\o)\to(\X_\b,\X'_\b;\Pi_\b)$ are pairs of functors $F\colon\X_\o\to\X_\b$ and $F'\colon\X_\o'\to\X_\b'$ which strictly commute with the fibrations, that is, $\Pi_\b\o F'=F\o\Pi_\o$, and preserve the Cartesian lifts, that is, $F'$ sends each Cartesian lift $\varphi_f\colon f^\*E\to E$ of a morphism $f\colon A\to B$ of $\X_\o$ to a Cartesian lift of $Ff$. Finally, $2$-morphisms of $\Fib$
\begin{align*}
&(\varphi,\varphi')\colon(F,F')\Rightarrow(G,G')\colon(\X_\o,\X_\o';\Pi_\o)\to(\X_\b,\X_\b';\Pi_\b)
\end{align*}
are pairs of natural transformations $\varphi\colon F\Rightarrow G$ and $\varphi'\colon F'\Rightarrow G'$ which commute with the fibrations, that is, $\Pi_\b\varphi'=\varphi_{\Pi_\o}$.
\end{example}

\begin{example}
\label{example:tangent-indexed-categories}
Tangent indexed categories were introduced in~\cite[Definition~5.3]{lanfranchi:grothendieck-tangent-cats} as tangentads in the $2$-category $\Indx$ of indexed categories. Concretely, the objects of $\Indx$ are pairs $(\X,\IND)$ formed by a category $\X$ and an indexed category, a.k.a., a pseudofunctor $\X^\op\to\Cat$. The $1$-morphisms $(F,F',\xi)\colon(\X_\o,\IND_\o)\to(\X_\b,\IND_\b)$ of $\Indx$ are pseudonatural tranformations, that are, functors $F\colon\X_\o\to\X_\b$ together with a collection of functors $F_A'\colon\IND_\o(A)\to\IND_\b(FA)$, indexed by the objects $A$ of $\X_\o$, and a collection of natural isomorphisms $\xi^f\colon F_A'\o\IND_\o(f)\Rightarrow\IND_\b(Ff)\o F_B'$, called distributors, indexed by the morphisms $f\colon A\to B$ of $\X_\o$. Finally, the $2$-morphisms of $\Indx$
\begin{align*}
&(\varphi,\varphi')\colon(F,F',\xi)\to(G,G',\kappa)\colon(\X_\o,\IND_\o)\to(\X_\b,\IND_\b)
\end{align*}
consist of a natural transformation $\varphi\colon F\Rightarrow G$ together with a collection $\varphi'$ of natural transformations $\varphi_A'\colon F_A'\Rightarrow G_A'\o\IND_\b(\varphi_A)$ which are compatible with the distributors.
\par A tangent indexed category, that is, a tangentad in $\Indx$, consists of the following data:
\begin{description}
\item[Base tangent category] A tangent category $(\X,\TT)$;

\item[Indexed category] An indexed category $\IND\colon\X^\op\to\Cat$ which sends each object $M\in\X$ to a category $\X^M$ and each morphism $f\colon M\to N$ to a functor $f^\*\X^N\to\X^N$;

\item[Indexed tangent functor] A list of functors $\T'^M\colon\X^M\to\X^{\T M}$, indexed by the objects of $M$, together with a list of natural isomorphisms $\xi^f_E\colon(\T f)^\*\T'^NE\to\T'^Mf^\*E$, indexed by the morphisms $f\colon M\to N$ of $\X$;

\item[Indexed projection] A list of natural transformations $p'^M_E\colon\T'^ME\to p^\*E$;

\item[Indexed zero morphism] A list of natural transformations $z'^M_E\colon E\to z^\*\T'^ME$;

\item[Indexed sum morphism] A list of natural transformations $s'^M_E\colon\T'^M_2E\to s^\*\T'^ME$;

\item[Indexed vertical lift] A list of natural transformations $l'^M_E\colon\T'^ME\to l^\*\T'^{\T M}\T'^ME$;

\item[Indexed canonical flip] A list of natural transformations $c'^M_E\colon\T'^{\T M}\T'^ME\to c^\*\T'^{\T M}\T'^ME$.
\end{description}
The indexed projection, zero morphism, sum morphism, vertical lift, and canonical lift must also be compatible with the distributors of $\T'$ and need to satisfy the axioms of a tangent structure.
\par Finally, we mention that \cite[Theorem~5.5]{lanfranchi:grothendieck-tangent-cats} shows a $2$-equivalence between the $2$-category $\TngFib$ of tangent fibrations and the $2$-category $\TngIndx$ of tangent indexed categories.
\end{example}

\begin{example}
\label{example:tangent-restriction-categories}
Tangent restriction categories, introduced in~\cite[Definition~6.14]{cockett:tangent-cats}, are restriction categories equipped with an endofunctor $\T$ which preserves the restriction idempotents, together with structural total natural transformations $p_M\colon\T M\to M$, $z_M\colon M\to\T M$, $s_M\colon\T_2M\to\T M$, $l_M\colon\T M\to\T^2M$, and $c_M\colon\T^2M\to\T^2M$ similar to the structural natural transformations of a tangent structure and that satisfy similar axioms, but in which the $n$-fold pullback of $p$ along itself and the pullback diagram of the universality of the vertical lift are replaced with restriction pullbacks. We suggest the reader consult~\cite{cockett:restrictionI} for a definition of restriction category and~\cite{cockett:restrictionIII} for a definition of restriction limits. As noticed in~\cite[Section~4.6]{lanfranchi:tangentads-I}, tangent restriction categories are not an example of tangentads. Nevertheless, as proved in~\cite[Proposition~4.35]{lanfranchi:tangentads-I}, tangent split restriction categories, that is, tangent restriction categories whose restriction idempotents split, are precisely the tangentads in the $2$-category $\sRestrCat$ of split restriction categories, restriction functors, and total natural transformations.
\par As proved by~\cite[Lemma~4.36]{lanfranchi:tangentads-I}, every tangent restriction category embeds into a tangent split restriction category. Thus, in particular, every tangent restriction category embeds into a tangentad.
\end{example}

%__________________________________________________________________________
\subsection{The pointwise tangent structure on the Hom categories}
\label{subsection:hom-categories}
In this section, we recall a construction presented in~\cite{lanfranchi:tangentads-II}. We show that the Hom-categories in the $2$-category of tangentads of a given $2$-category $\CC$ come with a tangent structure. This plays a crucial role in our story.

\begin{proposition}
\label{proposition:hom-tangent-categories}
For two tangentads $(\X,\TT)$ and $(\X',\TT')$ (with negatives), the category of lax tangent morphisms $\Tng(\CC)[\X',\TT';\X,\TT]$ from $(\X',\TT')$ to $(\X,\TT)$ comes with a tangent structure (with negatives) defined pointwise. Concretely, the tangent bundle functor sends a lax tangent morphism $(F,\alpha)\colon(\X',\TT')\to(\X,\TT)$ to the lax tangent morphism:
\begin{align*}
&\bar\T(F,\alpha)\colon(\X',\TT')\xrightarrow{(F,\alpha)}(\X,\TT)\xrightarrow{(\T,c)}(\X,\TT)
\end{align*}
Moreover, $\bar\T$ sends a morphism $\varphi\colon(F,\alpha)\to(G,\beta)$ of lax tangent morphisms to the morphism:
\begin{align*}
&\bar\T\varphi\=\T\varphi\colon\bar\T(F,\alpha)\to\bar\T(G,\beta)
\end{align*}
Furthermore, the Hom-category $2$-functor $\Tng(\CC)(-,?)$ lifts along the forgetful functor $\U\colon\TngCat\to\Cat$ to a $2$-functor
\begin{align*}
&\Tng(\CC)^\op\times\Tng(\CC)\to\TngCat
\end{align*}
strict and contravariant in the first argument and lax and covariant in the second.
\end{proposition}

In the following, we denote by $[\X',\TT'\|\X,\TT]$ the Hom-tangent category of two tangentads $(\X',\TT')$ and $(\X,\TT)$ of a given $2$-category.

%__________________________________________________________________________
\subsection{Cartesian tangentads}
\label{subsection:cartesian-tangentads}
In~\cite[Section~3.2]{lanfranchi:tangentads-I}, Cartesian tangentads were introduced. In this section, we recall this definition. In the following, we assume that $\CC$ is a Cartesian $2$-category, that is, a $2$-category with finite $2$-products and we denote by $\1$ the $2$-terminal object of $\CC$ and by $\X\times\X'$ the Cartesian $2$-product of two objects $\X$ and $\X'$ of $\CC$.

\begin{definition}[{\cite[Section~5]{carboni:cartesian-objects}}]
\label{definition:cartesian-object}
In a Cartesian $2$-category $\CC$, a \textbf{Cartesian object} consists of an object $\X$ of $\CC$ together with two $1$-morphisms
\begin{align*}
&\*\colon\1\to\X\\
&\times\colon\X\times\X\to\X
\end{align*}
which are right adjoints to the terminal $1$-morphism $!\colon\X\to\1$ and to the diagonal map $\Delta\colon\X\to\X\times\X$, respectively.
\end{definition}

Cartesian objects in the $2$-category of categories are precisely Cartesian categories, that is, categories with finite products. Cartesian objects in a $2$-category form a new $2$-category denoted by $\Cart(\CC)$ whose $1$-morphisms are $1$-morphisms of $\CC$ which preserve the Cartesian structure up to $2$-isomorphisms, and whose $2$-morphisms are $2$-morphisms of $\CC$ which commute with the coherence $2$-isomorphisms of the $1$-morphisms. Furthermore, if $\CC$ is a Cartesian $2$-category so does the $2$-category $\Tng(\CC)$ of tangentads of $\CC$. We suggest the reader consult~\cite{lanfranchi:tangentads-I} for details.

\begin{definition}[{\cite[Definition~3.18]{lanfranchi:tangentads-I}}]
\label{definition:cartesian-tangentad}
A \textbf{Cartesian tangentad} is equivalently a tangentad in the $2$-category $\Cart(\CC)$ of Cartesian objects of $\CC$ or a Cartesian object in the Cartesian $2$-category $\Tng(\CC)$ of tangentads of $\CC$.
\end{definition}

We denote by $\cTng(\CC)$ the $2$-category $\Cart(\Tng(\CC))\cong\Tng(\Cart(\CC))$ of Cartesian tangentads. Concretely, a Cartesian tangentad consists of a Cartesian object $\X$ of $\CC$ equipped with a tangent structure $\TT$ in $\CC$ whose tangent bundle $1$-morphism $\T\colon\X\to\X$ such that the $2$-morphisms
\begin{equation*}
% https://q.uiver.app/#q=WzAsMyxbMCwwLCJcXDEiXSxbMSwwLCJcXFgiXSxbMSwxLCJcXFgiXSxbMCwxLCJcXCoiXSxbMSwyLCJcXFQiXSxbMCwyLCJcXCoiLDIseyJjdXJ2ZSI6Mn1dLFsxLDUsIlxcYWxwaGFeXFwqIiwwLHsic2hvcnRlbiI6eyJ0YXJnZXQiOjEwfX1dXQ==
\begin{tikzcd}
\1 & \X \\
& \X
\arrow["{\*}", from=1-1, to=1-2]
\arrow[""{name=0, anchor=center, inner sep=0}, "{\*}"', curve={height=12pt}, from=1-1, to=2-2]
\arrow["\T", from=1-2, to=2-2]
\arrow["{\alpha^\*}", shorten <=5pt, shorten >=5pt, Rightarrow, from=1-2, to=0]
\end{tikzcd}\hfill\quad
% https://q.uiver.app/#q=WzAsNCxbMCwwLCJcXFhcXHRpbWVzXFxYIl0sWzEsMCwiXFxYIl0sWzEsMSwiXFxYIl0sWzAsMSwiXFxYXFx0aW1lc1xcWCJdLFswLDEsIlxcdGltZXMiXSxbMSwyLCJcXFQiXSxbMCwzLCJcXFRcXHRpbWVzXFxUIiwyXSxbMywyLCJcXHRpbWVzIiwyXSxbMSwzLCJcXGFscGhhXlxcdGltZXMiLDAseyJsZXZlbCI6Mn1dXQ==
\begin{tikzcd}
{\X\times\X} & \X \\
{\X\times\X} & \X
\arrow["\times", from=1-1, to=1-2]
\arrow["{\T\times\T}"', from=1-1, to=2-1]
\arrow["{\alpha^\times}", Rightarrow, from=1-2, to=2-1]
\arrow["\T", from=1-2, to=2-2]
\arrow["\times"', from=2-1, to=2-2]
\end{tikzcd}
\end{equation*}
induced by the adjunctions $\*\dashv!$ and $\times\dashv\Delta$ are invertible and compatible with the structural $2$-morphisms of the tangent structure.

%__________________________________________________________________________
\subsection{The Cartesian structure of the Hom-tangent categories}
\label{subsection:cartesian-hom-categories}
In this section, we show that the Hom-tangent categories inherit a Cartesian structure provided the target tangentad is Cartesian. Consider an ambient $2$-category $\CC$ which admits finite $2$-products, whose terminal object is denoted by $\1$ and Cartesian $2$-product by $\times$. A Cartesian object of $\CC$ is an object $\X$ of $\CC$ together with two $1$-morphisms
\begin{align*}
&\*\colon\1\to\X\\
&\times\colon\X\times\X\to\X
\end{align*}
which are right adjoints to the $1$-morphisms $!\colon\X\to\1$ and $\Delta\=\<\id_\X,\id_\X\>\colon\X\to\X\times\X$, respectively. The next result establishes that the Hom-categories $\CC[\X',\X]$ are Cartesian provided that $\X$ is a Cartesian object of $\CC$.

\begin{lemma}
\label{lemma:cartesian-hom-cats}
If $\X$ is a Cartesian object of $\CC$, the Hom-categories $\CC[\X',\X]$ are Cartesian, whose terminal object is the $1$-morphism
\begin{align*}
&\top\colon\X'\xrightarrow{!'}\1\xrightarrow{\*}\X
\end{align*}
and whose Cartesian product between two $1$-morphisms $A,B\colon\X'\to\X$ is the $1$-morphism
\begin{align*}
&A\Box B\colon\X'\xrightarrow{\<A,B\>}\X\times\X\xrightarrow{\times}\X
\end{align*}
and whose projections are the $2$-morphisms
\begin{align*}
&\pi_1\colon A\times B=\times\o\<A,B\>\xrightarrow{\pi_1\<A,B\>}\Pi_1\o\<A,B\>=A\\
&\pi_2\colon A\times B=\times\o\<A,B\>\xrightarrow{\pi_2\<A,B\>}\Pi_2\o\<A,B\>=B
\end{align*}
where $\Pi_1,\Pi_2\colon\X\times\X\to\X$ are the two projections and $\pi_k\colon\times\Rightarrow\Pi_k$ are the corresponding $2$-morphisms.
\end{lemma}
\begin{proof}
Consider the units and the counits of the adjunctions $(\eta^\*,\epsilon^\*)\colon!\dashv\*$ and $(\eta^\times,\epsilon^\times)\colon\Delta\dashv\times$ in $\CC$:
\begin{align*}
\eta^\*&\colon\id_\X\to\*\:\o\:!                &\epsilon^\*&\colon!\o\*\to\id_\1\\
\eta^\times&\colon\id_\X\to\times\o\Delta   &\epsilon^\times&\colon\Delta\o\times\to\id_{\X\times\X}
\end{align*}
Consider a $1$-morphism $A\colon\X'\to\X$ and define the following $2$-morphism
\begin{align*}
&A\xrightarrow{\eta^\*_A}\*\:\o\:!\o A=\*\:\o\:!'
\end{align*}
where we used that $!\o A=!'$, since $\1$ is $2$-terminal in $\CC$. If $\varphi\colon A\to\top$ is a $2$-morphism, by employing the counit $\epsilon^\*$ we define a $2$-morphism
\begin{align*}
&!\o A\xrightarrow{\varphi}!\o\*\:\o\:!\o A\xrightarrow{\epsilon^\*!A}!\o A
\end{align*}
However, since $\1$ is $2$-terminal in $\CC$, such a morphism must coincide with the identity on $!\o A$. Therefore, by employing the unit again and using the triangle identities, we conclude that $\varphi$ must coincide with $\eta^\*$, that is, that $\top$ is the terminal object of $\CC[\X',\X]$. To prove that $A\Box B$ is the Cartesian product of two $1$-morphisms $A,B\colon\X'\to\X$, first, notice that the counit $\epsilon^\times$ gives rise to the two projections
\begin{align*}
&\pi_1\colon\times=\Pi_1\o\Delta\o\times\xrightarrow{\Pi_1(\epsilon^\times)}\Pi_1\\
&\pi_2\colon\times=\Pi_2\o\Delta\o\times\xrightarrow{\Pi_2(\epsilon^\times)}\Pi_2
\end{align*}
where we used that $\Pi_k\o\Delta=\id_\X$ since $\Delta=\<\id_\X,\id_\X\>$, by definition. From this and using the triangle identities of $\eta^\times$ and $\epsilon^\times$, one deduces that
\begin{equation*}
% https://q.uiver.app/#q=WzAsMyxbMCwwLCJBIl0sWzEsMCwiQVxcQm94IEIiXSxbMiwwLCJCIl0sWzEsMCwiXFxwaV8xIiwyXSxbMSwyLCJcXHBpXzIiXV0=
\begin{tikzcd}
A & {A\Box B} & B
\arrow["{\pi_1}"', from=1-2, to=1-1]
\arrow["{\pi_2}", from=1-2, to=1-3]
\end{tikzcd}
\end{equation*}
is a product diagram in $\CC[\X',\X]$.
\end{proof}

\cite[Definition~3.13]{lanfranchi:tangentads-I} establishes that a Cartesian tangentad is a Cartesian object in the $2$-category $\Tng(\CC)$ of tangentads of $\CC$. From this observation, it follows by Lemma~\ref{lemma:cartesian-hom-cats} that the Hom-categories $[\X',\TT'\|\X,\TT]$ are Cartesian categories, provided that $(\X,\TT)$ is a Cartesian tangentad. One also needs the tangent structure of Proposition~\ref{proposition:hom-tangent-categories} to be compatible with the Cartesian structure.

\begin{proposition}
\label{proposition:cartesian-hom-tangent-cats}
The Hom-tangent categories $[\X',\TT'\|\X,\TT]$ are Cartesian tangent categories provided that $(\X,\TT)$ is a Cartesian tangentad.
\end{proposition}
\begin{proof}
Since the tangent structure and the Cartesian structure of $[\X',\TT'\|\X,\TT]$ are both induced by the tangent structure and the Cartesian structure of $(\X,\TT)$, which are compatible since it is a Cartesian object in $\Tng(\CC)$, it follows that $[\X',\TT'\|\X,\TT]$ is a Cartesian tangent category.
\end{proof}

\section{The formal theory of differential objects}
\label{section:differential-objects}
In every Cartesian tangent category, one can define a suitable class of objects, called \textbf{differential objects}, that generalize Euclidean spaces in differential geometry. This section aims to formalize this notion in the context of tangentads.

%__________________________________________________________________________
\subsection{Differential objects in tangent category theory}
\label{subsection:tangent-category-differential-objects}
We start by recalling the definition of a differential object.

\begin{definition}[{\cite[Definition~4.8]{cockett:tangent-cats}}]
\label{definition:differential-object}
A \textbf{differential object} in a Cartesian tangent category $(\X,\TT)$ is a four-tuple $(A,\zeta_A,\sigma_A,{\hat p}_A)$ consisting of an object $A$ of $\X$ and three morphisms, so defined:
\begin{description}
\item[Zero morphism] A morphism $\zeta_A\colon\*\to A$;

\item[Sum morphism] A morphism $\sigma_A\colon A\times A\to A$;
\end{description}
such that the triple $(A,\zeta_A,\sigma_A)$ is a commutative monoid in the Cartesian category $\X$. Moreover:
\begin{description}
\item[Differential projection] A morphism ${\hat p}_A\colon\T A\to A$;
\end{description}
which satisfies the following conditions:
\begin{enumerate}
\item Additivity 1. The morphism ${\hat p}_A\colon(\T A,\T\zeta_A,\T\sigma_A)\to(A,\zeta_A,\sigma_A)$ is a morphism of commutative monoids, where we identify $\T(A\times A)$ with $\T A\times\T A$ and $\T\*$ with $\*$. Concretely, this corresponds to the commutativity of the following diagrams:
\begin{equation*}
% https://q.uiver.app/#q=WzAsMyxbMCwwLCJcXFQgQSJdLFsxLDAsIkEiXSxbMCwxLCIqIl0sWzIsMCwiXFxUXFx6ZXRhIl0sWzIsMSwiXFx6ZXRhIiwyXSxbMCwxLCJcXGhhdCBwIl1d
\begin{tikzcd}
{\T A} & A \\
{*}
\arrow["{{\hat p}_A}", from=1-1, to=1-2]
\arrow["{\T\zeta_A}", from=2-1, to=1-1]
\arrow["\zeta_A"', from=2-1, to=1-2]
\end{tikzcd}\hfill\quad
% https://q.uiver.app/#q=WzAsNCxbMCwwLCJcXFQgQVxcdGltZXNcXFQgQSJdLFsxLDAsIkFcXHRpbWVzIEEiXSxbMCwxLCJcXFQgQSJdLFsxLDEsIkEiXSxbMCwxLCJcXGhhdCBwXFx0aW1lc1xcaGF0IHAiXSxbMiwzLCJcXGhhdCBwIiwyXSxbMCwyLCJcXFRcXHNpZ21hIiwyXSxbMSwzLCJcXHNpZ21hIl1d
\begin{tikzcd}
{\T A\times\T A} & {A\times A} \\
{\T A} & A
\arrow["{{\hat p}_A\times{\hat p}_A}", from=1-1, to=1-2]
\arrow["{\T\sigma_A}"', from=1-1, to=2-1]
\arrow["\sigma_A", from=1-2, to=2-2]
\arrow["{{\hat p}_A}"', from=2-1, to=2-2]
\end{tikzcd}
\end{equation*}

\item Additivity 2. The morphism $(!,{\hat p}_A)\colon(p\colon\T A\to A,z_A,s_A)\to(!\colon A\to\*,\zeta_A,\sigma_A)$ is an additive bundle morphism. Concretely, this corresponds to the commutativity of the following diagrams:
\begin{equation*}
% https://q.uiver.app/#q=WzAsNCxbMCwwLCJcXFQgQSJdLFsxLDAsIkEiXSxbMSwxLCIqIl0sWzAsMSwiQSJdLFswLDEsIlxcaGF0IHAiXSxbMiwxLCJcXHpldGEiLDJdLFszLDAsInpfQSJdLFszLDIsIiEiLDJdXQ==
\begin{tikzcd}
{\T A} & A \\
A & {*}
\arrow["{{\hat p}_A}", from=1-1, to=1-2]
\arrow["{z_A}", from=2-1, to=1-1]
\arrow["{!}"', from=2-1, to=2-2]
\arrow["\zeta_A"', from=2-2, to=1-2]
\end{tikzcd}\hfill\quad
% https://q.uiver.app/#q=WzAsNCxbMCwwLCJcXFRfMkEiXSxbMSwwLCJBXFx0aW1lcyBBIl0sWzAsMSwiXFxUIEEiXSxbMSwxLCJBIl0sWzAsMSwiXFxoYXQgcFxcdGltZXNfIVxcaGF0IHAiXSxbMiwzLCJcXGhhdCBwIiwyXSxbMCwyLCJzX0EiLDJdLFsxLDMsIlxcc2lnbWEiXV0=
\begin{tikzcd}
{\T_2A} & {A\times A} \\
{\T A} & A
\arrow["{{\hat p}_A\times_!{\hat p}_A}", from=1-1, to=1-2]
\arrow["{s_A}"', from=1-1, to=2-1]
\arrow["\sigma_A", from=1-2, to=2-2]
\arrow["{{\hat p}_A}"', from=2-1, to=2-2]
\end{tikzcd}
\end{equation*}

\item Linearity. The following diagram commutes:
\begin{equation*}
% https://q.uiver.app/#q=WzAsNCxbMSwwLCJcXFQgQSJdLFsxLDEsIkEiXSxbMCwwLCJcXFReMkEiXSxbMCwxLCJcXFQgQSJdLFswLDEsIlxcaGF0IHAiXSxbMiwwLCJcXFRcXGhhdCBwIl0sWzMsMiwibF9BIl0sWzMsMSwiXFxoYXQgcCIsMl1d
\begin{tikzcd}
{\T^2A} & {\T A} \\
{\T A} & A
\arrow["{\T{\hat p}_A}", from=1-1, to=1-2]
\arrow["{{\hat p}_A}", from=1-2, to=2-2]
\arrow["{l_A}", from=2-1, to=1-1]
\arrow["{{\hat p}_A}"', from=2-1, to=2-2]
\end{tikzcd}
\end{equation*}

\item Universality. The following is a product diagram:
\begin{equation*}
% https://q.uiver.app/#q=WzAsMyxbMSwwLCJcXFQgQSJdLFsyLDAsIkEiXSxbMCwwLCJBIl0sWzAsMSwiXFxoYXQgcCJdLFswLDIsInBfQSIsMl1d
\begin{tikzcd}
A & {\T A} & A
\arrow["{p_A}"', from=1-2, to=1-1]
\arrow["{{\hat p}_A}", from=1-2, to=1-3]
\end{tikzcd}
\end{equation*}
\end{enumerate}
\end{definition}

The commutative monoid structure of a differential object captures the additive structure of the vector space underlying a Euclidean space, where $\zeta_A\colon\*\to A$ sets the zero vector and $\sigma_A\colon A\times A\to A$ represents the sum of vectors. The differential projection ${\hat p}_A\colon\T A\to A$ corresponds to the canonical map which identifies the tangent space $\T_xA$ at each point $x$ of a Euclidean space $A$ to the space $A$ itself, by sending a tangent vector $v$ of $\T_xA$ to the point $x+v$, where $x$ is intepreted both as a point of $A$ and as a vector.
\par In a Cartesian tangent category $(\X,\TT)$ there is a Cartesian tangent category, denoted by $\DO(\X,\TT)$, whose objects are differential objcts $(A,\zeta_A,\sigma_A,{\hat p}_A)$ and morphisms are linear morphisms of differential objects, that are, morphisms $f\colon A\to B$ that commute with the differential projection, $\hat p_A\o f=\T f\o\hat p_B$. The tangent bundle functor of $\DO(\X,\TT)$ sends a differential object $(A,\zeta_,\sigma_A,{\hat p}_A)$ to the differential object $(\T A,\T\zeta_A,\T\sigma_A,{\hat p}_{\T A})$, where we identified $\T(A\times A)$ with $\T A\times\T A$ and $\T\*$ with $\*$, and whose differential projection is defined as follows:
\begin{align*}
&{\hat p}_{\T A}\colon\T^2A\xrightarrow{c}\T^2A\xrightarrow{\T{\hat p}_A}\T A
\end{align*}
Finally, the structural natural transformations of $\DO(\X,\TT)$ are defined as in $(\X,\TT)$.

\begin{lemma}[{\cite[Section~4.1]{cockett:tangent-cats}}]
\label{lemma:DO-functoriality}
There is a $2$-endofunctor
\begin{align*}
&\DO\colon\cTngCat_\cong\to\cTngCat_\cong
\end{align*}
on the category of Cartesian tangent categories and strong tangent morphisms which preserve finite products, that sends a Cartesian tangent category $(\X,\TT)$ to the Cartesian tangent category $\DO(\X,\TT)$ of differential objects of $(\X,\TT)$ and that sends a strong tangent morphism $(F,\alpha)\colon(\X',\TT')\to(\X,\TT)$ which preserves finite products to the strong tangent morphism which sends a differential object $(A,\zeta_A,\sigma_A,{\hat p}_A)$ to the differential object $(FA,F\zeta_A,F\sigma_A,{\hat p}_{FA})$, whose differential projection is defined as follows:
\begin{align*}
&{\hat p}_{FA}\colon\T FA\xrightarrow{\alpha^{-1}}F\T'A\xrightarrow{F{\hat p}_A}FA
\end{align*}
\end{lemma}

%__________________________________________________________________________
\subsection{The universal property of differential objects}
\label{subsection:universal-property-differential-objects}
This section aims to characterize the correct universal property of the tangent category $\DO(\X,\TT)$ of differential objects and linear morphisms of a Cartesian tangent category $(\X,\TT)$. For starters, consider the forgetful functor $\U\colon\DO(\X,\TT)\to(\X,\TT)$ which sends a differential object $A$ to its underlying object of $(\X,\TT)$.
\par Since $(\X,\TT)$ is Cartesian, there are two (strong) tangent morphisms
\begin{align*}
\*&\colon\1\to(\X,\TT)              &\times&\colon(\X,\TT)\times(\X,\TT)\to(\X,\TT)
\end{align*}
sending the unique object of the terminal tangent category $\1$ to the (chosen) terminal object $\*$ of $(\X,\TT)$ and sending a pair of objects $(A,B)$ to their Cartesian product $A\times B$, respectively. Therefore, we can define the following natural transformations
\begin{align*}
&\UnivZ_{(A,\zeta_A,\sigma_A,\hat p_A)}\colon\*(!(A,\zeta_A,\sigma_A,\hat p_A))=\*\xrightarrow{\zeta_A}A=\U(A,\zeta_A,\sigma_A,\hat p_A)\\
&\UnivS_{(A,\zeta_A,\sigma_A,\hat p_A)}\colon\times(\<\U,\U\>(A,\zeta_A,\sigma_A,\hat p_A))=A\times A\xrightarrow{\sigma_A}A=\U(A,\zeta_A,\sigma_A,\hat p_A)
\end{align*}
where $!\colon(\X,\TT)\to\*$ is the terminal morphism, that is, the constant functor. However, as shown by Lemma~\ref{lemma:cartesian-hom-cats} and by Proposition~\ref{proposition:cartesian-hom-tangent-cats}, the functors $\*\:\o\:!$ and $\times\o\<\U,\U\>$ are precisely the terminal object $\top$ and the Cartesian product $\U\Box\U$ in the Cartesian Hom-tangent category $[\DO(\X,\TT)\|\X,\TT]$.

\begin{lemma}
\label{lemma:universal-DO-cartesian-structure}
The forgetful functor $\U\colon\DO(\X,\TT)\to(\X,\TT)$ equipped with the natural transformations $\UnivZ$ and $\UnivS$, is a commutative monoid in the Cartesian Hom-tangent category $[\DO(\X,\TT)\|\X,\TT]$.
\end{lemma}
\begin{proof}
First, we need to prove that $\UnivZ$ and $\UnivS$ are tangent natural transformations. Thanks to~\cite[Propositions~3.7 and~2.16]{cockett:differential-bundles}, linear morphisms of differential objects are additive, that is they preserve the commutative monoid structures. In particular, this implies that $\UnivZ$ and $\UnivS$ are natural transformations. To prove they are compatible with the tangent structure, notice that, by definition of the tangent structure on $\DO(\X,\TT)$, we have
\begin{align*}
&\UnivZ_{\T^\DO(A,\zeta_A,\sigma_A,\hat p_A)}=\T\zeta_A=\bar\T(\UnivZ_{(A,\zeta_A,\sigma_A,\hat p_A)})\\
&\UnivS_{\T^\DO(A,\zeta_A,\sigma_A,\hat p_A)}=\T\sigma_A=\bar\T(\UnivS_{(A,\zeta_A,\sigma_A,\hat p_A)})
\end{align*}
Thus, $\UnivZ$ and $\UnivS$ are compatible with the tangent structure. To prove that $(\U,\UnivZ,\UnivS)$ defines a commutative monoid, one simply uses the fact that each $(A,\zeta_A,\sigma_A)$ constitutes a commutative monoid.
\end{proof}

The next step is to show that the commutative monoid $(\U,\UnivZ,\UnivS)$ comes with a differential projection. Define:
\begin{align*}
&\Univ{\hat p}_{(A,\zeta_A,\sigma_A,\hat p_A)}\colon\bar\T\U(A,\zeta_A,\sigma_A,\hat p_A)=\T A\xrightarrow{\hat p_A}A=\U(A,\zeta_A,\sigma_A,\hat p_A)
\end{align*}
By the linearity of the morphisms in $\DO(\X,\TT)$, $\Univ{\hat p}$ is natural. Furthermore, since
\begin{align*}
&\Univ{\hat p}_{\T^\DO(A,\zeta_A,\sigma_A,\hat p_A)}=\T\hat p_A\o c
\end{align*}
and $c^2=\id_{\T^2}$, the following diagram commutes:
\begin{equation*}
% https://q.uiver.app/#q=WzAsNSxbMCwxLCJcXFVcXG9cXGJhclxcVCJdLFsyLDAsIlxcYmFyXFxUXjJcXG9cXFUiXSxbMSwwLCJcXGJhclxcVF4yXFxvXFxVIl0sWzAsMCwiXFxiYXJcXFRcXG9cXFVcXG9cXGJhclxcVCJdLFsyLDEsIlxcVFxcb1xcVSJdLFszLDAsIlxcaGF0IHBfe1xcYmFyXFxUfSIsMl0sWzAsNCwiIiwyLHsibGV2ZWwiOjIsInN0eWxlIjp7ImhlYWQiOnsibmFtZSI6Im5vbmUifX19XSxbMSw0LCJcXGJhclxcVFxcaGF0IHAiXSxbMywyLCIiLDAseyJsZXZlbCI6Miwic3R5bGUiOnsiaGVhZCI6eyJuYW1lIjoibm9uZSJ9fX1dLFsyLDEsImNfXFxVIl1d
\begin{tikzcd}
{\bar\T\o\U\o\bar\T} & {\bar\T^2\o\U} & {\bar\T^2\o\U} \\
{\U\o\bar\T} && {\T\o\U}
\arrow[equals, from=1-1, to=1-2]
\arrow["{\Univ{\hat p}_{\bar\T}}"', from=1-1, to=2-1]
\arrow["{c_\U}", from=1-2, to=1-3]
\arrow["{\bar\T\Univ{\hat p}}", from=1-3, to=2-3]
\arrow[equals, from=2-1, to=2-3]
\end{tikzcd}
\end{equation*}
Thus, $\Univ{\hat p}$ is a tangent natural transformation.

\begin{definition}
\label{definition:pointwise-differential-object}
A \textbf{pointwise differential object} consists of a differential object $(G,\beta;\zeta,\sigma,\hat p)$ in a Hom-tangent category $[\X',\TT'\|\X,\TT]$ such that, the diagram:
\begin{equation*}
% https://q.uiver.app/#q=WzAsMyxbMSwwLCJcXGJhclxcVChHLFxcYmV0YSkiXSxbMCwwLCIoRyxcXGJldGEpIl0sWzIsMCwiKEcsXFxiZXRhKSJdLFswLDEsIlxcYmFyIHAiLDJdLFswLDIsIlxcaGF0IHAiXV0=
\begin{tikzcd}
{(G,\beta)} & {\bar\T(G,\beta)} & {(G,\beta)}
\arrow["{\bar p}"', from=1-2, to=1-1]
\arrow["{\hat p}", from=1-2, to=1-3]
\end{tikzcd}
\end{equation*}
is a pointwise product diagram in $[\X',\TT'\|\X,\TT]$.
\end{definition}

\begin{proposition}
\label{proposition:universality-differential-object}
The forgetful functor $\U\colon\DO(\X,\TT)\to(\X,\TT)$ equipped with the tangent natural transformations $\UnivZ$, $\UnivS$, and $\Univ{\hat p}$ defines a pointwise differential object in the Cartesian Hom-tangent category $[\DO(\X,\TT)\|\X,\TT]$. 
\end{proposition}
\begin{proof}
In Lemma~\ref{lemma:universal-DO-cartesian-structure} we proved that $(\U,\UnivZ,\UnivS)$ is a commutative monoid in $[\DO(\X,\TT)\|\X,\TT]$. The equational axioms required for $\Univ{\hat p}$ to be a differential projection of $(\U,\UnivZ,\UnivS)$, are a consequence of $\Univ{\hat p}_{(A,\zeta_A,\sigma_A,\hat p_A)}$ being the differential projection of each differential object $(A,\zeta_A,\sigma_A,\hat p_A)$. Notice that the compatibility with the vertical lift is also a consequence of the compatibility between $l_A\colon\T A\to\T^2A$ and $\hat p_A$. Finally, to prove the universal property of $\Univ{\hat p}$, notice that limits on the target tangent category induce pointwise limits in the Hom-tangent category.
\end{proof}

To characterize the correct universal property enjoyed by the pointwise differential object of Proposition~\ref{proposition:universality-differential-object}, for starters, let us unwrap the definition of a differential object in each Hom-tangent category $[\X',\TT'\|\X'',\TT'']$, where $(\X',\TT')$ and $(\X'',\TT'')$ are Cartesian tangent categories:
\begin{description}
\item[Base object] The base object is a lax tangent morphism $(G,\beta)\colon(\X',\TT')\to(\X'',\TT'')$;

\item[Zero morphism] The zero morphism consists of a tangent natural transformation:
\begin{align*}
&\zeta\colon\top\to(G,\beta)
\end{align*}

\item[Sum morphism] The sum morphism consists of a tangent natural transformation:
\begin{align*}
&\sigma\colon(G,\beta)\:\Box\:(G,\beta)\to(G,\beta)
\end{align*}

\item[Differential projection] The differential projection consists of a tangent natural transformation:
\begin{align*}
&\hat p\colon\bar\T(G,\beta)\to(G,\beta)
\end{align*}
\end{description}
satisfying the axioms of a pointwise differential object. A morphism of differential objects in the Hom-tangent category $[\X,\TT\|\X',\TT']$ from a pointwise differential object $(G,\beta;\zeta,\sigma,\hat p)$ to another one $(G',\beta';\zeta',\sigma',\hat p')$ consists of a tangent natural transformation $\varphi\colon(G,\beta)\to(G',\beta')$ which commutes with the differential structures.
\par For each tangent category $(\X',\TT')$, a pointwise differential object $(G,\beta;\zeta,\sigma,\hat p)$ in the Hom-tangent category $[\X'',\TT''\|\X''',\TT''']$ induces a functor
\begin{align*}
&\Gamma_{(G,\beta;\zeta,\sigma,\hat p)}\colon[\X',\TT'\|\X'',\TT'']\to\DO[\X',\TT'\|\X''',\TT''']
\end{align*}
which sends a lax tangent morphism $(H,\gamma)\colon(\X',\TT')\to(\X'',\TT'')$ to the pointwise differential object
\begin{align*}
&\Gamma_{(G,\beta;\zeta,\sigma,\hat p)}(H,\gamma)\=((G,\beta)\o(H,\gamma);\zeta_{(H,\gamma)},\sigma_{(H,\gamma)},\hat p_{(H,\gamma)})
\end{align*}
in the Hom-tangent category $[\X',\TT'\|\X''',\TT''']$. Concretely, the base object of $\Gamma_{(G,\beta;\zeta,\sigma,\hat p)}(H,\gamma)$ is the lax tangent morphism
\begin{align*}
&(\X',\TT')\xrightarrow{(H,\gamma)}(\X'',\TT'')\xrightarrow{(G,\beta)}(\X''',\TT''')
\end{align*}
while the zero, the sum, and the differential projection of $\Gamma_{(G,\beta;\zeta,\sigma,\hat p)}(H,\gamma)$ are the natural transformations:
\begin{equation*}
% https://q.uiver.app/#q=WzAsMyxbMCwwLCIoXFxYJyxcXFRUJykiXSxbMSwwLCIoXFxYJycsXFxUVCcnKSJdLFszLDAsIihcXFgnJycsXFxUVCcnJykiXSxbMCwxLCIoSCxcXGdhbW1hKSJdLFsxLDIsIlxcdG9wIiwyLHsiY3VydmUiOjN9XSxbMSwyLCIoRyxcXGJldGEpIiwwLHsiY3VydmUiOi0zfV0sWzQsNSwiXFx6ZXRhJyIsMix7InNob3J0ZW4iOnsic291cmNlIjoyMCwidGFyZ2V0IjoyMH19XV0=
\begin{tikzcd}
{(\X',\TT')} & {(\X'',\TT'')} && {(\X''',\TT''')}
\arrow["{(H,\gamma)}", from=1-1, to=1-2]
\arrow[""{name=0, anchor=center, inner sep=0}, "\top"', curve={height=18pt}, from=1-2, to=1-4]
\arrow[""{name=1, anchor=center, inner sep=0}, "{(G,\beta)}", curve={height=-18pt}, from=1-2, to=1-4]
\arrow["{\zeta}"', shorten <=5pt, shorten >=5pt, Rightarrow, from=0, to=1]
\end{tikzcd}
\end{equation*}
\begin{equation*}
% https://q.uiver.app/#q=WzAsMyxbMCwwLCIoXFxYJyxcXFRUJykiXSxbMSwwLCIoXFxYJycsXFxUVCcnKSJdLFszLDAsIihcXFgnJycsXFxUVCcnJykiXSxbMCwxLCIoSCxcXGdhbW1hKSJdLFsxLDIsIihHLFxcYmV0YSkiLDIseyJjdXJ2ZSI6M31dLFsxLDIsIihHLFxcYmV0YSlcXEJveChHLFxcYmV0YSkiLDAseyJjdXJ2ZSI6LTN9XSxbNSw0LCJcXHNpZ21hJyIsMCx7InNob3J0ZW4iOnsic291cmNlIjoyMCwidGFyZ2V0IjoyMH19XV0=
\begin{tikzcd}
{(\X',\TT')} & {(\X'',\TT'')} && {(\X''',\TT''')}
\arrow["{(H,\gamma)}", from=1-1, to=1-2]
\arrow[""{name=0, anchor=center, inner sep=0}, "{(G,\beta)}"', curve={height=18pt}, from=1-2, to=1-4]
\arrow[""{name=1, anchor=center, inner sep=0}, "{(G,\beta)\Box(G,\beta)}", curve={height=-18pt}, from=1-2, to=1-4]
\arrow["{\sigma}", shorten <=5pt, shorten >=5pt, Rightarrow, from=1, to=0]
\end{tikzcd}
\end{equation*}
\begin{equation*}
% https://q.uiver.app/#q=WzAsMyxbMCwwLCIoXFxYJyxcXFRUJykiXSxbMSwwLCIoXFxYJycsXFxUVCcnKSJdLFszLDAsIihcXFgnJycsXFxUVCcnJykiXSxbMCwxLCIoSCxcXGdhbW1hKSJdLFsxLDIsIihHLFxcYmV0YSkiLDIseyJjdXJ2ZSI6M31dLFsxLDIsIlxcYmFyXFxUKEcsXFxiZXRhKSIsMCx7ImN1cnZlIjotM31dLFs1LDQsIlxcaGF0IHAnIiwwLHsic2hvcnRlbiI6eyJzb3VyY2UiOjIwLCJ0YXJnZXQiOjIwfX1dXQ==
\begin{tikzcd}
{(\X',\TT')} & {(\X'',\TT'')} && {(\X''',\TT''')}
\arrow["{(H,\gamma)}", from=1-1, to=1-2]
\arrow[""{name=0, anchor=center, inner sep=0}, "{(G,\beta)}"', curve={height=18pt}, from=1-2, to=1-4]
\arrow[""{name=1, anchor=center, inner sep=0}, "{\bar\T(G,\beta)}", curve={height=-18pt}, from=1-2, to=1-4]
\arrow["{\hat p}", shorten <=5pt, shorten >=5pt, Rightarrow, from=1, to=0]
\end{tikzcd}
\end{equation*}

As observed for vector fields~\cite[Section~3]{lanfranchi:tangentads-II}, this construction is functorial and extends to a tangent morphism.

\begin{lemma}
\label{lemma:induced-lax-tangent-morphism-from-differential-objects}
A differential object $(G,\beta;\zeta,\sigma,\hat p)$ in the Hom-tangent category $[\X'',\TT''\|\X''',\TT''']$ induces a lax tangent morphism
\begin{align*}
&\Gamma_{(G,\beta;\zeta,\sigma,\hat p)}\colon[\X',\TT'\|\X'',\TT'']\to\DO[\X',\TT'\|\X''',\TT''']
\end{align*}
for each tangentad $(\X',\TT')$. Furthermore, $\Gamma_{(G,\beta;\zeta,\sigma,\hat p)}$ is natural in $(\X',\TT')$ and it is strong (strict) when $(G,\beta)$ is strong (strict).
\end{lemma}
\begin{proof}
To define the induced lax tangent morphism $\Gamma_{(G,\beta;\zeta,\sigma,\hat p)}$, one can readily extend the argument used in the discussion above for tangent categories by replacing the natural transformations with $2$-morphisms of $\CC$. We leave the details to the reader. Furthermore, the distributive law $\beta$ induces a distributive law on $\Gamma_{(G,\beta;\zeta,\sigma,\hat p)}$ such that when $(G,\beta)$ is strong, respectively strict, so is $\Gamma_{(G,\beta;\zeta,\sigma,\hat p)}$. To prove that $(G,\beta;\zeta,\sigma,\hat p)$ is natural in $(\X',\TT')$, consider two lax tangent morphisms $(F,\alpha)\colon(\X,\TT)\to(\X',\TT')$ and $(H,\gamma)\colon(\X',\TT')\to(\X'',\TT'')$. Thus:
\begin{align*}
&(\Gamma_{(G,\beta;\zeta,\sigma,\hat p)}\o[F,\alpha\|\X',\TT'])(H,\gamma))=\Gamma_{(G,\beta;\zeta,\sigma,\hat p)}((H,\gamma)\o(F,\alpha))\\
&\quad=\Gamma_{(G,\beta;\zeta,\sigma,\hat p)}((H,\gamma))\o(F,\alpha)=\DO[F,\alpha\|\X''',\TT''']\o\Gamma_{(G,\beta;\zeta,\sigma,\hat p)}(H,\gamma)
\end{align*}
Finally, using that each $[F,\alpha\|\X'',\TT'']$ is a strict tangent morphism, $\Gamma_{(G,\beta;\zeta,\sigma,\hat p)}$ becomes a tangent natural transformation.
\end{proof}

In particular, by Lemma~\ref{lemma:induced-lax-tangent-morphism-from-differential-objects}, the differential object $(\U;\UnivZ,\UnivS,\Univ{\hat p})$ of Proposition~\ref{proposition:universality-differential-object} induces a strict tangent natural transformation
\begin{align*}
&\Gamma_{(\UnivZ,\UnivS,\Univ{\hat p})}\colon[\X',\TT'\|\DO(\X,\TT)]\to\DO[\X',\TT'\|\X,\TT]
\end{align*}
natural in $(\X',\TT')$.
\par We want to prove that $\Gamma_{(\UnivZ,\UnivS,\Univ{\hat p})}$ is invertible. Consider another tangent category $(\X',\TT')$ together with a differential object $(G,\beta;\zeta,\sigma,\hat p)$ in the Hom-tangent category $[\X',\TT'\|\X,\TT]$. For every $A\in\X'$, the tuple
\begin{align*}
&\left(GA,\zeta_A'\colon\*\to GA,\sigma_A'\colon GA\times GA\to GA,\hat p_A'\colon GA\to\T GA\right)
\end{align*}
is a differential object in $(\X,\TT)$. Therefore, we can define a functor
\begin{align*}
\Lambda[G,\zeta,\sigma,\hat p]\colon(\X',\TT')\to\DO(\X,\TT)
\end{align*}
which sends an object $A$ of $\X'$ to the differential object $(GA,\zeta_A,\sigma_A,\hat p_A)$, and a morphism $f\colon A\to B$ to $Gf$. By the naturality of $\zeta$, $\sigma$, and $\hat p$,
\begin{align*}
&Gf\colon(GA,\zeta_A,\sigma_A,\hat p_A)\to(GB,\zeta_B,\sigma_B,\hat p_B)
\end{align*}
becomes a linear morphism of differential objects of $(\X,\TT)$. Furthermore, $\Lambda[G,\zeta,\sigma,\hat p]$ comes with a distributive law
\begin{align*}
&\Lambda[\beta]\colon\Lambda[G,\zeta,\sigma,\hat p](\T'A)=(G\T'A,\zeta_{\T'A},\sigma_{\T'A},\hat p_{\T'A})\xrightarrow{\beta_A}(\T GA,\T\zeta_A',\T\sigma_A,c_A\o\T\hat p_A)=\T^\DO(\Lambda[G,\zeta,\sigma,\hat p](A))
\end{align*}
which is a morphism of differential objects since each of the structural morphisms $\zeta$, $\sigma$, and $\hat p$ are tangent natural transformations and thus compatible with the distributive law $\beta$.

\begin{lemma}
\label{lemma:universality-differential-objects}
Given a lax tangent morphism $(G,\beta)\colon(\X',\TT')\to(\X,\TT)$ and a differential structure $(\zeta,\sigma,\hat p)$ for $(G,\beta)$ in the Hom-tangent category $[\X',\TT'\|\X,\TT]$, the functor $\Lambda[G,\zeta,\sigma,\hat p]$ together with the distributive law $\Lambda[\beta]$ defines a lax tangent morphism:
\begin{align*}
&\Lambda[G,\beta;\zeta,\sigma,\hat p]\colon(\Lambda[G,\zeta,\sigma,\hat p],\Lambda[\beta])\colon(\X',\TT')\to\DO(\X,\TT)
\end{align*}
\end{lemma}
\begin{proof}
In the previous discussion, we defined $\Lambda[G,\beta;\zeta,\sigma,\hat p]$ as the lax tangent morphism which picks out the differential object $(GA;\zeta_A,\sigma_A,\hat p_A)$ for each $A\in\X'$ and whose distributive law is induced by $\beta$. In particular, the naturality of the structure morphisms $\zeta$, $\sigma$, and $\hat p$ makes the functor $\Lambda[G,\zeta,\sigma,\hat p]$ well-defined, the compatibility of the structure morphisms and $\beta$ makes $\Lambda[\beta]$ into a morphism of differential objects, and the compatibility of $\beta$ with the tangent structures makes $\Lambda[G,\zeta,\sigma,\hat p]$ into a lax tangent morphism.
\end{proof}

We can now prove the universal property of differential objects, which is the main result of this section.

\begin{theorem}
\label{theorem:universality-differential-objects}
The differential object $(\U,\UnivZ,\UnivS,\Univ{\hat p})$ of Proposition~\ref{proposition:universality-differential-object} is universal. Concretely, the induced strict tangent natural transformation
\begin{align*}
&\Gamma_{(\UnivZ,\UnivS,\Univ{\hat p})}\colon[\X',\TT'\|\DO(\X,\TT)]\to\DO[\X',\TT'\|\X,\TT]
\end{align*}
makes the functor
\begin{align*}
&\TngCat^\op\xrightarrow{[-\|\X,\TT]}\TngCat^\op\xrightarrow{\DO^\op}\TngCat^\op
\end{align*}
which sends a tangent category $(\X',\TT')$ to the tangent category $\DO[\X',\TT'\|\X,\TT]$, into a corepresentable functor. In particular, $\Gamma_{(\UnivZ,\UnivS,\Univ{\hat p})}$ is invertible.
\end{theorem}
\begin{proof}
Consider a tangent category $(\X',\TT')$ and a lax tangent morphism $(F,\alpha)\colon(\X',\TT')\to\DO(\X,\TT)$, which sends each $A\in\X'$ to the differential object $(FA,\zeta_{A},\sigma_{A},\hat p_{A})$. We want to compare $(F,\alpha)$ with $\Lambda[\Gamma_{(\UnivZ,\UnivS,\Univ{\hat p})}(F,\alpha)]$. Concretely, $\Lambda[\Gamma_{(\UnivZ,\UnivS,\Univ{\hat p})}(F,\alpha)]$ sends an object $A$ of $\X'$ to the differential objects $\Gamma_{(\UnivZ,\UnivS,\Univ{\hat p})}(F,\alpha)_A$. However, this corresponds to the differential object $(FA,\zeta_{A},\sigma_{A},\hat p_{A})$. Given a morphism $\varphi\colon(F,\alpha)\to(F',\alpha')$ of $[\X',\TT'\|\DO(\X,\TT)]$
\begin{align*}
&\varphi_A\colon(FA,\zeta_A,\sigma_A,\hat p_A)\to(F'A,\zeta'_A,\sigma'_A,\hat p'_A)
\end{align*}
which coincides with the morphism $\Lambda[\Gamma_{(\UnivZ,\UnivS,\Univ{\hat p})}(\varphi)]$, since:
\begin{align*}
&\Gamma_{(\UnivZ,\UnivS,\Univ{\hat p})}(\varphi)_A=FA\xrightarrow{\varphi}F'A
\end{align*}
Conversely, consider a differential object $(G,\beta;\zeta,\sigma,\hat p)$ of the Hom-tangent category $[\X',\TT'\|\X,\TT]$. Therefore, $\Lambda[G,\beta;\zeta,\sigma,\hat p]$ sends each $A$ to the differential object $(GA,\zeta_A,\sigma_A,\hat p_A)$. Thus, $\Gamma_{(\UnivZ,\UnivS,\Univ{\hat p})}(\Lambda[G,\beta;\zeta,\sigma,\hat p])$ coincides with $(G,\beta;\zeta,\sigma,\hat p)$. Similarly, for a given morphism $\varphi(G,\beta;\zeta,\sigma,\hat p)\to(G',\beta';\zeta',\sigma',\hat p')$ of differential objects, $\Gamma_{(\UnivZ,\UnivS,\Univ{\hat p})}(\Lambda[\varphi])$ coincides with $\varphi$.
\par Finally, it is not hard to see that the functor
\begin{align*}
&\Lambda\colon\DO[\X',\TT'\|\X,\TT]\to[\X',\TT'\|\DO(\X,\TT)]
\end{align*}
extends to a strict tangent morphism, since both $\Lambda[\T^\DO(G,\beta;\zeta,\sigma,\hat p)]$ and $\bar\T(\Lambda[G,\beta;\zeta,\sigma,\hat p])$ send each $A\in\X'$ to the differential object
\begin{align*}
&(\T GA,\T\zeta_A,\T\sigma_A,c_A\o\T\hat p_A)
\end{align*}
Therefore, $\Lambda$ inverts $\Gamma_{(\UnivZ,\UnivS,\Univ{\hat p})}$.
\end{proof}

Theorem~\ref{theorem:universality-differential-objects} establishes the correct universal property of the construction of differential objects. Thanks to this result, we can finally introduce the notion of differential objects in the formal context of tangentads.

\begin{definition}
\label{definition:construction-differential-objects}
A Cartesian tangentad $(\X,\TT)$ in a $2$-category $\CC$ \textbf{admits the construction of differential objects} if there exists a Cartesian tangentad $\DO(\X,\TT)$ of $\CC$ together with a (strict) tangent morphism $\U\colon\DO(\X,\TT)\to(\X,\TT)$ and a pointwise differential object $(\U;\UnivZ,\UnivS,\Univ{\hat p})$ in the Hom-tangent category $[\DO(\X,\TT)\|\X,\TT]$ such that the induced tangent natural transformation $\Gamma_{(\UnivZ,\UnivS,\Univ{\hat p})}$ of Lemma~\ref{lemma:induced-lax-tangent-morphism-from-differential-objects} is invertible. The pointwise differential object $(\U;\UnivZ,\UnivS,\Univ{\hat p})$ is called the \textbf{universal pointwise differential object} of $(\X,\TT)$ and $\DO(\X,\TT)$ is called the \textbf{tangentad of differential objects} of $(\X,\TT)$.
\end{definition}

\begin{definition}
\label{definition:construction-differential-objects-2-category}
A Cartesian $2$-category $\CC$ \textbf{admits the construction of differential objects} provided that every Cartesian tangentad of $\CC$ admits the construction of differential objects.
\end{definition}

We can now rephrase Theorem~\ref{theorem:universality-differential-objects} as follows.

\begin{corollary}
\label{corollary:universality-differential-objects}
The $2$-category $\Cat$ of categories admits the construction of differential objects and the tangentad of differential objects of a tangentad $(\X,\TT)$ of $\Cat$ is the tangent category $\DO(\X,\TT)$ of differential objects of $(\X,\TT)$.
\end{corollary}

To ensure that Definitions~\ref{definition:construction-differential-objects} and~\ref{definition:construction-differential-objects-2-category} are well-posed, one requires the tangentad of differential objects of a given tangentad to be unique. The next proposition establishes that such a construction is defined uniquely up to a unique isomorphism.

\begin{proposition}
\label{proposition:uniqueness-differential-objects}
If a tangentad $(\X,\TT)$ admits the construction of differential objects, the tangentad of differential objects $\DO(\X,\TT)$ of $(\X,\TT)$ is unique up to a unique isomorphism which extends to an isomorphism of the corresponding universal differential objects of $(\X,\TT)$.
\end{proposition}
\begin{proof}
This follows directly from the universal property of the construction. In particular, suppose that $\DO'(\X,\TT)$ is another tangentad of differential objects of $(\X,\TT)$. Using the universal property, the identity tangent morphism of $\DO'(\X,\TT)$ must correspond to a morphism of pointwise differential objects in the Hom-tangent category $[\DO'(\X,\TT)\|\X,\TT]$, which, by the universal property of $\DO(\X,\TT)$, corresponds to a tangent morphism $\DO'(\X,\TT)\to\DO(\X,\TT)$. Dually, we can construct a tangent morphism $\DO(\X,\TT)\to\DO'(\X,\TT)$ and, by the universal property, these two tangent morphisms must invert each other.
\end{proof}

%__________________________________________________________________________
\subsection{The formal structures of differential objects}
\label{subsection:structures-differential-objects}
In Section~\ref{subsection:tangent-category-differential-objects}, we recalled that (a) differential objects of a tangent category $(\X,\TT)$ form a new tangent category $\DO(\X,\TT)$; (b) the assigment which sends a tangent category $(\X,\TT)$ to $\DO(\X,\TT)$ is $2$-functorial.
\par In this section, we recover the $2$-functoriality of the construction of differential objects in the context of tangentads. For starters, notice that, in the context of tangentads, (a) is a built-in property of the definition of the tangentad of differential objects $\DO(\X,\TT)$ of a tangentad $(\X,\TT)$, since $\DO(\X,\TT)$ is, by definition, a tangentad.
\par In the following, $\CC$ denotes a $2$-category which admits the construction of differential objects.

%__________________________________________________________________________
\subsubsection*{The functoriality of the construction of differential objects}
\label{subsubsection:functoriality-differential-objects}
As proved in Lemma~\ref{lemma:DO-functoriality}, there is a $2$-functor
\begin{align*}
&\DO\colon\cTngCat_\cong\to\cTngCat_\cong
\end{align*}
which sends a Cartesian tangent category $(\X,\TT)$ to the Cartesian tangent category $\DO(\x,\TT)$ of differential objects of $(\X,\TT)$.
\par In this section, we show that the construction of differential objects is $2$-functorial in the $2$-category of tangentads of $\CC$. Consider a strong tangent morphism $(F,\alpha)\colon(\X,\TT)\to(\X',\TT')$ between two Cartesian tangentads of $\CC$ which preserves finite products. It is not hard to prove that
\begin{align*}
&[\DO(\X,\TT)\|\X,\TT]\xrightarrow{[\DO(\X,\TT)\|F,\alpha]}[\DO(\X,\TT)\|\X',\TT']
\end{align*}
is a strong tangent morphism which preserves finite products. Thus, $\DO$ sends $[\DO(\X,\TT)\|F,\alpha]$ to a strong tangent morphism
\begin{align*}
&\DO[\DO(\X,\TT)\|\X,\TT]\xrightarrow{\DO[\DO(\X,\TT)\|F,\alpha]}\DO[\DO(\X,\TT)\|\X',\TT']\cong[\DO(\X,\TT)\|\DO(\X',\TT')]
\end{align*}
where we used the universal property of differential objects to identify $\DO[\DO(\X,\TT)\|\X',\TT']$ with the Hom-tangent category $[\DO(\X,\TT)\|\DO(\X',\TT')]$. In particular, the universal differential object $\UnivZ,\UnivS,\Univ{\hat p}$ is sent by this tangent morphism to a tangent morphism
\begin{align*}
&\DO(F,\alpha)\colon\DO(\X,\TT)\to\DO(\X',\TT')
\end{align*}
Now, consider a $2$-morphism $\varphi\colon(F,\alpha)\Rightarrow(G,\beta)\colon(\X,\TT)\to(\X',\TT')$ and define the tangent natural transformation:
\begin{align*}
&\DO[\DO(\X,\TT)\|\varphi]\colon\DO[\DO(\X,\TT)\|F,\alpha]\Rightarrow\DO[\DO(\X,\TT)\|G,\beta]
\end{align*}
Thus, by evaluating such a transformation on the universal differential object and by employing the universal property of differential objects, we obtain a $2$-morphism:
\begin{align*}
&\DO(\varphi)\colon\DO(F,\alpha)\Rightarrow\DO(G,\beta)
\end{align*}

\begin{proposition}
\label{proposition:DO-functoriality}
There is a $2$-functor
\begin{align*}
&\DO\colon\cTng(\CC)_\cong\to\cTng(\CC)_\cong
\end{align*}
which sends a Cartesian tangentad $(\X,\TT)$ of $\CC$ to the tangentad of differential objects $\DO(\X,\TT)$ of $(\X,\TT)$. In particular, for each strong tangent morphism $(F,\alpha)\colon(\X,\TT)\to(\X',\TT')$, the tangent morphism $\DO(F,\alpha)\colon\DO(\X,\TT)\to\DO(\X',\TT')$ is also strong.
\end{proposition}
\begin{proof}
Consider a strong tangent morphism $(F,\alpha)\colon(\X,\TT)\to(\X',\TT')$ between two Cartesian tangentads which preserves finite products. As previously discussed, let $\DO(F,\alpha)$ the tangent morphism which correspond to the universal differential object $\UnivZ,\UnivS,\Univ{\hat p}$ via the functor
\begin{align*}
&\DO[\DO(\X,\TT)\|\X,\TT]\xrightarrow{\DO[\DO(\X,\TT)\|F,\alpha]}\DO[\DO(\X,\TT)\|\X',\TT']\cong[\DO(\X,\TT)\|\DO(\X',\TT')]
\end{align*}
The underlying tangent morphism of $\DO[\DO(\X,\TT)\|F,\alpha](\UnivZ,\UnivS,\Univ{\hat p})$ is simply given by
\begin{align*}
&\DO(\X,\TT)\xrightarrow{\U}(\X,\TT)\xrightarrow{(F,\alpha)}(\X,\TT)
\end{align*}
Given a $2$-morphism $\varphi\colon(F,\alpha)\Rightarrow(G,\beta)$, $\DO(\varphi)$ is the $2$-morphism corresponding to the tangent natural transformation
\begin{align*}
&\DO[\DO(\X,\TT)\|\varphi]\colon\DO[\DO(\X,\TT)\|F,\alpha]\Rightarrow\DO[\DO(\X,\TT)\|G,\beta]
\end{align*}
evaluated at the universal differential object.
\end{proof}

%__________________________________________________________________________
\subsection{Applications}
\label{subsection:examples-differential-objects}
In Section~\ref{subsection:definition-tangentad}, we listed some examples of tangentads. In this section, we compute the construction of differential objects for each of these examples. Unfortunately, the construction of differential objects cannot be obtained by means of PIE limits as done for vector fields in~\cite[Theorem~4.8]{lanfranchi:tangentads-II}. The reason is that the universal property enjoyed by the differential projection of a differential object is a non-equational axiom, and PIE limits can only construct algebraic theories. Therefore, in this section, we follow a direct approach to construct differential objects for the different tangentads.

%__________________________________________________________________________
\subsubsection*{Tangent monads}
\label{subsubsection:differential-objects-tangent-monads}
In order to compute the construction of differential objects for tangent monads we need to first understand the notion of Cartesian tangent monads, which are Cartesian objects in the $2$-category of tangent monads.
\par A \textbf{Cartesian} tangent monad consists of a Cartesian tangent category $(\X,\TT)$ together with a tangent monad $(S,\alpha)$ on $(\X,\TT)$ which preserves finite products. In general, we do not expect $(S,\alpha)$ to lift to the categories of differential objects of $(\X,\TT)$ unless the distributive law $\alpha$ is invertible. We call a \textbf{strong} Cartesian tangent monad a Cartesian monad whose underlying tangent morphism $(S,\alpha)$ is strong.

\begin{lemma}
\label{lemma:differential-objects-tangent-monads}
A strong Cartesian tangent monad $(S,\alpha)$ on a Cartesian tangent category $(\X,\TT)$ lifts to a Cartesian tangent monad on the Cartesian tangent category $\DO(\X,\TT)$ of differential objects of $(\X,\TT)$.
\end{lemma}
\begin{proof}
Strong tangent morphisms send differential objects to differential objects. In particular, $S(A,\zeta,\sigma,\hat p)$ is the differential object $(SA,S\zeta,S\sigma,S\hat p\o\alpha^{-1})$. Furthermore, by naturality and thanks to the compatibility with $\alpha$, the unit $\eta$ and the multiplication $\mu$ of $S$ are morphisms of differential objects.
\end{proof}

\begin{theorem}
\label{theorem:differential-objects-tangent-monads}
In the $2$-category $\TngMnd$ of tangent monads, each strong Cartesian tangent monad $(S,\alpha)$ on a Cartesian tangent category admits the construction of differential objects. Moreover, the tangentad of differential objects of $(S,\alpha)$ is the Cartesian tangent monad $\DO(S,\alpha)$ on $\DO(\X,\TT)$.
\end{theorem}
\begin{proof}
In Lemma~\ref{lemma:differential-objects-tangent-monads}, we already proved that $\DO(S,\alpha)$ defines a tangent monad on the tangent category $\DO(\X,\TT)$ of differential objects of $(\X,\TT)$. Now, we want to prove that $(\DO(\X,\TT),\DO(S,\alpha))$ represents the tangentad of differential objects of the tangent monad $(S,\alpha)$. Consider a tangent category $(\X',\TT')$ equipped with a tangent monad $(S',\alpha')$. A differential object in the Hom-tangent category $[(\X',\TT';S',\alpha')\|(\X,\TT;S,\alpha)]$ is a lax tangent morphism $(G,\beta)\colon(\X',\TT')\to(\X,\TT)$ equipped with a distributive law $\tau\colon S\o G\to G\o S'$, compatible with the monad structures and the distributive laws $\alpha$ and $\alpha'$. Furthermore, $(G,\beta;\tau)$ comes equipped with tangent natural transformations
\begin{align*}
&\zeta_A\colon\*\to GA\\
&\sigma_A\colon GA\times GA\to GA\\
&\hat p_A\colon\T GA\to GA
\end{align*}
which make each $GA$ into a differential object of $(\X,\TT)$. Moreover, $\zeta$, $\sigma$, $\hat p$ must also be natural transformations of monads, that is, they must be compatible with $\tau$. Thus, the lax tangent morphism
\begin{align*}
(H,\gamma)\colon&(\X',\TT')\to\DO(\X,\TT)
\end{align*}
which sends each $A\in(\X',\TT')$ to $(GA,\zeta_A,\sigma_A,\hat p_A)$ comes equipped with a distributive law $\theta$
\begin{align*}
&\theta_A\colon\DO(S,\alpha)(HA)=(SGA,S\zeta_A,S\sigma_A,\alpha^{-1}\o S\hat p_A)\xrightarrow{\gamma_A}(GS'A,GS'\zeta_A,GS'\sigma_A,{\alpha'}^{-1}\o GS'\hat p_A)=H(S'A)
\end{align*}
which makes $(H,\gamma)$ into a morphisms of tangent monads. We let the reader prove that $(H,\gamma;\tau)$ is the unique such morphism.
\end{proof}

%__________________________________________________________________________
\subsubsection*{Tangent fibrations}
\label{subsubsection:differential-objects-tangent-fibrations}
In this section, we consider the construction of differential objects for tangent fibrations. First, let us unwrap the definition of a Cartesian tangent fibration, that is, a Cartesian object in the $2$-category of tangent fibrations.
\par A \textbf{Cartesian} tangent fibration consists of a tangent fibration $\Pi\colon(\X',\TT')\to(\X,\TT)$ between two Cartesian tangent categories $(\X',\TT')$ and $(\X,\TT)$ whose underlying functor strictly preserves finite products, that is, $\Pi$ sends the terminal object of $\X'$ to the (chosen) terminal object of $\X$ and sends each Cartesian product $E\times F$ of $\X'$ to the (chosen) product $\Pi(E)\times\Pi(F)$ of $\X$. Moreover, for each pair of morphisms $f\colon A\to C$ and $g\colon B\to D$ of $\X$, the Cartesian product
\begin{align*}
&\varphi_f\times\varphi_g\colon f^\*E\times g^\*F\to E\times F
\end{align*}
of the (cloven) Cartesian lifts $\varphi_f\colon f^\*E\to E$ and $\varphi_g\colon g^\*F\to F$ is the (cloven) Cartesian lift $\varphi_{f\times g}$ of $f\times g$ onto $E\times F$.

\begin{lemma}
\label{lemma:differential-objects-tangent-fibration}
Consider a (cloven) Cartesian tangent fibration $\Pi\colon(\X',\TT')\to(\X,\TT)$ between two Cartesian tangent categories. The strict tangent morphism
\begin{align*}
&\DO(\Pi)\colon\DO(\X',\TT')\to\DO(\X,\TT)
\end{align*}
which sends a differential object $(E,\zeta_E,\sigma_E,\hat p_E)$ to the differential object $(\Pi(E),\Pi(\zeta_E),\Pi(\sigma_E),\Pi(\hat p_E))$ is a (cloven) Cartesian tangent fibration.
\end{lemma}
\begin{proof}
Consider a (linear) morphism $f\colon(A,\zeta_A,\sigma_A,\hat p_A)\to(B,\zeta_B,\sigma_B,\hat p_B)$ of differential objects in $(\X,\TT)$ and let $(E,\zeta_E,\sigma_E,\hat p_E)$ be a differential object of $(\X',\TT')$ such that $\Pi(E,\zeta_E,\sigma_E,\hat p_E)=(B,\zeta_B,\sigma_B,\hat p_B)$. Using the universal property of the Cartesian lift $\varphi_f\colon f^\*E\to E$ we can define the morphisms
\begin{align*}
&f^\*\zeta_E\colon\*\to f^\*E\\
&f^\sigma_E\colon f^\*E\times f^\*E\to f^\*E\\
&f^\*\hat p_E\colon\T'f^\*E\to f^\*E
\end{align*}
uniquely defined as the lifts of the structure morphisms of the differential object $A$ which commute with $\varphi_f$ is the obvious way. Since $f^\*\colon\Pi^{-1}(B)\to\Pi^{-1}(A)$ is functorial, the equational axioms which make $f^\*E\=(f^\*E,f^\*\zeta_E,f^\*\sigma_E,f^\*\hat p_E)$ into a differential object hold. Thus, we only need to prove the universal property of $f^\*\hat p_E$. Consider the following diagram:
\begin{equation*}
% https://q.uiver.app/#q=WzAsNyxbMiwwLCJYIl0sWzIsMSwiXFxUJ2ZeXFwqRSJdLFswLDEsImZeXFwqRSJdLFs0LDEsImZeXFwqRSJdLFsyLDIsIlxcVCdFIl0sWzAsMiwiRSJdLFs0LDIsIkUiXSxbMSwzLCJmXlxcKlxcaGF0IHBfRSIsMl0sWzAsMiwiXFxhbHBoYSIsMl0sWzAsMywiXFxiZXRhIl0sWzQsNSwicCdfRSJdLFs0LDYsIlxcaGF0IHBfRSIsMl0sWzEsNCwiXFxUJ1xcdmFycGhpX2YiXSxbMiw1LCJcXHZhcnBoaV9mIiwyXSxbMyw2LCJcXHZhcnBoaV9mIl0sWzEsMiwicCdfe2ZeXFwqRX0iXV0=
\begin{tikzcd}
&& X \\
{f^\*E} && {\T'f^\*E} && {f^\*E} \\
E && {\T'E} && E
\arrow["\alpha"', from=1-3, to=2-1]
\arrow["\beta", from=1-3, to=2-5]
\arrow["{\varphi_f}"', from=2-1, to=3-1]
\arrow["{p'_{f^\*E}}", from=2-3, to=2-1]
\arrow["{f^\*\hat p_E}"', from=2-3, to=2-5]
\arrow["{\T'\varphi_f}", from=2-3, to=3-3]
\arrow["{\varphi_f}", from=2-5, to=3-5]
\arrow["{p'_E}", from=3-3, to=3-1]
\arrow["{\hat p_E}"', from=3-3, to=3-5]
\end{tikzcd}
\end{equation*}
By the universal property of $\hat p_E$, there exists a unique morphism $\<\alpha,\beta\>_E\colon X\to\T'E$ making the following diagram commute:
\begin{equation*}
% https://q.uiver.app/#q=WzAsNyxbMiwwLCJYIl0sWzIsMSwiXFxUJ2ZeXFwqRSJdLFswLDEsImZeXFwqRSJdLFs0LDEsImZeXFwqRSJdLFsyLDIsIlxcVCdFIl0sWzAsMiwiRSJdLFs0LDIsIkUiXSxbMSwzLCJmXlxcKlxcaGF0IHBfRSIsMl0sWzAsMiwiXFxhbHBoYSIsMl0sWzAsMywiXFxiZXRhIl0sWzQsNSwicCdfRSJdLFs0LDYsIlxcaGF0IHBfRSIsMl0sWzEsNCwiXFxUJ1xcdmFycGhpX2YiXSxbMiw1LCJcXHZhcnBoaV9mIiwyXSxbMyw2LCJcXHZhcnBoaV9mIl0sWzAsNCwiXFw8XFxhbHBoYSxcXGJldGFcXD5fRSIsMix7ImxhYmVsX3Bvc2l0aW9uIjozMCwiY3VydmUiOjMsInN0eWxlIjp7ImJvZHkiOnsibmFtZSI6ImRhc2hlZCJ9fX1dLFsxLDIsInAnX3tmXlxcKkV9Il1d
\begin{tikzcd}
&& X \\
{f^\*E} && {\T'f^\*E} && {f^\*E} \\
E && {\T'E} && E
\arrow["\alpha"', from=1-3, to=2-1]
\arrow["\beta", from=1-3, to=2-5]
\arrow["{\<\alpha,\beta\>_E}"'{pos=0.3}, curve={height=18pt}, dashed, from=1-3, to=3-3]
\arrow["{\varphi_f}"', from=2-1, to=3-1]
\arrow["{p'_{f^\*E}}", from=2-3, to=2-1]
\arrow["{f^\*\hat p_E}"', from=2-3, to=2-5]
\arrow["{\T'\varphi_f}", from=2-3, to=3-3]
\arrow["{\varphi_f}", from=2-5, to=3-5]
\arrow["{p'_E}", from=3-3, to=3-1]
\arrow["{\hat p_E}"', from=3-3, to=3-5]
\end{tikzcd}
\end{equation*}
However, this diagram lifts the diagram in $(\X,\TT)$:
\begin{equation*}
% https://q.uiver.app/#q=WzAsNyxbMiwwLCJcXFBpKFgpIl0sWzIsMSwiXFxUIEEiXSxbMCwxLCJBIl0sWzQsMSwiQSJdLFsyLDIsIlxcVCBCIl0sWzAsMiwiQiJdLFs0LDIsIkIiXSxbMSwzLCJcXGhhdCBwX0EiLDJdLFswLDIsIlxcUGkoXFxhbHBoYSkiLDJdLFswLDMsIlxcUGkoXFxiZXRhKSJdLFs0LDUsInBfQiJdLFs0LDYsIlxcaGF0IHBfQiIsMl0sWzEsNCwiXFxUIGYiXSxbMiw1LCJmIiwyXSxbMyw2LCJmIl0sWzAsNCwiXFw8XFxQaShcXGFscGhhKSxcXFBpKFxcYmV0YSlcXD5fQiIsMSx7ImxhYmVsX3Bvc2l0aW9uIjozMCwiY3VydmUiOjMsInN0eWxlIjp7ImJvZHkiOnsibmFtZSI6ImRhc2hlZCJ9fX1dLFsxLDIsInBfQSJdXQ==
\begin{tikzcd}
&& {\Pi(X)} \\
A && {\T A} && A \\
B && {\T B} && B
\arrow["{\Pi(\alpha)}"', from=1-3, to=2-1]
\arrow["{\Pi(\beta)}", from=1-3, to=2-5]
\arrow["{\<\Pi(\alpha),\Pi(\beta)\>_B}"{description, pos=0.3}, curve={height=18pt}, dashed, from=1-3, to=3-3]
\arrow["f"', from=2-1, to=3-1]
\arrow["{p_A}", from=2-3, to=2-1]
\arrow["{\hat p_A}"', from=2-3, to=2-5]
\arrow["{\T f}", from=2-3, to=3-3]
\arrow["f", from=2-5, to=3-5]
\arrow["{p_B}", from=3-3, to=3-1]
\arrow["{\hat p_B}"', from=3-3, to=3-5]
\end{tikzcd}
\end{equation*}
Thus, by the universal property of $\hat p_A$ and by the universal property of $\T'\varphi_f$, there exists a unique morphism $\<\alpha,\beta\>_{f^\*E}\colon X\to\T'f^\*E$ such that $\Pi(\<\alpha,\beta\>_{f^\*E})=\<\Pi(\alpha),\Pi(\beta)\>_A\colon\Pi(X)\to\T A$, making the following diagram commute:
\begin{equation*}
% https://q.uiver.app/#q=WzAsNyxbMiwwLCJYIl0sWzIsMSwiXFxUJ2ZeXFwqRSJdLFswLDEsImZeXFwqRSJdLFs0LDEsImZeXFwqRSJdLFsyLDIsIlxcVCdFIl0sWzAsMiwiRSJdLFs0LDIsIkUiXSxbMSwzLCJmXlxcKlxcaGF0IHBfRSIsMl0sWzAsMiwiXFxhbHBoYSIsMl0sWzAsMywiXFxiZXRhIl0sWzQsNSwicCdfRSJdLFs0LDYsIlxcaGF0IHBfRSIsMl0sWzEsNCwiXFxUJ1xcdmFycGhpX2YiXSxbMiw1LCJcXHZhcnBoaV9mIiwyXSxbMyw2LCJcXHZhcnBoaV9mIl0sWzEsMiwicCdfe2ZeXFwqRX0iXSxbMCw0LCJcXDxcXGFscGhhLFxcYmV0YVxcPl9FIiwyLHsibGFiZWxfcG9zaXRpb24iOjMwLCJjdXJ2ZSI6M31dLFswLDEsIlxcPFxcYWxwaGEsXFxiZXRhXFw+X3tmXlxcKkV9IiwwLHsic3R5bGUiOnsiYm9keSI6eyJuYW1lIjoiZGFzaGVkIn19fV1d
\begin{tikzcd}
&& X \\
{f^\*E} && {\T'f^\*E} && {f^\*E} \\
E && {\T'E} && E
\arrow["\alpha"', from=1-3, to=2-1]
\arrow["{\<\alpha,\beta\>_{f^\*E}}", dashed, from=1-3, to=2-3]
\arrow["\beta", from=1-3, to=2-5]
\arrow["{\<\alpha,\beta\>_E}"'{pos=0.3}, curve={height=18pt}, from=1-3, to=3-3]
\arrow["{\varphi_f}"', from=2-1, to=3-1]
\arrow["{p'_{f^\*E}}", from=2-3, to=2-1]
\arrow["{f^\*\hat p_E}"', from=2-3, to=2-5]
\arrow["{\T'\varphi_f}", from=2-3, to=3-3]
\arrow["{\varphi_f}", from=2-5, to=3-5]
\arrow["{p'_E}", from=3-3, to=3-1]
\arrow["{\hat p_E}"', from=3-3, to=3-5]
\end{tikzcd}
\end{equation*}
Thus, $f^\*\hat p_E$ is universal. We leave it to the reader to show that $\varphi_f$ is Cartesian in $\DO(\X',\TT')$.
\end{proof}

\begin{theorem}
\label{theorem:differential-objects-tangent-fibrations}
The $2$-category $\Fib$ of (cloven) fibrations admits the construction of differential objects. In particular, the tangentad of differential objects of a (cloven) Cartesian tangent fibration $\Pi$ is the Cartesian tangent fibration $\DO(\Pi)\colon\DO(\X',\TT')\to\DO(\X,\TT)$.
\end{theorem}
\begin{proof}
Consider a tangent fibration $\Pi_\b\colon(\X'_\b,\TT_\b)\to(\X_\b,\TT_\b)$ and a morphism $(G,\beta;G',\beta')\colon\Pi_\b\to\Pi$ of tangent fibrations, together with a differential object structure in the Hom-tangent category $[\Pi_\b,\Pi]$. Concretely, this consists of a list of natural transformations
\begin{align*}
\zeta_A&\colon\*\to GA              &\zeta'_E&\colon\*\to G'E\\
\sigma_A&\colon GA\times GA\to GA   &\sigma'_E&\colon G'E\times G'E\to G'E\\
\hat p_A&\colon\T GA\to GA          &\hat p'_E&\colon\T'G'E\to G'E
\end{align*}
for each $A\in(\X_\b,\TT_\b)$ and $E\in(\X_\b',\TT_\b')$, such that $(GA,\zeta_A,\sigma_A,\hat p_A)$ and $(G'E,\zeta_E',\sigma_E',\hat p_E')$ are differential objects in $(\X,\TT)$ and $(\X',\TT')$, respectively and satisfying the following equations:
\begin{align*}
&\zeta_{\Pi_\b(E)}=\Pi(\zeta'_E)\\
&\sigma_{\Pi_\b(E)}=\Pi(\sigma'_E)\\
&\hat p_{\Pi_\b(E)}=\Pi(\hat p'_E)
\end{align*}
for every $E\in(\X',\TT')$. In particular, for each $E\in(\X',\TT')$,
\begin{align*}
&\DO(\Pi)(G'E,\zeta'_E,\sigma'_E,\hat p'_E)=(\Pi(G'E),\Pi(\zeta'_E),\Pi(\sigma'_E),\Pi(\hat p'_E)=(G\Pi_\b(E),\zeta_{\Pi_\b(E)},\sigma_{\Pi_\b(E)},\hat p_{\Pi_\b(E)})
\end{align*}
Define the following tangent morphisms:
\begin{align*}
(H,\gamma)&\colon(\X_\b,\TT_\b)\to\DO(\X,\TT)\\
(H',\gamma')&\colon(\X_\b',\TT_\b')\to\DO(\X',\TT')
\end{align*}
which send each $A\in(\X_\b,\TT_\b)$ and each $E\in(\X_\b',\TT_\b')$ to $(GA,\zeta_A,\sigma_A,\hat p_A)$ and $(G'E,\zeta'_E,\sigma'_E,\hat p'_E)$, respectively. Thus:
\begin{align*}
&\DO(\Pi)(H'(E))=\DO(\Pi)(G'E,\zeta'_E,\sigma'_E,\hat p'_E)=(G\Pi_\b(E),\zeta_{\Pi_\b(E)},\sigma_{\Pi_\b(E)},\hat p_{\Pi_\b(E)})=H'_{\Pi_\b(E)}
\end{align*}
In particular, $(H,\gamma)$ and $(H',\gamma')$ commute with the tangent fibrations $\Pi_\b$ and $\DO(\Pi)$. To prove that $(H,\gamma;H',\gamma')$ is a morphism of tangent fibrations, we need to show that $(H,H')$ preserves Cartesian lifts. Consider a morphism $f\colon A\to B$ of $(\X_\b,\TT_\b)$ and let $\varphi_f\colon f^\*E\to E$ be a Cartesian lift of $f$ along $\Pi_\b$. Since $(G,G')\colon\Pi_\b\to\Pi$ is a morphism of fibrations, $G'\varphi_f\colon G'F^\*E\to G'E$ is a Cartesian lift of $Gf\colon GA\to GB$. Furthermore, by the naturality of the structure morphisms $(\zeta,\zeta')$, $(\sigma',\sigma')$, and $(\hat p,\hat p')$, $H'\varphi_f=G'\varphi_f\colon H'f^\*E\to H'E$ is a morphism of differential objects. We need to prove that $H'\varphi_f=G'\varphi_f$ is Cartesian in $\DO(\X',\TT')$. Consider a commutative diagram
\begin{equation*}
% https://q.uiver.app/#q=WzAsMyxbMCwxLCIoR0EsXFx6ZXRhX0EsXFxzaWdtYV9BLFxcaGF0IHBfQSkiXSxbMSwxLCIoR0IsXFx6ZXRhX0IsXFxzaWdtYV9CLFxcaGF0IHBfQikiXSxbMCwwLCIoR0MsXFx6ZXRhX0MsXFxzaWdtYV9DLFxcaGF0IHBfQykiXSxbMCwxLCJIZiIsMl0sWzIsMSwiZyJdLFsyLDAsImgiLDJdXQ==
\begin{tikzcd}
{(GC,\zeta_C,\sigma_C,\hat p_C)} \\
{(GA,\zeta_A,\sigma_A,\hat p_A)} & {(GB,\zeta_B,\sigma_B,\hat p_B)}
\arrow["h"', from=1-1, to=2-1]
\arrow["g", from=1-1, to=2-2]
\arrow["Hf"', from=2-1, to=2-2]
\end{tikzcd}
\end{equation*}
in $\DO(\X,\TT)$. Since $G'\varphi_f$ is Cartesian in $(\X',\TT')$, there exists a unique morphism $\xi_h\colon G'f^\*E\to G'E$ in $(\X',\TT')$ such that $\Pi(\xi_h)=h$ and making the following diagram commutative:
\begin{equation*}
% https://q.uiver.app/#q=WzAsMyxbMCwxLCJHJ2ZeXFwqRSJdLFsxLDEsIkcnRSJdLFswLDAsImdeXFwqRydFIl0sWzAsMSwiRydcXHZhcnBoaV9mIiwyXSxbMiwxLCJcXHZhcnBoaV9nIl0sWzIsMCwiXFx4aV9oIiwyXV0=
\begin{tikzcd}
{g^\*G'E} \\
{G'f^\*E} & {G'E}
\arrow["{\xi_h}"', from=1-1, to=2-1]
\arrow["{\varphi_g}", from=1-1, to=2-2]
\arrow["{G'\varphi_f}"', from=2-1, to=2-2]
\end{tikzcd}
\end{equation*}
We want to prove that $\xi_h$ is a morphism of differential objects. For the sake of brevity, we show how to prove the compatibility between $\xi_h$ and the differential projections and leave it to the reader to apply the same technique for the other compatibilities. Consider the diagram:
\begin{equation*}
% https://q.uiver.app/#q=WzAsNixbMSwyLCJHJ2ZeXFwqRSJdLFs0LDEsIkcnRSJdLFsxLDAsImdeXFwqRydFIl0sWzAsMiwiXFxUJ0cnZl5cXCpFIl0sWzAsMCwiXFxUJ2deXFwqRydFIl0sWzIsMSwiXFxUJ0cnRSJdLFswLDEsIkcnXFx2YXJwaGlfZiIsMl0sWzIsMCwiXFx4aV9oIl0sWzQsMywiXFxUJ1xceGlfaCIsMl0sWzQsNSwiXFxUJ1xcdmFycGhpX2ciLDIseyJsYWJlbF9wb3NpdGlvbiI6MjB9XSxbNCwyLCJnXlxcKlxcaGF0IHBfe0cnRX0iXSxbMywwLCJHJ2ZeXFwqXFxoYXQgcF9FIiwyXSxbMiwxLCJcXHZhcnBoaV9nIl0sWzUsMSwiXFxoYXQgcF97RydFfSIsMV0sWzMsNSwiXFxUJ0cnXFx2YXJwaGlfZiIsMCx7ImxhYmVsX3Bvc2l0aW9uIjoyMH1dXQ==
\begin{tikzcd}
{\T'g^\*G'E} & {g^\*G'E} \\
&& {\T'G'E} && {G'E} \\
{\T'G'f^\*E} & {G'f^\*E}
\arrow["{g^\*\hat p_{G'E}}", from=1-1, to=1-2]
\arrow["{\T'\varphi_g}"'{pos=0.2}, from=1-1, to=2-3]
\arrow["{\T'\xi_h}"', from=1-1, to=3-1]
\arrow["{\varphi_g}", from=1-2, to=2-5]
\arrow["{\xi_h}", from=1-2, to=3-2]
\arrow["{\hat p_{G'E}}"{description}, from=2-3, to=2-5]
\arrow["{\T'G'\varphi_f}"{pos=0.2}, from=3-1, to=2-3]
\arrow["{G'f^\*\hat p_E}"', from=3-1, to=3-2]
\arrow["{G'\varphi_f}"', from=3-2, to=2-5]
\end{tikzcd}
\end{equation*}
The right and the left triangles commute since $G'\varphi_f\o\xi_h=\varphi_g$ and so do the top and bottom parallelogram-shaped diagrams since $G'\varphi_f$ and $\varphi_g$ are morphisms of differential objects. Therefore:
\begin{align*}
&G'\varphi_f\o\xi_h\o g^\*\hat p_{G'E}=G'\varphi_f\o G'f^\*\hat p_E\o\T'\xi_h
\end{align*}
However, since $h$ is a morphism of differential objects in $(\X,\TT)$, we can also compute:
\begin{align*}
&\Pi(\xi_h\o g^\*\hat p_{G'E})=h\o\hat p_C=\hat p_B\o\T h=\Pi(G'f^\*\hat p_E\o\T'\xi_h)
\end{align*}
Thus, by the universal property of $G'\varphi_f$, we conclude that:
\begin{align*}
&\xi_h\o g^\*\hat p_{G'E}=G'f^\*\hat p_E\o\T'\xi_h
\end{align*}
proving that $\xi_h$ is compatible with the differential projections. We leave it to the reader to show that $(H,\gamma;H',\gamma')$ is the unique morphism of tangent fibrations which corresponds to $(G,\beta;G',\beta';\zeta,\zeta';\sigma,\sigma';\hat p,\hat p')$ in the Hom-tangent category $[\Pi_\b,\Pi]$.
\end{proof}

%__________________________________________________________________________
\subsubsection*{Tangent indexed categories}
\label{subsubsection:differential-objects-tangent-categories}
In this section, we employ the Grothendieck $2$-equivalence $\TngFib\simeq\TngIndx$ to compute the tangentad of differential objects of a tangent indexed category. Let us begin by defining a Cartesian tangent indexed category, that is, a Cartesian object in the $2$-category of tangent indexed category. A \textbf{Cartesian} tangent indexed category consists of a tangent indexed category $(\X,\TT;\IND,\TT')$ whose base tangent category $(\X,\TT)$ is Cartesian. Moreover, there is an object $\*$ in $\X^\*\=\IND(\*)$ such that, for any object $E\in\X^M$ for each $M\in(\X,\TT)$, there exists a unique morphism:
\begin{align*}
&!^E\colon E\to {!^M}^\*\*
\end{align*}
where $!^M\colon M\to\*$ is the terminal morphism. Furthermore, for each pairs of objects $E\in\X^M$ and $F\in\X^N$, there exists an object $E\times F\in\X^{M\times N}$ together with two morphisms
\begin{align*}
&\pi_1\colon E\times F\to{\pi_1^{M\times N}}^\*E\\
&\pi_2\colon\colon E\times F\to{\pi_2^{M\times N}}^\*F\\
\end{align*}
where, $\pi_1^{M\times N}\colon M\times N\to M$ and $\pi_2^{M\times N}\colon M\times N\to N$ are the projections, such that, for any pair of morphisms $f\colon L\to M$ and $g\colon L\to N$ of $(\X,\TT)$ and any pair of morphisms $\varphi\colon X\to f^\*E$ and $\psi\colon X\to g^\*F$ of $\X^L$, there exists a unique morphism 
\begin{align*}
&\<\varphi,\psi\>\colon X\to\<f,g\>^\*(E\times F)
\end{align*}
of $\X^L$ which makes the following diagram commutative:
\begin{equation*}
% https://q.uiver.app/#q=WzAsNixbMSwwLCJYIl0sWzEsMSwiXFw8ZixnXFw+XlxcKihFXFx0aW1lcyBGKSJdLFswLDEsIlxcPGYsZ1xcPl5cXCp7XFxwaV8xXnsoTVxcdGltZXMgTil9fV5cXCooRSkiXSxbMiwxLCJcXDxmLGdcXD5eXFwqe1xccGlfMl57KE1cXHRpbWVzIE4pfX1eXFwqKEYpIl0sWzAsMCwiZl5cXCpFIl0sWzIsMCwiZ15cXCpGIl0sWzAsMSwiXFw8XFx2YXJwaGksXFxwc2lcXD4iLDAseyJzdHlsZSI6eyJib2R5Ijp7Im5hbWUiOiJkYXNoZWQifX19XSxbMSwyLCJcXDxmLGdcXD5eXFwqXFxwaV8xIl0sWzEsMywiXFw8ZixnXFw+XlxcKlxccGlfMiIsMl0sWzAsNCwiXFx2YXJwaGkiLDJdLFswLDUsIlxccHNpIl0sWzQsMiwiXFxJXzJeey0xfSIsMl0sWzUsMywiXFxJXzJeey0xfSJdXQ==
\begin{tikzcd}
{f^\*E} & X & {g^\*F} \\
{\<f,g\>^\*{\pi_1^{M\times N}}^\*(E)} & {\<f,g\>^\*(E\times F)} & {\<f,g\>^\*{\pi_2^{M\times N}}^\*(F)}
\arrow["{\IND_2^{-1}}"', from=1-1, to=2-1]
\arrow["\varphi"', from=1-2, to=1-1]
\arrow["\psi", from=1-2, to=1-3]
\arrow["{\<\varphi,\psi\>}", dashed, from=1-2, to=2-2]
\arrow["{\IND_2^{-1}}", from=1-3, to=2-3]
\arrow["{\<f,g\>^\*\pi_1}", from=2-2, to=2-1]
\arrow["{\<f,g\>^\*\pi_2}"', from=2-2, to=2-3]
\end{tikzcd}
\end{equation*}
where we used the inverse of the compositor $\IND_2$ of $\IND$:
\begin{align*}
&\IND_2\colon\<f,g\>^\*{\pi_1^{M\times N}}^\*\cong(\pi_1^{M\times N}\o\<f,g\>)^\*=f^\*\\
&\IND_2\colon\<f,g\>^\*{\pi_2^{M\times N}}^\*\cong(\pi_2^{M\times N}\o\<f,g\>)^\*=g^\*
\end{align*}
The indexed tangent functor $\T^M\colon\X^M\to\T^{\T M}$ preserves the indexed terminal object $\*\in\X^\*$ and the indexed Cartesian products $E\times F\in\X^{M\times N}$, that is:
\begin{align*}
&\T^{\*}\*=\*\\
&\T^{M\times N}(E\times F)=\T^ME\times\T^NF\\
&\T^M(!^E\colon E\to{!^M}^\*\*)=(!_{\T^ME}\colon\T^ME\to !^{\T M}\*)\\
&\T^{M\times N}(\pi_1\colon E\times F\to{\pi_1^{M\times N}}^\*E)=(\pi_1\colon\T^ME\times\T^NF\to{\pi_1^{\T M\times\T N}}^\*\T^ME\\
&\T^{M\times N}(\pi_2\colon E\times F\to{\pi_2^{M\times N}}^\*F)=(\pi_2\colon\T^ME\times\T^NF\to{\pi_2^{\T M\times\T N}}^\*\T^NF
\end{align*}

Let us start by proving that $2$-equivalences preserve the construction of differential objects.

\begin{proposition}
\label{proposition:equivalence-differential-objects}
Let $\CC$ and $\CC'$ be two $2$-categories and suppose there is a $2$-equivalence $\Xi\colon\Tng(\CC)\simeq\Tng(\CC')$ of the $2$-categories of tangentads of $\CC$ and $\CC'$. If a Cartesian tangentad $(\X,\TT)$ of $\CC$ admits the construction of differential objects and $\DO(\X,\TT)$ denotes the tangentad of differential objects of $(\X,\TT)$, so does $\Xi(\X,\TT)$ and the tangentad of differential objects of $\Xi(\X,\TT)$ is $\Xi(\DO(\X,\TT))$.
\end{proposition}
\begin{proof}
The proof is similar to the one of~\cite[Proposition~5.5]{lanfranchi:tangentads-II}.
\end{proof}

Proposition~\ref{proposition:equivalence-differential-objects} together with the Grothendieck $2$-equivalence $\TngFib\cong\TngIndx$ allows us to compute the differential objects of tangent indexed categories. Consider a tangent indexed category $(\X,\TT;\IND,\TT')$ whose underlying indexed category $\IND\colon\X^\op\to\Cat$ sends an object $M$ of $\X$ to the category $\X^M$ and a morphism $f\colon M\to N$ to the functor $f^\*\colon\X^N\to\X^M$. Define $\DO(\X,\TT;\IND,\TT')$ as follows:
\begin{description}
\item[Base tangent category] The base tangent category is $\DO(\X,\TT)$;

\item[Indexed category] The indexed category $\DO(\IND)\colon\DO(\X,\TT)^\op\to\Cat$ sends each $(A,\zeta_A,\sigma_A,\hat p_A)$ to the category $\DO^{A}$ whose objects are four-tuples $(E,\zeta_E,\sigma_E,\hat p_E)$ formed by an object $E\in\X^M$ and three morphisms
\begin{align*}
\zeta_E&\colon\*\to\zeta_A^\*E\\
\sigma_E&\colon E\times E\to\sigma_A^\*E\\
\hat p_E&\colon\T'^AE\to\hat p^\*E
\end{align*}
such that $(A,E)$ equipped with 
\begin{align*}
(\zeta_A,\zeta_E)&\colon(\*,\*)\to(A,E)\\
(\sigma_A,\sigma_E)&\colon(A,E)\times(A,E)\to(A,E)\\
(\hat p_A,\hat p_E)&\colon\T(A,E)\to(A,E)
\end{align*}
becomes a differential object in the tangent category of elements of the indexed tangent category.\\
Furthermore, $\DO(\IND)$ sends a morphism $f\colon(A,\zeta_A,\sigma_A,\hat p_A)\to(B,\zeta_B,\sigma_B,\hat p_B)$ of differential objects to the functor
\begin{align*}
f^\*&\colon\DO^{(A,\zeta_A,\sigma_A,\hat p_A)}\to\DO^{(B,\zeta_B,\sigma_B,\hat p_B)}
\end{align*}
which sends a tuple $(E,\zeta_E,\sigma_E,\hat p_E)$ to the tuple $(f^\*E,\zeta_{f^\*E},\sigma_{f^\*E},\hat p_{f^\*E})$, where:
\begin{align*}
\zeta_{f^\*E}&\colon\*\xrightarrow{\zeta_E}\zeta_A^\*E\xrightarrow{\IND_2}\zeta_B^\*f^\*E\\
\sigma_{f^\*E}&\colon f^\*E\times f^\*E\xrightarrow{\cong}(f\times f)^\*(E\times E)\xrightarrow{(f\times f)^\*\sigma_E}(f\times f)^\*\sigma_A^\*E\xrightarrow{\IND_2}\sigma_Bf^\*E\\
\hat p_{f^\*E}&\colon\T^Bf^\*E\xrightarrow{\cong}(\T f)^\*\T'^AE\xrightarrow{(\T f)^\*\hat p_E}(\T f)^\*(\hat p_A)^\*E\xrightarrow{\IND_2}\hat p_A^\*f^\*E
\end{align*}

\item[Indexed tangent bundle functor] The indexed tangent bundle functor $\DO(\T')$ consists of the list of functors
\begin{align*}
&\T'^{(A,\zeta_A,\sigma_A,\hat p_A)}\colon\DO^{(A,\zeta_A,\sigma_A,\hat p_A)}\to\DO^{(\T A,\T\zeta_A,\T\sigma_A,\hat p_{\T A})}
\end{align*}
which send a tuple $(E,\zeta_E,\sigma_E,\hat p_E)$ to the tuple $(\T'^AE,\T'^A\zeta_E,\T'^A\sigma_E,\hat p_{\T'^AE})$ where $w_{\T'^AE})$ is defined as follows:
\begin{align*}
&\hat p_{\T'^AE}\colon\T'^{\T A}\T'^AE\xrightarrow{c'^A}c_A^\*\T'^{\T A}\T'^AE\xrightarrow{c^\*_A\T'^{\T A}\hat p_E}c_A^\*\T'^{\T A}\hat p_A^\*E\xrightarrow{c_A^\*\xi^{\hat p_A}}c_A^\*(\T\hat p_A)^\*\T'^AE\xrightarrow{\IND_2}\hat p_{\T A}^\*\T'^AE
\end{align*}
Moreover, $\T'^{(A,\zeta_A,\sigma_A,\hat p_A)}$ sends a morphism $\varphi\colon(E,\zeta_E,\sigma_E,\hat p_E)\to(E',\zeta_E',\sigma_E',\hat p_E')$ to $\T'^A\varphi$. The distributors of $\T'^{(A,\zeta_A,\sigma_A,\hat p_A)}$ coincide with the distributors $\xi^f$ of $\T'^A$;

\item[Indexed natural transformations] The structural indexed natural transformations of $\DO(\IND)$ are the same as for $\IND$.
\end{description}

\begin{theorem}
\label{theorem:differential-objects-tangent-indexed-categories}
The $2$-category $\Indx$ of indexed categories admits the construction of differential objects. In particular, given a tangent indexed category $(\X,\TT;\IND,\TT')$, the tangent indexed category of differential objects of $(\X,\TT;\IND,\TT')$ is the tangent indexed category $\DO(\X,\TT;\IND,\TT')=(\DO(\X,\TT),\DO(\IND),\DO(\TT'))$.
\end{theorem}
\begin{proof}
Thanks to Proposition~\ref{proposition:equivalence-differential-objects}, the Grothendieck $2$-equivalence $\Fib\simeq\Indx$ of~\cite[Theorem~5.5]{lanfranchi:grothendieck-tangent-cats} implies that $\Indx$ admits the construction of differential objects. Furthermore, $\DO(\X,\TT;\IND,\TT')$ corresponds via the Grothendieck equivalence to $\DO(\Pi_{(\IND,\TT')})$, where $\Pi_{(\IND,\TT')}$ is the tangent fibration associated to $(\IND,\TT')$. By spelling out the details, one finds out that this corresponds to the tangent indexed category $(\DO(\X,\TT),\DO(\IND),\DO(\TT'))$.
\end{proof}

%__________________________________________________________________________
\subsubsection*{Tangent split restriction categories}
\label{subsubsection:differential-objects-tangent-split-restriction-categories}
In this section, we consider the construction of differential objects for tangent split restriction categories. First, notice that a Cartesian tangent split restriction category, that is, a Cartesian object in the $2$-category of tangent split restriction categories, is precisely a tangent split restriction category with restriction finite products which are preserved by the tangent bundle functor.

\begin{lemma}
\label{lemma:differential-objects-tangent-split-restriction-categories}
Consider a Cartesian tangent split restriction category $(\X,\TT)$ and let $\DO(\X,\TT)$ be the category whose objects are differential objects $(A,\zeta_A,\sigma_A,\hat p_A)$ in the Cartesian tangent subcategory $\Tot(\X,\TT)$ and morphisms $f\colon(A,\zeta_A,\sigma_A,\hat p_A)\to(B,\zeta_B,\sigma_B,\hat p_B)$ are morphisms $f$ of differential objects such that also their restriction idempotents $\bar f$ are morphisms of differential objects. The tangent split restriction structure on $(\X,\TT)$ lifts to $\DO(\X,\TT)$. In particular, the tangent bundle functor sends each differential object $(A,\zeta_A,\sigma_A,\hat p_A)$ to $(\T A,\T\zeta_A,\T\sigma_A,\hat p_{\T A})$, where
\begin{align*}
&\hat p_{\T A}\colon\T^2A\xrightarrow{c}\T^2A\xrightarrow{\T\hat p_A}\T A
\end{align*}
\end{lemma}
\begin{proof}
It is not hard to show that, since the restriction idempotents commute with the structural morphisms of differential objects, the splittings of the restriction idempotents in $(\X,\TT)$ lift to $\DO(\X,\TT)$. Thus, $\DO(\X,\TT)$ is a split restriction category. Moreover, $\TT^\DO$ defines a tangent restriction structure on $\DO(\X,\TT)$ if and only if it restricts to a tangent structure on $\Tot(\DO(\X,\TT))$, since restriction products in a split restriction category coincide with total products in the subcategory of total maps. However, by definition, $\TT^\DO$ restricts to a tangent structure on $\Tot(\DO(\X,\TT))$. Thus, $\DO(\X,\TT)$ is a tangent split restriction category.
\end{proof}

\begin{definition}
\label{definition:restriction-differential-object}
A \textbf{restriction differential object} in a Cartesian tangent split restriction category $(\X,\TT)$ is an object of $\DO(\X,\TT)$.
\end{definition}

Concretely, a restriction differential object consists of an object $A$ of $(\X,\TT)$ together with three total maps
\begin{align*}
\zeta_A&\colon\*\to A       &\sigma_A&\colon A\times A\to A     &\hat p_A&\colon\T A\to A
\end{align*}
where $A\times A$ is the restriction product of $A$ with itself. The total maps $\zeta_A$, $\sigma_A$, and $\hat p_A$ satisfy the same equational axioms of the structure morphisms of a differential object in ordinary tangent category theory. However, the universality of the differential projection $\hat p_A$ is replaced with the following axiom: the following is a restriction pullback diagram:
\begin{equation*}
% https://q.uiver.app/#q=WzAsMyxbMCwwLCJBIl0sWzEsMCwiXFxUIEEiXSxbMiwwLCJBIl0sWzEsMCwicF9BIiwyXSxbMSwyLCJcXGhhdCBwX0EiXV0=
\begin{tikzcd}
A & {\T A} & A
\arrow["{p_A}"', from=1-2, to=1-1]
\arrow["{\hat p_A}", from=1-2, to=1-3]
\end{tikzcd}
\end{equation*}

\begin{theorem}
\label{theorem:differential-objects-tangent-split-restriction-categories}
The $2$-category $\sRestrCat$ of split restriction categories admits the construction of differential objects. In particular, the tangentad of differential objects of a Cartesian tangent split restriction category $(\X,\TT)$ is the Cartesian tangent split restriction category $\DO(\X,\TT)$.
\end{theorem}
\begin{proof}
Consider a tangent split restriction category $(\X',\TT')$ with a lax tangent morphism $(F,\alpha)\colon(\X',\TT')\to(\X,\TT)$, that is, a restriction functor $F\colon\X\to\X'$ together with a total natural transformation $\alpha_A\colon F\T'A\to\T FA$, compatible with the tangent structures. Furthermore, consider also a differential object $(\zeta,\sigma,\hat p)$ in the Hom-tangent category $[\X',\TT'\|\X,\TT]$ over $(F,\alpha)$. Thus, $(FA,\zeta_A,\*\to FA,\sigma_A\colon FA\times FA\to FA,\hat p_A\colon\T A\to A)$ is a restriction differential object. Thus, we can define a tangent morphism $(H,\gamma)\colon(\X',\TT')\to\DO(\X,\TT)$ which sends each object $A\in\X'$ to the tuple $(FA,\zeta_A,\sigma_A,\hat p_A)$ and each morphism $f\colon A\to B$ to $Ff$. Since $F$ is a restriction functor, $H$ also preserves the restriction idempotents, that is, $H\bar f=F\bar f=\bar{Ff}=\bar{Hf}$. It is easy to prove that $(H,\gamma)\colon(\X',\TT')\to\DO(\X,\TT)$ is the unique lax tangent morphism which corresponds to the differential object $(F,\alpha;\zeta,\sigma,\hat p)$.
\end{proof}

%__________________________________________________________________________
\subsubsection*{Tangent restriction categories: a general approach}
\label{subsubsection:differential-objects-tangent-restriction-categories}
Mimicking the work done in~\cite[Section~5.5]{lanfranchi:tangentads-II} to extend the construction of vector fields to general tangent (non-necessarily split) restriction categories, now, we introduce a general procedure for extending the formal theory of differential objects to tangent-like concepts which might fail to be tangentads. Thus, we apply this general construction to the case of tangent restriction categories.
\par Consider a \textbf{pullback-extension context}, that is, two $2$-categories $\CC$ and $\DD$ together with two $2$-functors
\begin{align*}
&\Xi\colon\DD\leftrightarrows\Tng(\CC)\colon\Inc
\end{align*}
together with a natural $2$-transformation
\begin{align*}
\eta_\X&\colon\X\to\Inc(\Xi(\X))
\end{align*}
natural in $\X\in\DD$. Let us also consider an object $\X$ of $\DD$ such that the Cartesian tangentad $\Xi(\X)$ of $\CC$ admits the construction of differential objects. Finally, let us assume the existence of the following $2$-pullback in $\DD$:
\begin{equation}
\label{equation:pullback-extension-differential-objects}
% https://q.uiver.app/#q=WzAsNCxbMSwxLCJcXEluYyhcXFhpKFxcWCkpIl0sWzAsMSwiXFxYIl0sWzAsMCwiXFxETyhcXFgpIl0sWzEsMCwiXFxJbmMoXFxETyhcXFhpKFxcWCkpIl0sWzEsMCwiXFxldGEiLDJdLFsyLDEsIlxcVV9cXFgiLDJdLFsyLDAsIiIsMSx7InN0eWxlIjp7Im5hbWUiOiJjb3JuZXIifX1dLFszLDAsIlxcSW5jKFxcVV97XFxYaShcXFgpfSkiXSxbMiwzLCJcXERPKFxcZXRhKSJdXQ==
\begin{tikzcd}
{\DO(\X)} & {\Inc(\DO(\Xi(\X))} \\
\X & {\Inc(\Xi(\X))}
\arrow["{\DO(\eta)}", from=1-1, to=1-2]
\arrow["{\U_\X}"', from=1-1, to=2-1]
\arrow["\lrcorner"{anchor=center, pos=0.125}, draw=none, from=1-1, to=2-2]
\arrow["{\Inc(\U_{\Xi(\X)})}", from=1-2, to=2-2]
\arrow["\eta"', from=2-1, to=2-2]
\end{tikzcd}
\end{equation}

\begin{definition}
\label{definition:differential-objects-extension}
Let $(\CC,\DD;\Inc,\Xi;\eta)$ be a pullback-extension context. An object $\X$ of $\DD$ admits the \textbf{extended construction of differential objects} (w.r.t. to the pullback-extension context) if $\Xi(\X)$ is a Cartesian tangentad of $\CC$, admits the construction of differential objects in $\CC$, and the $2$-pullback diagram of Equation~\eqref{equation:pullback-extension-differential-objects} exists in $\DD$. In this scenario, the \textbf{object of differential objects} (w.r.t. to the pullback-extension context) of $\X$ is the object $\DO(\X)$ of $\DD$, that is, the $2$-pullback of $\Inc(\U_{\Xi(\X)})$ along $\eta$.
\end{definition}

In the previous section, we showed that the tangentad of differential objects of a tangent split restriction tangent category $(\X,\TT)$ is the tangent split restriction category $\DO(\X,\TT)$ whose objects are tuples $(A,\zeta_A,\sigma_A,\hat p_A)$ which form differential objects in the tangent subcategory of total maps of $(\X,\TT)$.
\par Now, consider a generic tangent restriction category $(\X,\TT)$ and let us unwrap the definition of $\DO(\Split_R(\X,\TT))$. The objects of $\DO(\Split_R(\X,\TT))$ are tuples $(A,e;\zeta_A,\sigma_A,\hat p_A)$ formed by an object $A$ of $(\X,\TT)$ together with a restriction idempotent $e=\bar e\colon A\to A$ of $(\X,\TT)$ and three morphisms
\begin{align*}
\zeta_A&\colon\*\to A &\sigma_A&\colon A\times A\to A &\hat p_A&\colon\T A\to A
\end{align*}
such that:
\begin{align*}
&\bar{\zeta_A}=\id_\*       &&e\o\zeta_A=\zeta_A\\
&\bar{\sigma_A}=e\times e   &&e\o\sigma_A=\sigma_A\\
&\bar{\hat p_A}=\T e        &&e\o\hat p_A=\hat p_A 
\end{align*}
Moreover, $\zeta_A$, $\sigma_A$, $\hat p_A$ satisfy the usual equational axioms of a differential object, and for each restriction idempotent $e'=\bar{e'}\colon X\to X$ and each pair of morphisms $f,g\colon X\to A$ such that $e\o f=f=f\o e'$ and $e\o g=g=g\o e'$, there exists a unique morphism $\<f,g\>\colon X\to\T A$ such that $\T e\o\<f,g\>=\<f,g\>\o e'$ and $p_A\o\<f,g\>=f$ and $\hat p_A\o\<f,g\>=g$.
\par Notice that the equation $\bar{\hat p_A}=\T e$ fully determines $e$ as follows:
\begin{align*}
&e=p\o z\o e=p\o\T e\o z=p\o\bar{\hat p_A}\o z
\end{align*}
In the following, we omit $e$ in the notation: $(A,\zeta_A,\sigma_A,\hat p_A)=(A,e;\zeta_A,\sigma_A,\hat p_A)$.
\par A morphism $f\colon(A,\zeta_A,\sigma_A,\hat p_A)\to(B,\zeta_B,\sigma_B,\hat p_B)$ of $\DO(\Split_R(\X,\TT))$ consists of a morphism $f\colon A\to B$ of $(\X,\TT)$ which commutes with the structural morphism of the differential objects and with their restriction idempotents. The tangent bundle functor sends each $(A,\zeta_A,\sigma_A,\hat p_A)$ to $(\T A,\T\zeta_A,\T\sigma_A,\T\hat p_A\o c_A)$.
\par Now, consider a tangent (non-necessarily split) restriction category $(\X,\TT)$ and define $\DO(\X,\TT)$ to be the full subcategory of $\DO(\Split_R(\X,\TT))$ spanned by the objects $(A,\zeta_A,\sigma_A,\hat p_A)$ where $\zeta_A$, $\sigma_A$, and $\hat p_A$ are total maps in $(\X,\TT)$, which is precisely when $\hat p_A$ is total.

\begin{lemma}
\label{lemma:differential-objects-tangent-restriction-categories}
Let $(\X,\TT)$ be a tangent restriction category. The subcategory $\DO(\X,\TT)$ of $\DO(\Split_R(\X,\TT))$ spanned by the objects $(A,\zeta_A,\sigma_A,\hat p_A)$ where $\hat p_A$ is total in $(\X,\TT)$ is a tangent restriction category.
\end{lemma}
\begin{proof}
We leave it to the reader to complete this tedious but straightforward proof.
\end{proof}

We can finally prove the main theorem of this section.

\begin{theorem}
\label{theorem:differential-objects-tangent-restriction-categories}
Consider the pullback-extension context $(\RestrCat,\TngRestrCat;\Inc,\Split_R;\eta)$ of tangent restriction categories. The $2$-category $\TngRestrCat$ admits the extended construction of differential objects with respect to this pullback-extension context. Moreover, the object of differential objects of a tangent restriction category is the tangent restriction category $\DO(\X,\TT)$.
\end{theorem}
\begin{proof}
We need to compute the $2$-pullback of Equation~\eqref{equation:pullback-extension-differential-objects}, that is, the $2$-pullback of the forgetful functor $\Inc(\U)\colon\Inc(\DO(\Split_R(\X,\TT)))\to\Inc(\Split_R(\X,\TT))$ along $\eta_{(\X,\TT)}\colon(\X,\TT)\to\Inc(\Split_R(\X,\TT))$ for each tangent restriction category $(\X,\TT)$. However, using a similar argument to~\cite[Theorem~5.11]{lanfranchi:tangentads-II}, we can show that the objects of such a pullback must be differential objects $(A,\zeta_A,\sigma_A,\hat p_A)\in\DO(\Split_R(\X,\TT))$ whose idempotent $e=p\o\bar{\hat p_A}\o z$ is trivial, which is equivalent to $\hat p_A$ be total in $(\X,\TT)$. Thus, by the universal property of the pullback, $\DO(\X,\TT)$ must be the largest subcategory of $\DO(\Split_R(\X,\TT))$ spanned by those differential objects.
\end{proof}

%__________________________________________________________________________
%__________________________________________________________________________

\section{The formal theory of differential bundles}
\label{section:differential-bundles}
In every tangent category, one can define a suitable class of bundles, called (display) differential bundles, which generalize vector bundles in differential geometry. This section aims to formalize this notion in the context of tangentads.

%__________________________________________________________________________
\subsection{Differential bundles in tangent category theory}
\label{subsection:tangent-category-differential-bundles}
For starters, recall the notion of tangent display maps, first introduced in~\cite{cruttwell:tangent-display-maps}.

\begin{definition}[{\cite[Definition~2.5]{cruttwell:tangent-display-maps}}]
\label{definition:tangent-display-map}
A \textbf{tangent display map} in a tangent category $(\X,\TT)$ is a morphism $q\colon E\to M$ such that, for every $n\geq0$ and every morphism $f\colon M'\to\T^nM$, the pullback of $\T^nq\colon\T^nE\to\T^nM$ along $f$
\begin{equation*}
% https://q.uiver.app/#q=WzAsNCxbMCwwLCJFJyJdLFsxLDAsIlxcVF5uRSJdLFsxLDEsIlxcVF5uTSJdLFswLDEsIk0nIl0sWzMsMiwiZiIsMl0sWzEsMiwiXFxUXm5xIl0sWzAsM10sWzAsMV0sWzAsMiwiIiwxLHsic3R5bGUiOnsibmFtZSI6ImNvcm5lciJ9fV1d
\begin{tikzcd}
{E'} & {\T^nE} \\
{M'} & {\T^nM}
\arrow[from=1-1, to=1-2]
\arrow[from=1-1, to=2-1]
\arrow["\lrcorner"{anchor=center, pos=0.125}, draw=none, from=1-1, to=2-2]
\arrow["{\T^nq}", from=1-2, to=2-2]
\arrow["f"', from=2-1, to=2-2]
\end{tikzcd}
\end{equation*}
exists and is preserved by all iterates $\T^m$ of the tangent bundle functor.
\end{definition}

Let us also introduce tangent morphisms which preserve tangent display maps.

\begin{definition}
\label{definition:display-preserving}
A \textbf{display tangent morphism} is a strong tangent morphism $(F,\alpha)\colon(\X,\TT)\to(\X',\TT')$ whose underlying functor $F\colon\X\to\X'$ sends each tangent display map $q\colon E\to M$ of $(\X,\TT)$ to a tangent display map $Fq\colon FE\to FM$ and that, for each morphism $f\colon N\to M$ preserves the tangent pullback of $q$ along $f$.
\end{definition}

Tangent categories, display tangent morphisms, and tangent natural transformations form a $2$-category denoted by $\TngCat_\Dsply$.

\begin{definition}[{\cite[Definition~2.3]{cockett:differential-bundles}}]
\label{definition:differential-bundles}
A \textbf{display differential bundle} in a tangent category $(\X,\TT)$ is a tuple $\q\=(q,z_q,s_q,l_q)$ consisting of:
\begin{description}
\item[Projection] A tangent display map $q\colon E\to M$ of $\X$;

\item[Zero morphism] A morphism $z_q\colon M\to E$;

\item[Sum morphism] A morphism $s_q\colon E\times_ME\to E$;
\end{description}
for which the triple $(q\colon E\to M,z_q,s_q)$ is a commutative monoid in the slice category $\X/M$ over $M$ with respect to the pullback of two maps $q\colon E\to M$ and $q'\colon E'\to M$. A differential bundle also comprises:
\begin{description}
\item[Vertical lift] A morphism $l_q\colon E\to\T E$;
\end{description}
which satisfies the following conditions:
\begin{enumerate}
\item Additivity 1. The morphism $(z,l_q)\colon(q,z_q,s_q)\to(\T q,\T z_q,\T s_q)$ is a morphism of additive bundles. Concretely, this corresponds to the commutativity of the following diagrams:
\begin{equation*}
% https://q.uiver.app/#q=WzAsNCxbMCwwLCJFIl0sWzAsMSwiTSJdLFsxLDAsIlxcVCBFIl0sWzEsMSwiXFxUIE0iXSxbMiwzLCJcXFQgcSJdLFswLDEsInEiLDJdLFsxLDMsInoiLDJdLFswLDIsImxfcSJdXQ==
\begin{tikzcd}
E & {\T E} \\
M & {\T M}
\arrow["{l_q}", from=1-1, to=1-2]
\arrow["q"', from=1-1, to=2-1]
\arrow["{\T q}", from=1-2, to=2-2]
\arrow["z"', from=2-1, to=2-2]
\end{tikzcd}\hfill\quad
% https://q.uiver.app/#q=WzAsNCxbMCwwLCJFIl0sWzAsMSwiTSJdLFsxLDAsIlxcVCBFIl0sWzEsMSwiXFxUIE0iXSxbMywyLCJcXFQgel9xIiwyXSxbMSwwLCJ6X3EiXSxbMSwzLCJ6IiwyXSxbMCwyLCJsX3EiXV0=
\begin{tikzcd}
E & {\T E} \\
M & {\T M}
\arrow["{l_q}", from=1-1, to=1-2]
\arrow["{z_q}", from=2-1, to=1-1]
\arrow["z"', from=2-1, to=2-2]
\arrow["{\T z_q}"', from=2-2, to=1-2]
\end{tikzcd}\hfill\quad
% https://q.uiver.app/#q=WzAsNCxbMCwxLCJFIl0sWzIsMSwiXFxUIEUiXSxbMCwwLCJFXzIiXSxbMiwwLCJcXFQgRV8yIl0sWzAsMSwibF9xIiwyXSxbMiwwLCJzX3EiLDJdLFszLDEsIlxcVCBzX3EiXSxbMiwzLCJcXDxcXHBpXzFsX3EsXFxwaV8yIGxfcVxcPiJdXQ==
\begin{tikzcd}
{E_2} && {\T E_2} \\
E && {\T E}
\arrow["{\<\pi_1l_q,\pi_2 l_q\>}", from=1-1, to=1-3]
\arrow["{s_q}"', from=1-1, to=2-1]
\arrow["{\T s_q}", from=1-3, to=2-3]
\arrow["{l_q}"', from=2-1, to=2-3]
\end{tikzcd}
\end{equation*}

\item Additivity 2. The morphism $(z_q,l)\colon(q,z_q,s_q)\to(p,z,s)$ is morphism of additive bundles. Concretely, this corresponds to the commutativity of the following diagrams:
\begin{equation*}
% https://q.uiver.app/#q=WzAsNCxbMCwwLCJFIl0sWzEsMCwiXFxUIEUiXSxbMSwxLCJFIl0sWzAsMSwiTSJdLFswLDEsImxfcSJdLFsxLDIsInAiXSxbMCwzLCJxIiwyXSxbMywyLCJ6X3EiLDJdXQ==
\begin{tikzcd}
E & {\T E} \\
M & E
\arrow["{l_q}", from=1-1, to=1-2]
\arrow["q"', from=1-1, to=2-1]
\arrow["p", from=1-2, to=2-2]
\arrow["{z_q}"', from=2-1, to=2-2]
\end{tikzcd}\hfill\quad
% https://q.uiver.app/#q=WzAsNCxbMCwwLCJFIl0sWzEsMCwiXFxUIEUiXSxbMSwxLCJFIl0sWzAsMSwiTSJdLFswLDEsImxfcSJdLFsyLDEsInoiLDJdLFszLDAsInpfcSJdLFszLDIsInpfcSIsMl1d
\begin{tikzcd}
E & {\T E} \\
M & E
\arrow["{l_q}", from=1-1, to=1-2]
\arrow["{z_q}", from=2-1, to=1-1]
\arrow["{z_q}"', from=2-1, to=2-2]
\arrow["z"', from=2-2, to=1-2]
\end{tikzcd}\hfill\quad
% https://q.uiver.app/#q=WzAsNCxbMCwxLCJFIl0sWzIsMSwiXFxUIEUiXSxbMCwwLCJFXzIiXSxbMiwwLCJcXFRfMkUiXSxbMCwxLCJsX3EiLDJdLFsyLDAsInNfcSIsMl0sWzMsMSwicyJdLFsyLDMsIlxcPFxccGlfMWxfcSxcXHBpXzIgbF9xXFw+Il1d
\begin{tikzcd}
{E_2} && {\T_2E} \\
E && {\T E}
\arrow["{\<\pi_1l_q,\pi_2 l_q\>}", from=1-1, to=1-3]
\arrow["{s_q}"', from=1-1, to=2-1]
\arrow["s", from=1-3, to=2-3]
\arrow["{l_q}"', from=2-1, to=2-3]
\end{tikzcd}
\end{equation*}

\item Linearity. The following diagram commutes:
\begin{equation*}
% https://q.uiver.app/#q=WzAsNCxbMCwwLCJFIl0sWzEsMCwiXFxUIEUiXSxbMSwxLCJcXFReMkUiXSxbMCwxLCJcXFQgRSJdLFswLDEsImxfcSJdLFsxLDIsIlxcVCBsX3EiXSxbMCwzLCJsX3EiLDJdLFszLDIsImwiLDJdXQ==
\begin{tikzcd}
E & {\T E} \\
{\T E} & {\T^2E}
\arrow["{l_q}", from=1-1, to=1-2]
\arrow["{l_q}"', from=1-1, to=2-1]
\arrow["{\T l_q}", from=1-2, to=2-2]
\arrow["l"', from=2-1, to=2-2]
\end{tikzcd}
\end{equation*}

\item Universality. The following is a pullback diagram:
\begin{equation*}
% https://q.uiver.app/#q=WzAsNCxbMCwwLCJFXzIiXSxbMiwwLCJcXFQgRSJdLFsyLDEsIlxcVCBNIl0sWzAsMSwiTSJdLFswLDEsIlxcPFxccGlfMWxfcSxcXHBpXzJ6XFw+Il0sWzEsMiwiXFxUIHEiXSxbMCwzLCJcXHBpXzFxIiwyXSxbMywyLCJ6IiwyXSxbMCwyLCIiLDEseyJzdHlsZSI6eyJuYW1lIjoiY29ybmVyIn19XV0=
\begin{tikzcd}
{E_2} && {\T E} \\
M && {\T M}
\arrow["{\<\pi_1l_q,\pi_2z\>}", from=1-1, to=1-3]
\arrow["{\pi_1q}"', from=1-1, to=2-1]
\arrow["\lrcorner"{anchor=center, pos=0.125}, draw=none, from=1-1, to=2-3]
\arrow["{\T q}", from=1-3, to=2-3]
\arrow["z"', from=2-1, to=2-3]
\end{tikzcd}
\end{equation*}
\end{enumerate}
\end{definition}

\begin{remark}
\label{remark:display-differential-bundles}
In the literature, there are a few variations on the definition of a differential bundle. In particular, the original definition due to Cockett and Cruttwell~\cite{cockett:differential-bundles} did not assume the projection to be display (tangent display maps were introduced by Cruttwell and by the author only more recently in~\cite{cruttwell:tangent-display-maps}). Instead, the existence of only some pullbacks was required. As shown in~\cite{cruttwell:tangent-display-maps}, requiring the projection to be display is, however, a very natural and desirable property, especially when dealing with connections. In the following, we refer to differential bundles as display differential bundles.
\end{remark}

Differential bundles generalize vector bundles in the context of tangent categories~\cite{macadam:vector-bundles}. In the following, we denote by $\q\colon E\to M$ a differential bundle with projection, zero morphism, sum morphism, and vertical lift the morphisms $q\colon E\to M$, $z_q\colon M\to E$, $s_q\colon E_2\to E$, and $l_q\colon E\to\T E$, respectively.
\par For each tangent category $(\X,\TT)$ there is a tangent category denoted by $\DB(\X,\TT)$ whose objects are differential bundles $\q\colon E\to M$ of $(\X,\TT)$ and morphisms are \textbf{linear} morphisms of differential bundles, that is, a pair $(f,g)\colon\q\to\q'$ of morphisms of $(\X,\TT)$ that commutes with the projections, $f\o q=q'\o g$, and the vertical lifts, $\T g\o l_q=l_{q'}\o g$. The tangent bundle functor sends a differential bundle $\q\colon E\to M$ to the differential bundle $\T^\DB\q\=(\T q,\T z_q,\T s_q,l_{\T q})$, where we identified $\T E_2$ with $\T_2E$, and where the vertical lift is defined as follows:
\begin{align*}
&l_{\T q}\colon\T E\xrightarrow{\T l_q}\T^2E\xrightarrow{c}\T^2E
\end{align*}
Finally, the structural natural transformations of $\DB(\X,\TT)$ are the same as the ones of $(\X,\TT)$.

\begin{lemma}
\label{lemma:DB-functoriality}
There is a $2$-endofunctor
\begin{align*}
&\DB\colon\TngCat_\Dsply\to\TngCat_\cong
\end{align*}
on the $2$-category of tangent categories and display tangent morphisms, which sends a tangent category $(\X,\TT)$ to the tangent category $\DB(\X,\TT)$ of display differential bundles of $(\X,\TT)$ and that sends a display tangent morphism $(F,\alpha)\colon(\X',\TT')\to(\X,\TT)$ to the display tangent morphism which sends a display differential bundle $\q\colon E\to M$ to the display differential bundle $(Fq,Fz_q,Fs_q,l_{Fq})$, whose vertical lift is defined as follows:
\begin{align*}
&l_{Fq}\colon FE\xrightarrow{Fl_q}F\T E\xrightarrow{\alpha}\T'FE
\end{align*}
\end{lemma}
\begin{proof}
\cite[Proposition~4.22]{cockett:differential-bundles} establishes that a Cartesian tangent morphism~\cite[Definition~4.16]{cockett:differential-bundles} $(F,\alpha)$ lifts to a functor $\DB(F)\colon\DB(\X,\TT)\to\DB(\X',\TT')$. However, display tangent morphisms are Cartesian; thus, every display tangent morphism $(F,\alpha)$ lifts to a functor $\DB(F,\alpha)$. It is not hard to see that the strong distributive law $\alpha$ lifts to a strong distributive law $\DB(\alpha)$ for the functor $\DB(F)$. Finally, $\DB$ sends tangent natural transformations to tangent natural transformations in an obvious way.
\end{proof}

%__________________________________________________________________________
\subsection{The universal property of differential bundles}
\label{subsection:universal-property-differential-bundles}
This section aims to characterize the correct universal property of the tangent category $\DB(\X,\TT)$ of differential bundles and linear morphisms of a tangent category $(\X,\TT)$. For starters, consider the functors
\begin{align*}
&\Tot\colon\DB(\X,\TT)\to(\X,\TT)
&\Base\colon\DB(\X,\TT)\to(\X,\TT)
\end{align*}
which send a differential bundle $\q\colon E\to M$ to its total object $E$ and its base object $M$, respectively.

The next step is to show that the functors $\Base$ and $\Tot$ come equipped with the structure of a display differential bundle in the Hom-tangent category $[\DB(\X,\TT)\|\X,\TT]$. Define:
\begin{align*}
&\Univ q_{\q\colon E\to M}\colon\Tot(\q)=E\xrightarrow{q}M=\Base(\q)
\end{align*}

\begin{lemma}
\label{lemma:display-universal-differential-bundle}
The tangent natural transformation $\Univ q\colon\Tot\Rightarrow\Base$ is a tangent display map.
\end{lemma}
\begin{proof}
By using that for every $\q\colon E\to M$ in $\DB(\X,\TT)$, $\T^nq$ admits all tangent pullbacks, it is easy to compute the pointwise pullback of ${\bar\T}^n\Univ q$ along $\varphi\colon F\to{\bar\T}^n\Base$ as the unique functor $G\colon\DB(\X,\TT)\to\X$ such that, for every $\q\colon E\to M$, the following is a pullback diagram:
\begin{equation*}
% https://q.uiver.app/#q=WzAsNCxbMCwxLCJGXFxxIl0sWzEsMSwie1xcYmFyXFxUfV5uTSJdLFsxLDAsIntcXGJhclxcVH1ebkUiXSxbMCwwLCJHXFxxIl0sWzIsMSwie1xcYmFyXFxUfV5ucSJdLFswLDEsIlxcdmFycGhpX1xccSIsMl0sWzMsMCwiXFx4aV9cXHEiLDJdLFszLDIsIlxccHNpX1xccSJdLFszLDEsIiIsMSx7InN0eWxlIjp7Im5hbWUiOiJjb3JuZXIifX1dXQ==
\begin{tikzcd}
{G\q} & {{\bar\T}^n\Tot(\q)} \\
{F\q} & {{\bar\T}^n\Base(\q)}
\arrow["{\psi_\q}", from=1-1, to=1-2]
\arrow["{\xi_\q}"', from=1-1, to=2-1]
\arrow["\lrcorner"{anchor=center, pos=0.125}, draw=none, from=1-1, to=2-2]
\arrow["{{\bar\T}^n\Univ q_\q}", from=1-2, to=2-2]
\arrow["{\varphi_\q}"', from=2-1, to=2-2]
\end{tikzcd}
\end{equation*}
Furthermore, using that the distributive law $\alpha$ of $(F,\alpha)$ is a morphism in $(\X,\TT)$, we obtain a unique morphism $\beta\colon G\T^\DB\q\to\T G\q$ as follows:
\begin{equation*}
% https://q.uiver.app/#q=WzAsOCxbMSwyLCJcXFQgRlxccSJdLFsyLDIsIntcXGJhclxcVH1ee24rMX1NIl0sWzIsMSwie1xcYmFyXFxUfV57bisxfUUiXSxbMSwxLCJcXFQgR1xccSJdLFswLDAsIkdcXFReXFxEQlxccSJdLFswLDMsIkZcXFReXFxEQlxccSJdLFszLDMsIntcXGJhclxcVH1ee24rMX1NIl0sWzMsMCwie1xcYmFyXFxUfV57bisxfUUiXSxbMiwxLCJ7XFxiYXJcXFR9XntuKzF9cSJdLFswLDEsIlxcVFxcdmFycGhpX1xccSIsMl0sWzMsMCwiXFxUXFx4aV9cXHEiLDJdLFszLDIsIlxcVFxccHNpX1xccSJdLFszLDEsIiIsMSx7InN0eWxlIjp7Im5hbWUiOiJjb3JuZXIifX1dLFs0LDUsIlxceGlfe1xcVF5cXERCXFxxfSIsMl0sWzcsNiwie1xcYmFyXFxUfV57bisxfXEiXSxbNSw2LCJcXHZhcnBoaV97XFxUXlxcREJcXHF9IiwyXSxbNCw3LCJcXHBzaV97XFxUXlxcREJcXHF9Il0sWzUsMCwiXFxhbHBoYV9cXHEiXSxbNiwxLCIiLDEseyJsZXZlbCI6Miwic3R5bGUiOnsiaGVhZCI6eyJuYW1lIjoibm9uZSJ9fX1dLFs3LDIsIiIsMSx7ImxldmVsIjoyLCJzdHlsZSI6eyJoZWFkIjp7Im5hbWUiOiJub25lIn19fV0sWzQsMywiXFxiZXRhX1xccSIsMCx7InN0eWxlIjp7ImJvZHkiOnsibmFtZSI6ImRhc2hlZCJ9fX1dXQ==
\begin{tikzcd}
{G\T^\DB\q} &&& {{\bar\T}^{n+1}\Tot(\q)} \\
& {\T G\q} & {{\bar\T}^{n+1}\Tot(\q)} \\
& {\T F\q} & {{\bar\T}^{n+1}\Base(\q)} \\
{F\T^\DB\q} &&& {{\bar\T}^{n+1}\Base(\q)}
\arrow["{\psi_{\T^\DB\q}}", from=1-1, to=1-4]
\arrow["{\beta_\q}", dashed, from=1-1, to=2-2]
\arrow["{\xi_{\T^\DB\q}}"', from=1-1, to=4-1]
\arrow[equals, from=1-4, to=2-3]
\arrow["{{\bar\T}^{n+1}\Univ q_\q}", from=1-4, to=4-4]
\arrow["{\T\psi_\q}", from=2-2, to=2-3]
\arrow["{\T\xi_\q}"', from=2-2, to=3-2]
\arrow["\lrcorner"{anchor=center, pos=0.125}, draw=none, from=2-2, to=3-3]
\arrow["{{\bar\T}^{n+1}\Univ q_\q}", from=2-3, to=3-3]
\arrow["{\T\varphi_\q}"', from=3-2, to=3-3]
\arrow["{\alpha_\q}", from=4-1, to=3-2]
\arrow["{\varphi_{\T^\DB\q}}"', from=4-1, to=4-4]
\arrow[equals, from=4-4, to=3-3]
\end{tikzcd}
\end{equation*}
It is not hard to prove that $(G,\beta)\colon\DB(\X,\TT)\to(\X,\TT)$ becomes a lax tangent morphism which is the pointwise pullback of ${\bar\T}^nq$ along $\varphi$ in the Hom-tangent category $[\DB(\X,\TT)\|\X,\TT]$.
\end{proof}

Notice that each pullback of $\Univ q$ along any tangent natural transformation $\varphi$ not only exists but is also pointwise, that is, it is preserved by all functors $\CC[F\|\X]$ for any $1$-morphism $F\colon\X'\to\DB(\X,\TT)$ of $\CC$. Since this plays an important role in our discussion, let us formally introduce this concept.

\begin{definition}
\label{definition:pointwise-tangent-display-map}
Given two tangentads $(\X,\TT)$ and $(\X',\TT')$ of a $2$-category $\CC$, a \textbf{pointwise tangent display map} in the Hom-tangent category $[\X,\TT\|\X',\TT']$ consists of a tangent display map $q$ in this tangent category for which each pullback of $q$ along any morphism of $[\X,\TT\|\X',\TT']$ is pointwise, that is, preserved by all functors of type $\CC[F\|\X']$ for each $F\colon\X''\to\X'$. Similarly, we call a \textbf{pointwise display differential bundle} a differential bundle in the Hom-tangent category $[\X,\TT\|\X',\TT']$ whose underlying projection is a pointwise tangent display map.
\end{definition}

We want to show that the tangent display map $\Univ q\colon\Tot\to\Base$ in the Hom-tangent category $[\DB(\X,\TT)\|\X,\TT]$ comes equipped with the structure of a display differential bundle. For starters, let us consider a display differential bundle $\q\colon E\to M$ of $(\X,\TT)$ and define the following morphisms:
\begin{align*}
&\Univ{z_q}_{\q}\colon\Base(\q)=M\xrightarrow{z_q}E=\Tot(\q)\\
&\Univ{s_q}_{\q}\colon\Tot_2(\q)=E_2\xrightarrow{s_q}E=\Tot(\q)\\
&\Univ{l_q}_{\q}\colon\Tot(\q)=E\xrightarrow{l_q}\T E=\bar\T\Tot(\q)
\end{align*}
In particular, $\Univ{z_q}$, $\Univ{s_q}$, and $\Univ{l_q}$ are the tangent natural transformations which pick out the zero morphism $z_q$, the sum morphism $s_q$, and the vertical lift $l_q$ of each of the display differential bundles $\q$ of $(\X,\TT)$, respectively.

\begin{proposition}
\label{proposition:universality-differential-bundle}
The strict tangent morphisms $\Base\colon\DB(\X,\TT)\to(\X,\TT)$ and $\Tot\colon\DB(\X,\TT)\to(\X,\TT)$ equipped with $\Univ q$, $\Univ{z_q}$, $\Univ{s_q}$, and $\Univ{l_q}$ form a pointwise display differential bundle in the Hom-tangent category $[\DB(\X,\TT)\|\X,\TT]$.
\end{proposition}
\begin{proof}
In Lemma~\ref{lemma:display-universal-differential-bundle} we already proved that $\q\colon\Tot\to\Base$ is a tangent display map of the Hom-tangent category $[\DB(\X,\TT)\|\X,\TT]$. The axioms required for $\UnivQ\=(\Univ q,\Univ{z_q},\Univ{s_q},\Univ{l_q})$ to be a pointwise display differential bundle are a consequence of $\UnivQ_{\q}=(q,z_q,s_q,l_q)$ being a display differential bundle for each $\q\in\DB(\X,\TT)$.
\end{proof}

To characterize the correct universal property enjoyed by the differential bundle of Proposition~\ref{proposition:universality-differential-bundle}, let us first unwrap the definition of a pointwise display differential bundle in each Hom-tangent category $[\X',\TT'\|\X'',\TT'']$, where $(\X',\TT')$ and $(\X'',\TT'')$ are tangent categories:
\begin{description}
\item[Base object] The base object is a lax tangent morphism $\Base\colon(\X',\TT')\to(\X'',\TT'')$;

\item[Total object] The total object is a lax tangent morphism $\Tot\colon(\X',\TT')\to(\X'',\TT'')$;

\item[Projection] The projection consists of a pointwise tangent display map:
\begin{align*}
&q\colon\Tot\to\Base
\end{align*}

\item[Zero morphism] The zero morphism consists of a tangent natural transformation:
\begin{align*}
&z_q\colon\Base\to\Tot
\end{align*}

\item[Sum morphism] The sum morphism consists of a tangent natural transformation:
\begin{align*}
&s_q\colon\Tot_2\to\Tot
\end{align*}

\item[Vertical lift] The vertical lift consists of a tangent natural transformation:
\begin{align*}
&l_q\colon\Tot\to\bar\T\Tot
\end{align*}
\end{description}
satisfying the axioms of a pointwise display differential bundle. A morphism of pointwise tangent differential bundles in the Hom-tangent category $[\X,\TT\|\X',\TT']$ from a differential bundle $\q\colon(K,\theta)\to(G,\beta)$ to a differential bundle $\q'\colon(K',\theta')\to(G',\beta')$ consists of a pair $(\varphi,\psi)$ of tangent natural transformations
\begin{align*}
&\varphi\colon(G,\beta)\Rightarrow(G',\beta')\\
&\psi\colon(K,\theta)\Rightarrow(K',\theta')
\end{align*}
which commutes with the projections $q$ and $q'$ and such that $\psi$ commutes with the vertical lifts $l_q$ and $l_q'$.
\par For each tangent category $(\X',\TT')$, a pointwise display differential bundle $\q\colon(K,\theta)\to(G,\beta)$ in the Hom-tangent category $[\X'',\TT''\|\X''',\TT''']$ induces a functor
\begin{align*}
&\Gamma_{\q}\colon[\X',\TT'\|\X'',\TT'']\to\DB_\pt[\X',\TT'\|\X''',\TT''']
\end{align*}
which sends a lax tangent morphism $(H,\gamma)\colon(\X',\TT')\to(\X'',\TT'')$ to the pointwise display differential bundle
\begin{align*}
&\Gamma_{\q}(H,\gamma)\=((G,\beta)\o(H,\gamma),(K,\theta)\o(H,\gamma);q_{(H,\gamma)},{z_q}_{(H,\gamma)},{s_q}_{(H,\gamma)},{l_q}_{(H,\gamma)})
\end{align*}
in the Hom-tangent category $[\X',\TT'\|\X''',\TT''']$. Concretely, the base and the total objects of $\Gamma_{\q}(H,\gamma)$ are the lax tangent morphisms
\begin{align*}
&(\X',\TT')\xrightarrow{(H,\gamma)}(\X'',\TT'')\xrightarrow{(G,\beta)}(\X''',\TT''')\\
&(\X',\TT')\xrightarrow{(H,\gamma)}(\X'',\TT'')\xrightarrow{(K,\theta)}(\X''',\TT''')
\end{align*}
respectively, while the zero, the sum, and the vertical lift of $\Gamma_{\q}(H,\gamma)$ are the natural transformations:
\begin{equation*}
% https://q.uiver.app/#q=WzAsMyxbMCwwLCIoXFxYJyxcXFRUJykiXSxbMSwwLCIoXFxYJycsXFxUVCcnKSJdLFszLDAsIihcXFgnJycsXFxUVCcnJykiXSxbMCwxLCIoSCxcXGdhbW1hKSJdLFsxLDIsIihLLFxcdGhldGEpIiwyLHsiY3VydmUiOjN9XSxbMSwyLCIoRyxcXGJldGEpIiwwLHsiY3VydmUiOi0zfV0sWzUsNCwicSIsMCx7InNob3J0ZW4iOnsic291cmNlIjoyMCwidGFyZ2V0IjoyMH19XV0=
\begin{tikzcd}
{(\X',\TT')} & {(\X'',\TT'')} && {(\X''',\TT''')}
\arrow["{(H,\gamma)}", from=1-1, to=1-2]
\arrow[""{name=0, anchor=center, inner sep=0}, "{(K,\theta)}"', curve={height=18pt}, from=1-2, to=1-4]
\arrow[""{name=1, anchor=center, inner sep=0}, "{(G,\beta)}", curve={height=-18pt}, from=1-2, to=1-4]
\arrow["q", shorten <=5pt, shorten >=5pt, Rightarrow, from=1, to=0]
\end{tikzcd}
\end{equation*}
\begin{equation*}
% https://q.uiver.app/#q=WzAsMyxbMCwwLCIoXFxYJyxcXFRUJykiXSxbMSwwLCIoXFxYJycsXFxUVCcnKSJdLFszLDAsIihcXFgnJycsXFxUVCcnJykiXSxbMCwxLCIoSCxcXGdhbW1hKSJdLFsxLDIsIlxcQmFzZSIsMix7ImN1cnZlIjozfV0sWzEsMiwiXFxUb3QiLDAseyJjdXJ2ZSI6LTN9XSxbNCw1LCJ6X3EiLDIseyJzaG9ydGVuIjp7InNvdXJjZSI6MjAsInRhcmdldCI6MjB9fV1d
\begin{tikzcd}
{(\X',\TT')} & {(\X'',\TT'')} && {(\X''',\TT''')}
\arrow["{(H,\gamma)}", from=1-1, to=1-2]
\arrow[""{name=0, anchor=center, inner sep=0}, "{(G,\beta)}"', curve={height=18pt}, from=1-2, to=1-4]
\arrow[""{name=1, anchor=center, inner sep=0}, "{(K,\theta)}", curve={height=-18pt}, from=1-2, to=1-4]
\arrow["{z_q}"', shorten <=5pt, shorten >=5pt, Rightarrow, from=0, to=1]
\end{tikzcd}
\end{equation*}
\begin{equation*}
% https://q.uiver.app/#q=WzAsMyxbMCwwLCIoXFxYJyxcXFRUJykiXSxbMSwwLCIoXFxYJycsXFxUVCcnKSJdLFszLDAsIihcXFgnJycsXFxUVCcnJykiXSxbMCwxLCIoSCxcXGdhbW1hKSJdLFsxLDIsIlxcVG90IiwyLHsiY3VydmUiOjN9XSxbMSwyLCJcXFRvdF8yIiwwLHsiY3VydmUiOi0zfV0sWzUsNCwic19xIiwwLHsic2hvcnRlbiI6eyJzb3VyY2UiOjIwLCJ0YXJnZXQiOjIwfX1dXQ==
\begin{tikzcd}
{(\X',\TT')} & {(\X'',\TT'')} && {(\X''',\TT''')}
\arrow["{(H,\gamma)}", from=1-1, to=1-2]
\arrow[""{name=0, anchor=center, inner sep=0}, "{(K,\theta)}"', curve={height=18pt}, from=1-2, to=1-4]
\arrow[""{name=1, anchor=center, inner sep=0}, "{(K,\theta)_2}", curve={height=-18pt}, from=1-2, to=1-4]
\arrow["{s_q}", shorten <=5pt, shorten >=5pt, Rightarrow, from=1, to=0]
\end{tikzcd}
\end{equation*}
\begin{equation*}
% https://q.uiver.app/#q=WzAsMyxbMCwwLCIoXFxYJyxcXFRUJykiXSxbMSwwLCIoXFxYJycsXFxUVCcnKSJdLFszLDAsIihcXFgnJycsXFxUVCcnJykiXSxbMCwxLCIoSCxcXGdhbW1hKSJdLFsxLDIsIlxcVG90IiwyLHsiY3VydmUiOjN9XSxbMSwyLCJcXGJhclxcVFxcVG90IiwwLHsiY3VydmUiOi0zfV0sWzQsNSwibF9xIiwyLHsic2hvcnRlbiI6eyJzb3VyY2UiOjIwLCJ0YXJnZXQiOjIwfX1dXQ==
\begin{tikzcd}
{(\X',\TT')} & {(\X'',\TT'')} && {(\X''',\TT''')}
\arrow["{(H,\gamma)}", from=1-1, to=1-2]
\arrow[""{name=0, anchor=center, inner sep=0}, "{(K,\theta)}"', curve={height=18pt}, from=1-2, to=1-4]
\arrow[""{name=1, anchor=center, inner sep=0}, "{\bar\T(K,\theta)}", curve={height=-18pt}, from=1-2, to=1-4]
\arrow["{l_q}"', shorten <=5pt, shorten >=5pt, Rightarrow, from=0, to=1]
\end{tikzcd}
\end{equation*}
As observed for vector fields and differential bundles, this construction is functorial and extends to a tangent morphism.

\begin{lemma}
\label{lemma:induced-lax-tangent-morphism-from-differential-bundles}
A differential bundle $\q\colon(K,\theta)\to(G,\beta)$ in the Hom-tangent category $[\X'',\TT''\|\X''',\TT''']$ induces a lax tangent morphism
\begin{align*}
&\Gamma_{\q}\colon[\X',\TT'\|\X'',\TT'']\to\DB_\pt[\X',\TT'\|\X''',\TT''']
\end{align*}
for each tangentad $(\X',\TT')$. Furthermore, $\Gamma_\q$ is natural in $(\X',\TT')$ and strong (strict) when both $(G,\beta)$ and $(K,\theta)$ are strong (strict).
\end{lemma}
\begin{proof}
The proof of this lemma is fairly similar to that of Lemma~\ref{lemma:induced-lax-tangent-morphism-from-differential-objects}. We leave it to the reader to complete the details.
\end{proof}

In particular, by Lemma~\ref{lemma:induced-lax-tangent-morphism-from-differential-bundles}, the differential bundle $\UnivQ\colon\Tot\to\Base$ of Proposition~\ref{proposition:universality-differential-bundle} induces a strict tangent natural transformation
\begin{align*}
&\Gamma_{\UnivQ}\colon[\X',\TT'\|\DB(\X,\TT)]\to\DB_\pt[\X',\TT'\|\X,\TT]
\end{align*}
natural in $(\X',\TT')$.
\par We want to prove that $\Gamma_{\UnivQ}$ is invertible. Consider another tangent category $(\X',\TT')$ together with a pointwise display differential bundle $\q\colon(K,\theta)\to(G,\beta)$ in the Hom-tangent category $[\X',\TT'\|\X,\TT]$. For every $A\in\X'$, the tuple
\begin{align*}
&(q_A\colon KA\to GA,{z_q}_A\colon GA\to KA,{s_q}_A\colon K_2A\to KA,{l_q}_A\colon KA\to\T KA)
\end{align*}
is a display differential bundle in $(\X,\TT)$. Therefore, we can define a functor
\begin{align*}
\Lambda[G,K;q,z_q,s_q,l_q]\colon(\X',\TT')\to\DB(\X,\TT)
\end{align*}
which sends an object $A$ of $\X'$ to the differential bundle $\q_A\colon KA\to GA$, and a morphism $f\colon A\to B$ to $(Gf,Kf)\colon\q_A\to\q_B$. By the naturality of $q$, $z_q$, $s_q$, and $l_q$,
\begin{align*}
&(Gf,Kf)\colon(q_A,{z_q}_A,{s_q}_A,{l_q}_A)\to(q_B,{z_q}_B,{s_q}_B,{l_q}_B)
\end{align*}
becomes a linear morphism of differential bundles of $(\X,\TT)$. Furthermore, $\Lambda[\q]$ comes with a distributive law
\begin{align*}
&\Lambda[\beta,\theta]\colon\Lambda[G,K;q,z_q,s_q,l_q](\T'A)=(q_{\T'A},{z_q}_{\T'A},{s_q}_{\T'A},{l_q}_{\T'A})\xrightarrow{(\beta_A,\theta_A)}(\T q_A,\T{z_q}_A,\T{s_q}_A,\T{l_q}_A)=\\
&\qquad=\T^\DB(\Lambda[G,K;q,z_q,s_q,l_q](A))
\end{align*}
which is a morphism of differential bundles since each of the structural morphisms $q$, $z_q$, $s_q$, and $l_q$ are tangent natural transformations and thus compatible with the distributive laws $\beta$ and $\theta$.

\begin{lemma}
\label{lemma:universality-differential-bundles}
Given two lax tangent morphisms $(G,\beta),(K,\theta)\colon(\X',\TT')\to(\X,\TT)$ and a pointwise display differential bundle structure $\q\colon(K,\theta)\to(G,\beta)$ in the Hom-tangent category $[\X',\TT'\|\X,\TT]$, the functor $\Lambda[G,K;q,z_q,s_q,l_q]$ together with the distributive law $\Lambda[\beta,\kappa]$ defines a lax tangent morphism:
\begin{align*}
&\Lambda[\q]\colon(\Lambda[G,K;q,z_q,s_q,l_q],\Lambda[\beta,\theta])\colon(\X',\TT')\to\DB(\X,\TT)
\end{align*}
\end{lemma}
\begin{proof}
In the previous discussion, we defined $\Lambda[\q]$ as the lax tangent morphism which picks out the differential bundle $(q_A\colon KA\to GA,{z_q}_A,{s_q}_A,{l_q}_A)$ for each $A\in\X'$ and whose distributive law is induced by $\beta$ and $\theta$. In particular, the naturality of the structure morphisms $q$, $z_q$ $s_q$, and $l_q$ makes the functor $\Lambda[G,K;q,z_q,s_q,l_q]$ well-defined, the compatibility of the structure morphisms, $\beta$ and $\theta$ makes $\Lambda[\beta,\theta]$ into a morphism of differential bundles, and the compatibility of $\beta$ and $\theta$ with the tangent structures makes $\Lambda[\q]$ into a lax tangent morphism.
\end{proof}

We can now prove the universal property of differential bundles, which is the main result of this section.

\begin{theorem}
\label{theorem:universality-differential-bundles}
The pointwise display differential bundle $\UnivQ\colon\Tot\to\Base$ of Proposition~\ref{proposition:universality-differential-bundle} is universal. Concretely, the induced strict tangent natural transformation
\begin{align*}
&\Gamma_{\UnivQ}\colon[\X',\TT'\|\DB(\X,\TT)]\to\DB_\pt[\X',\TT'\|\X,\TT]
\end{align*}
makes the functor
\begin{align*}
&\TngCat^\op\xrightarrow{[-\|\X,\TT]}\TngCat^\op\xrightarrow{\DB_\pt^\op}\TngCat^\op
\end{align*}
which sends a tangent category $(\X',\TT')$ to the tangent category $\DB_\pt[\X',\TT'\|\X,\TT]$, into a corepresentable functor. In particular, $(\Gamma_{\UnivQ}$ is invertible.
\end{theorem}
\begin{proof}
The proof of this theorem is fairly similar to that of Theorem~\ref{theorem:universality-differential-objects}.
\end{proof}

Theorem~\ref{theorem:universality-differential-bundles} establishes the correct universal property of the differential bundles construction. Thanks to this result, we can finally introduce the notion of differential bundles in the formal context of tangentads.

\begin{definition}
\label{definition:construction-differential-bundles}
A tangentad $(\X,\TT)$ in a $2$-category $\CC$ \textbf{admits the construction of display differential bundles} if there exists a tangentad $\DB(\X,\TT)$ of $\CC$ together with two strict tangent morphisms $\Base,\Tot\colon\DB(\X,\TT)\to(\X,\TT)$ and a pointwise display differential bundle $\UnivQ\colon\Tot\to\Base$ in the Hom-tangent category $[\DB(\X,\TT)\|\X,\TT]$ such that the induced tangent natural transformation $\Gamma_{\UnivQ}$ of Lemma~\ref{lemma:induced-lax-tangent-morphism-from-differential-bundles} is invertible. The pointwise display differential bundle $\UnivQ\colon\Tot\to\Base$ is called the \textbf{universal } (\textbf{pointwise}) \textbf{display differential bundle} of $(\X,\TT)$ and $\DB(\X,\TT)$ is called the \textbf{tangentad of display differential bundles} of $(\X,\TT)$.
\end{definition}

\begin{definition}
\label{definition:construction-differential-bundles-2-category}
A $2$-category $\CC$ \textbf{admits the construction of display differential bundles} provided that every tangentad of $\CC$ admits the construction of display differential bundles.
\end{definition}

We can now rephrase Theorem~\ref{theorem:universality-differential-bundles} as follows.

\begin{corollary}
\label{corollary:universality-differential-bundles}
The $2$-category $\Cat$ of categories admits the construction of display differential bundles and the tangentad of display differential bundles of a tangentad $(\X,\TT)$ of $\Cat$ is the tangent category $\DB(\X,\TT)$ of display differential bundles of $(\X,\TT)$.
\end{corollary}

To ensure that Definitions~\ref{definition:construction-differential-bundles} and~\ref{definition:construction-differential-bundles-2-category} are well-posed, one requires the tangentad of display differential bundles of a given tangentad to be unique. The next proposition establishes that such a construction is defined uniquely up to a unique isomorphism.

\begin{proposition}
\label{proposition:uniqueness-differential-bundles}
If a tangentad $(\X,\TT)$ admits the construction of display differential bundles, the tangentad of display differential bundles $\DB(\X,\TT)$ of $(\X,\TT)$ is unique up to a unique isomorphism which extends to an isomorphism of the corresponding universal display differential bundles of $(\X,\TT)$.
\end{proposition}
\begin{proof}
The proof of this proposition is fairly similar to that of Proposition~\ref{proposition:uniqueness-differential-objects}. We leave it to the reader to complete the details.
\end{proof}

%__________________________________________________________________________
\subsection{The formal structures of differential bundles}
\label{subsection:structures-differential-bundles}
As proved in Lemma~\ref{lemma:DB-functoriality}, there is a $2$-functor
\begin{align*}
&\DB\colon\TngCat_\Dsply\to\TngCat_\cong
\end{align*}
which sends a tangent category $(\X,\TT)$ to the tangent category $\DB(\x,\TT)$ of display differential bundles of $(\X,\TT)$.
\par In this section, we show that the construction of display differential bundles is $2$-functorial in the $2$-category of tangentads of $\CC$.

%__________________________________________________________________________
\subsubsection*{The functoriality of the construction of differential bundles}
\label{subsubsection:functoriality-differential-bundles}
Consider a strong tangent morphism $(F,\alpha)\colon(\X,\TT)\to(\X',\TT')$ between two tangentads of $\CC$. It is not hard to see that the induced functor
\begin{align*}
&[\DB(\X,\TT)\|\X,\TT]\xrightarrow{[\DB(\X,\TT)\|F,\alpha]}[\DB(\X,\TT)\|\X',\TT']
\end{align*}
is a strong tangent morphism. In order to apply the functor $\DB_\pt$ to $[\DB(\X,\TT)\|F,\alpha]$, $[\DB(\X,\TT)\|F,\alpha]$ must preserve pointwise tangent display maps. Let us introduce the classes of tangent morphisms between tangentads which satisfy this condition.

\begin{definition}
\label{definition:display-tangent-morphisms-formal}
A \textbf{display tangent morphism} from a tangentad $(\X,\TT)$ to another tangentad $(\X',\TT')$ consists of a strong tangent morphism $(F,\alpha)\colon(\X,\TT)\to(\X',\TT')$ such that, for every tangentad $(\X'',\TT'')$, the induced strong tangent morphism
\begin{align*}
&[\X'',\TT''\|\X,\TT]\xrightarrow{[\X'',\TT''\|F,\alpha]}[\X'',\TT''\|\X',\TT']
\end{align*}
is a display tangent morphism.
\end{definition}

Tangentads, display tangent morphisms, and tangent $2$-morphisms of a $2$-category $\CC$ form a $2$-category denoted by $\Tng_\Dsply(\CC)$.

\begin{proposition}
\label{proposition:DB-functoriality}
There is a $2$-functor
\begin{align*}
&\DB\colon\Tng_\Dsply(\CC)\to\Tng_\cong(\CC)
\end{align*}
which sends a tangentad $(\X,\TT)$ of $\CC$ to the tangentad of display differential bundles $\DB(\X,\TT)$ of $(\X,\TT)$. In particular, for each display tangent morphism $(F,\alpha)\colon(\X,\TT)\to(\X',\TT')$, the tangent morphism $\DB(F,\alpha)\colon\DB(\X,\TT)\to\DB(\X',\TT')$ is strong.
\end{proposition}
\begin{proof}
The proof is fairly similar to the proofs of Proposition~\ref{proposition:DO-functoriality}. Thus, we leave it to the reader to spell out the details.
\end{proof}

%__________________________________________________________________________
\subsection{Differential objects vs differential bundles}
\label{subsection:differential-bundles-vs-bundles}
\cite[Proposition~3.4]{cockett:differential-bundles} establishes an equivalence between differential objects in a Cartesian tangent category and differential bundles over the terminal object. In this section, we extend this result to the formal theory of tangentads.
\par For starters, consider a Cartesian tangentad $(\X,\TT)$ of a $2$-category $\CC$ with finite $2$-products and suppose $(\X,\TT)$ admits the construction of display differential bundles. Let $\Univ q\colon\Tot\to\Base$ denote the universal pointwise display differential bundle of $(\X,\TT)$ and $\DB(\X,\TT)$ the tangentad of display differential bundles of $(\X,\TT)$. To \textit{restrict} $\DB(\X,\TT)$ to those differential bundles whose base object is the terminal object, we need to consider the $2$-pullback of $\Base\colon\DB(\X,\TT)\to(\X,\TT)$ along $\*\colon\1\to(\X,\TT)$ in $\CC$:
\begin{equation}
\label{equation:restriction-differential-bundles-to-terminal}
% https://q.uiver.app/#q=WzAsNCxbMCwwLCJcXERCfF9cXDEoXFxYLFxcVFQpIl0sWzEsMCwiXFxEQihcXFgsXFxUVCkiXSxbMSwxLCIoXFxYLFxcVFQpIl0sWzAsMSwiXFwxIl0sWzMsMiwiXFwqIiwyXSxbMCwzXSxbMSwyLCJcXEJhc2UiXSxbMCwxLCJcXEluYyJdLFswLDIsIiIsMSx7InN0eWxlIjp7Im5hbWUiOiJjb3JuZXIifX1dXQ==
\begin{tikzcd}
{\DB|_\1(\X,\TT)} & {\DB(\X,\TT)} \\
\1 & {(\X,\TT)}
\arrow["\Inc", from=1-1, to=1-2]
\arrow[from=1-1, to=2-1]
\arrow["\lrcorner"{anchor=center, pos=0.125}, draw=none, from=1-1, to=2-2]
\arrow["\Base", from=1-2, to=2-2]
\arrow["{\*}"', from=2-1, to=2-2]
\end{tikzcd}
\end{equation}
Notice that, by~\cite[Proposition~4.7]{lanfranchi:tangentads-II}, the diagram of Equation~\eqref{equation:restriction-differential-bundles-to-terminal}, becomes a $2$-pullback in $\Tng(\CC)$.
Let $\Inc\colon\DB|_\1(\X,\TT)\to\DB(\X,\TT)$ denote the inclusion of the $2$-pullback $\DB|_\1(\X,\TT)$ of $\Base$ along $\*$ to $\DB(\X,\TT)$. We want to prove that $\DB|_\1(\X,\TT)$ carries the structure of the universal differential object of $(\X,\TT)$. Firstly, consider the tangent morphism
\begin{align*}
&[\Inc\|\X,\TT]\colon[\DB(\X,\TT)\|\X,\TT]\to[\DB|_\1(\X,\TT)\|\X,\TT]
\end{align*}
which precomposes a lax tangent morphism from $\DB(\X,\TT)$ to $(\X,\TT)$ with $\Inc$. By Proposition~\ref{proposition:hom-tangent-categories}, this is a strict tangent morphism thus, by Lemma~\ref{lemma:DB-functoriality}, we can apply the $2$-functor $\DB_\pt$ to it:
\begin{align*}
&\DB_\pt[\Inc\|\X,\TT]\colon\DB_\pt[\DB(\X,\TT)\|\X,\TT]\to\DB_\pt[\DB|_\1(\X,\TT)\|\X,\TT]
\end{align*}
In particular, this functor sends the universal pointwise differential bundle $\Univ q$ to a pointwise display differential bundle $\Univ q_\Inc$. However, the base object of $\Univ q_\Inc$ corresponds to the tangent morphism $\Base\o\Inc=\*\:\o\:!=\top$, which is the terminal object of $[\DB|_\1(\X,\TT)\|\X,\TT]$. Thus, $\Univ q_\Inc$ is a differential bundle on the terminal object. By \cite[Proposition~3.4]{cockett:differential-bundles}, $\Univ q_\Inc$ is then a differential object in the Hom-tangent category $[\DB|_\1(\X,\TT)\|\X,\TT]$. The goal is to prove that $\Univ q_\Inc$ is the universal differential object of $[\DB|_\1(\X,\TT)\|\X,\TT]$.
\par Recall that every differential object $(G,\beta;\zeta,\sigma,\hat p)$ of $[\DB|_\1(\X,\TT)\|\X,\TT]$ induces a functor
\begin{align*}
&\Gamma_{(G,\beta;\zeta,\sigma,\hat p)}\colon[\X',\TT'\|\DB|_\1(\X,\TT)]\to\DO[\X',\TT'\|\X,\TT]
\end{align*}
which simply postcomposes by $(G,\beta;\zeta,\sigma,\hat p)$. So, in particular, the differential object $\Univ q_\Inc$ of $[\DB|_\1(\X,\TT)\|\X,\TT]$ induces a functor
\begin{align*}
&\Gamma_{\Univ q_\Inc}\colon[\X',\TT'\|\DB|_\1(\X,\TT)]\to\DO[\X',\TT'\|\X,\TT]\cong\DB_\pt|_\1[\X',\TT'\|\X,\TT]
\end{align*}
where $\DB|_\1[\X',\TT'\|\X,\TT]$ denotes the tangent category of pointwise display differential bundles of $[\X',\TT'\|\X,\TT]$ over the terminal object $\top$. Consider the following commutative diagram of tangent categories
\begin{equation}
\label{equation:differential-objects-from-bundles}
% https://q.uiver.app/#q=WzAsNCxbMCwwLCJbXFxYJyxcXFRUJ1xcfFxcREJ8X1xcMShcXFgsXFxUVCldIl0sWzEsMCwiXFxEQnxfXFwxW1xcWCcsXFxUVCdcXHxcXFgsXFxUVF0iXSxbMCwxLCJbXFxYJyxcXFRUJ1xcfFxcREIoXFxYLFxcVFQpXSJdLFsxLDEsIlxcREJbXFxYJyxcXFRUJ1xcfFxcWCxcXFRUXSJdLFswLDEsIlxcR2FtbWFfe1xcVW5pdiBxX1xcSW5jfSJdLFsyLDMsIlxcR2FtbWFfe1xcVW5pdiBxfSIsMl0sWzAsMiwiW1xcWCcsXFxUVCdcXHxcXEluY10iLDJdLFsxLDMsIlxcSW5jIl1d
\begin{tikzcd}
{[\X',\TT'\|\DB|_\1(\X,\TT)]} & {\DB|_\1[\X',\TT'\|\X,\TT]} \\
{[\X',\TT'\|\DB(\X,\TT)]} & {\DB[\X',\TT'\|\X,\TT]}
\arrow["{\Gamma_{\Univ q_\Inc}}", from=1-1, to=1-2]
\arrow["{[\X',\TT'\|\Inc]}"', from=1-1, to=2-1]
\arrow["\Inc", from=1-2, to=2-2]
\arrow["{\Gamma_{\Univ q}}"', from=2-1, to=2-2]
\end{tikzcd}
\end{equation}
The next lemma shows that this diagram is a pullback in the category $\Cat$ of categories.

\begin{lemma}
\label{lemma:differential-objects-from-bundles}
The diagram of Equation~\eqref{equation:differential-objects-from-bundles} is a pullback in $\Cat$.
\end{lemma}
\begin{proof}
Consider a category $\X''$ and two functors $A\colon\X''\to[\X',\TT'\|\DB(\X,\TT)]$ and $B\colon\DB|_\1[\X',\TT'\|\X,\TT]$ such that the following diagram commutes:
\begin{equation*}
% https://q.uiver.app/#q=WzAsNSxbMSwxLCJbXFxYJyxcXFRUJ1xcfFxcREJ8X1xcMShcXFgsXFxUVCldIl0sWzIsMSwiXFxEQnxfXFwxW1xcWCcsXFxUVCdcXHxcXFgsXFxUVF0iXSxbMSwyLCJbXFxYJyxcXFRUJ1xcfFxcREIoXFxYLFxcVFQpXSJdLFsyLDIsIlxcREJbXFxYJyxcXFRUJ1xcfFxcWCxcXFRUXSJdLFswLDAsIlxcWCcnIl0sWzAsMSwiXFxHYW1tYV97XFxVbml2IHFfXFxJbmN9Il0sWzIsMywiXFxHYW1tYV97XFxVbml2IHF9IiwyXSxbMCwyLCJbXFxYJyxcXFRUJ1xcfFxcSW5jXSIsMl0sWzEsMywiXFxJbmMiXSxbNCwyLCJBIiwyLHsiY3VydmUiOjV9XSxbNCwxLCJCIiwwLHsiY3VydmUiOi00fV1d
\begin{tikzcd}
{\X''} \\
& {[\X',\TT'\|\DB|_\1(\X,\TT)]} & {\DB|_\1[\X',\TT'\|\X,\TT]} \\
& {[\X',\TT'\|\DB(\X,\TT)]} & {\DB[\X',\TT'\|\X,\TT]}
\arrow["B", curve={height=-24pt}, from=1-1, to=2-3]
\arrow["A"', curve={height=30pt}, from=1-1, to=3-2]
\arrow["{\Gamma_{\Univ q_\Inc}}", from=2-2, to=2-3]
\arrow["{[\X',\TT'\|\Inc]}"', from=2-2, to=3-2]
\arrow["\Inc", from=2-3, to=3-3]
\arrow["{\Gamma_{\Univ q}}"', from=3-2, to=3-3]
\end{tikzcd}
\end{equation*}
Consider an object $X$ of $\X''$. The functor $A$ sends $X$ to a lax tangent morphism
\begin{align*}
&A(X)\=(H_X,\gamma_X)\colon(\X',\TT')\to\DB(\X,\TT)
\end{align*}
such that, when postcomposed by the structural $2$-morphisms of the universal pointwise display differential bundle $\Univ q$, that is, by applying $\Gamma_{\Univ q}$ to $(H_X,\gamma_X)$, coincides with the display differential bundle $B(X)$, whose base object is the terminal object $\top$ of $[\X',\TT'\|\X,\TT]$. In particular, $\Base\o(H_X,\gamma_X)=\top=\*\:\o\:!$. However, from the universal property of the $2$-pullback of Equation~\eqref{equation:restriction-differential-bundles-to-terminal}, we obtain a unique tangent morphism $C(X)\colon(\X',\TT')\to\DB|_\1(\X,\TT)$ which makes the following diagram to commute:
\begin{equation*}
% https://q.uiver.app/#q=WzAsNSxbMSwxLCJcXERCfF9cXDEoXFxYLFxcVFQpIl0sWzIsMSwiXFxEQihcXFgsXFxUVCkiXSxbMiwyLCIoXFxYLFxcVFQpIl0sWzEsMiwiXFwxIl0sWzAsMCwiKFxcWCcsXFxUVCcpIl0sWzMsMiwiXFwqIiwyXSxbMCwzXSxbMSwyLCJcXEJhc2UiXSxbMCwxLCJcXEluYyJdLFswLDIsIiIsMSx7InN0eWxlIjp7Im5hbWUiOiJjb3JuZXIifX1dLFs0LDEsIihIX1gsXFxnYW1tYV9YKSIsMCx7ImN1cnZlIjotM31dLFs0LDMsIiEiLDIseyJjdXJ2ZSI6M31dLFs0LDAsIkMoWCkiLDAseyJzdHlsZSI6eyJib2R5Ijp7Im5hbWUiOiJkYXNoZWQifX19XV0=
\begin{tikzcd}
{(\X',\TT')} \\
& {\DB|_\1(\X,\TT)} & {\DB(\X,\TT)} \\
& \1 & {(\X,\TT)}
\arrow["{C(X)}", dashed, from=1-1, to=2-2]
\arrow["{(H_X,\gamma_X)}", curve={height=-18pt}, from=1-1, to=2-3]
\arrow["{!}"', curve={height=18pt}, from=1-1, to=3-2]
\arrow["\Inc", from=2-2, to=2-3]
\arrow[from=2-2, to=3-2]
\arrow["\lrcorner"{anchor=center, pos=0.125}, draw=none, from=2-2, to=3-3]
\arrow["\Base", from=2-3, to=3-3]
\arrow["{\*}"', from=3-2, to=3-3]
\end{tikzcd}
\end{equation*}
Consider now a morphism $f\colon X\to Y$ of $\X''$ and the corresponding tangent $2$-morphism
\begin{align*}
&A(f)\=\varphi_f\colon(H_X,\gamma_X)\Rightarrow(H_Y,\gamma_Y)
\end{align*}
By the universal property of the $2$-pullback of Equation~\eqref{equation:restriction-differential-bundles-to-terminal}, we obtain a unique $2$-morphism $C(f)\colon C(X)\Rightarrow C(Y)$, satisfying the following equation:
\begin{align*}
&\Inc(C(f))=\varphi_f\colon\Inc(C(X))=(H_X,\gamma_X)\to(H_Y,\gamma_Y)=\Inc(C(Y))
\end{align*}
Let us define a functor
\begin{align*}
&C\colon\X''\to[\X',\TT'\|\DB|_\1(\X,\TT)]
\end{align*}
which sends each $X$ to $C(X)$ and each $f\colon X\to Y$ to $C(f)$. By using the universal property of the $2$-pullback of Equation~\eqref{equation:restriction-differential-bundles-to-terminal} one can show that $C$ is a functor. Moreover, by postcomposing each $C(X)$ by $\Inc$, we obtain back $(H_X,\gamma_X)$, that is, $A(X)$, and since $\Gamma_{\Univ q_\Inc}(H_X,\gamma_X)$ must coincide with $\Inc(B(X))$, it is easy to prove that $\Gamma_{\Univ q_\Inc}(C(X))$ is equal to $B(X)$. Thus, $C$ is the unique functor which makes the following diagram commute:
\begin{equation*}
% https://q.uiver.app/#q=WzAsNSxbMSwxLCJbXFxYJyxcXFRUJ1xcfFxcREJ8X1xcMShcXFgsXFxUVCldIl0sWzIsMSwiXFxEQnxfXFwxW1xcWCcsXFxUVCdcXHxcXFgsXFxUVF0iXSxbMSwyLCJbXFxYJyxcXFRUJ1xcfFxcREIoXFxYLFxcVFQpXSJdLFsyLDIsIlxcREJbXFxYJyxcXFRUJ1xcfFxcWCxcXFRUXSJdLFswLDAsIlxcWCcnIl0sWzAsMSwiXFxHYW1tYV97XFxVbml2IHFfXFxJbmN9Il0sWzIsMywiXFxHYW1tYV97XFxVbml2IHF9IiwyXSxbMCwyLCJbXFxYJyxcXFRUJ1xcfFxcSW5jXSIsMl0sWzEsMywiXFxJbmMiXSxbNCwyLCJBIiwyLHsiY3VydmUiOjV9XSxbNCwxLCJCIiwwLHsiY3VydmUiOi00fV0sWzQsMCwiQyIsMCx7InN0eWxlIjp7ImJvZHkiOnsibmFtZSI6ImRhc2hlZCJ9fX1dXQ==
\begin{tikzcd}
{\X''} \\
& {[\X',\TT'\|\DB|_\1(\X,\TT)]} & {\DB|_\1[\X',\TT'\|\X,\TT]} \\
& {[\X',\TT'\|\DB(\X,\TT)]} & {\DB[\X',\TT'\|\X,\TT]}
\arrow["C", dashed, from=1-1, to=2-2]
\arrow["B", curve={height=-24pt}, from=1-1, to=2-3]
\arrow["A"', curve={height=30pt}, from=1-1, to=3-2]
\arrow["{\Gamma_{\Univ q_\Inc}}", from=2-2, to=2-3]
\arrow["{[\X',\TT'\|\Inc]}"', from=2-2, to=3-2]
\arrow["\Inc", from=2-3, to=3-3]
\arrow["{\Gamma_{\Univ q}}"', from=3-2, to=3-3]
\end{tikzcd}
\end{equation*}
This proves the statement of the lemma.
\end{proof}

We can now prove the main result of this section.

\begin{theorem}
\label{theorem:differential-objects-from-differential-bundles}
Let $\CC$ be a Cartesian $2$-category and $(\X,\TT)$ be a Cartesian tangentad of $\CC$. If $(\X,\TT)$ admits the construction of display differential bundles and the $2$-pullback of Equation~\eqref{equation:restriction-differential-bundles-to-terminal} exists in $\Tng(\CC)$, $(\X,\TT)$ admits the construction of differential objects and the tangentad of differential objects of $(\X,\TT)$ is $\DB|_\1(\X,\TT)$.
\end{theorem}
\begin{proof}
Since isomorphisms are stable under pullbacks and $\Gamma_{\Univ q}$ is an isomorphism of tangent categories, by Lemma~\ref{lemma:differential-objects-from-bundles}, also $\Gamma_{\Univ q_\Inc}$ is an isomorphism of categories. However, since $\Gamma_{\Univ q_\Inc}$ is a strict tangent morphism, its inverse is also a strict tangent morphism. Thus, $\Gamma_{\Univ q_\Inc}$ becomes an isomorphism of tangent categories. By~\cite[Proposition~3.4]{cockett:differential-bundles}, every differential bundle on the terminal object is a differential object, thus, $\DB|_\1[\X',\TT'\|\X,\TT]$ is isomorphic to $\DO[\X',\TT'\|\X,\TT]$. Thus, we obtain an isomorphism of tangent categories:
\begin{align*}
&\Gamma_{\Univ q_\Inc}\colon[\X',\TT'\|\DB|_\1(\X,\TT)]\cong\DO[\X',\TT'\|\X,\TT]
\end{align*}
Therefore, $\DB|_\1(\X,\TT)$ is the tangentad of differential objects of $(\X,\TT)$.
\end{proof}

%__________________________________________________________________________
\subsection{Applications}
\label{subsection:examples-differential-bundles}
In Section~\ref{subsection:definition-tangentad}, we listed some examples of tangentads, and in Section~\ref{subsection:examples-differential-objects}, we computed the constructions of differential objects for tangent monads, tangent fibrations, tangent indexed categories, tangent split restriction categories, and showed how to extend these constructions to tangent restriction categories. In this section, we compute the construction of differential bundles for the same examples of tangentads and show how to extend it to tangent restriction categories. Since the proofs are fairly similar to those for differential objects, we provide the final constructions without giving the details of the proofs.

%__________________________________________________________________________
\subsubsection*{Tangent monads}
\label{subsubsection:differential-bundles-tangent-monads}
Let us start by considering tangent monads. We previously discussed that strong tangent morphisms lift to the tangent categories of differential bundles. It is not hard to convince ourselves that every strong tangent monad $(S,\alpha)$ on a tangent category $(\X,\TT)$, that is, a tangent monad whose underlying tangent morphism is strong, lifts to $\DB(\X,\TT)$ as a tangent monad $\DB(S,\alpha)$.

\begin{theorem}
\label{theorem:differential-bundles-tangent-monads}
In the $2$-category $\TngMnd$ of tangent monads, each strong tangent monad $(S,\alpha)$ on a tangent category admits the construction of differential bundles. Moreover, the tangentad of differential bundles of $(S,\alpha)$ is the tangent monad $\DB(S,\alpha)$ on $\DB(\X,\TT)$.
\end{theorem}

%__________________________________________________________________________
\subsubsection*{Tangent fibrations}
\label{subsubsection:differential-bundles-tangent-fibrations}
Since the underlying functor $\Pi\colon(\X',\TT')\to(\X,\TT)$ of a tangent fibration is a strict tangent morphism, hence strong, it lifts to a strict tangent morphism $\DB(\Pi)\colon\DB(\X',\TT')\to\DB(\X,\TT)$. However, since $\Pi$ is a tangent fibration, one can show that $\DB(\Pi)$ is also a tangent fibration.

\begin{lemma}
\label{lemma:differential-bundles-tangent-fibration}
Consider a (cloven) tangent fibration $\Pi\colon(\X',\TT')\to(\X,\TT)$ between two tangent categories. The strict tangent morphism
\begin{align*}
&\DB(\Pi)\colon\DB(\X',\TT')\to\DB(\X,\TT)
\end{align*}
which sends a differential bundle $\q'\colon E'\to M'$ to the differential bundle $(\Pi(E),\Pi(\zeta_E),\Pi(\sigma_E),\Pi(\hat p_E))$ is a (cloven) tangent fibration.
\end{lemma}

\begin{theorem}
\label{theorem:differential-bundles-tangent-fibrations}
The $2$-category $\Fib$ of (cloven) fibrations admits the construction of differential bundles. Moreover, the tangentad of differential bundles of a (cloven) tangent fibration $\Pi$ is $\DB(\Pi)\colon\DB(\X',\TT')\to\DB(\X,\TT)$.
\end{theorem}

%__________________________________________________________________________
\subsubsection*{Tangent indexed categories}
\label{subsubsection:differential-bundles-tangent-categories}
In this section, we employ the Grothendieck $2$-equivalence $\TngFib\simeq\TngIndx$ to compute the tangentad of differential bundles of a tangent indexed category. Let us start by proving that $2$-equivalences preserve the construction of differential bundles.

\begin{proposition}
\label{proposition:equivalence-differential-bundles}
Let $\CC$ and $\CC'$ be two $2$-categories and suppose there is a $2$-equivalence $\Xi\colon\Tng(\CC)\simeq\Tng(\CC')$ of the $2$-categories of tangentads of $\CC$ and $\CC'$. If a tangentad $(\X,\TT)$ of $\CC$ admits the construction of differential bundles and $\DB(\X,\TT)$ denotes the tangentad of differential bundles of $(\X,\TT)$, so does $\Xi(\X,\TT)$ and the tangentad of differential bundles of $\Xi(\X,\TT)$ is $\Xi(\DB(\X,\TT))$.
\end{proposition}

Proposition~\ref{proposition:equivalence-differential-bundles} together with the Grothendieck $2$-equivalence $\TngFib\cong\TngIndx$ allows us to compute the differential bundles of tangent indexed categories. Consider a tangent indexed category $(\X,\TT;\IND,\TT')$ whose underlying indexed category $\IND\colon\X^\op\to\Cat$ sends an object $M$ of $\X$ to the category $\X^M$ and a morphism $f\colon M\to N$ to the functor $f^\*\colon\X^N\to\X^M$. Define $\DB(\X,\TT;\IND,\TT')$ as follows:
\begin{description}
\item[Base tangent category] The base tangent category is $\DB(\X,\TT)$;

\item[Indexed category] The indexed category $\DB(\IND)\colon\DB(\X,\TT)^\op\to\Cat$ sends a differential bundle $\q\colon E\to M$ of $(\X,\TT)$ to the category $\DB^\q$ whose objects are tuples $\q'\=(q',z_q',s_q',l_q')$ formed by an $M'\in\X^M$, an object $E'\in\X^E$, and four morphisms
\begin{align*}
q'&\colon E'\to q^\*M'                  &z_q'&\colon M'\to z_q^\*E'\\
s_q'&\colon E'_2\to s_q^\*E             &l_q'&\colon E'\to l_q^\*\T'E'
\end{align*}
such that $(\q,\q')\colon(E,E')\to(M,M')$ becomes a display differential bundle in the tangent category of elements of the indexed tangent category.\\
Furthermore, $\DB(\IND)$ sends a morphism $(f,g)\colon(\q_1\colon E_1\to M_1)\to(\q_2\colon E_2\to M_2)$ of differential bundles to the functor
\begin{align*}
(f,g)^\*&\colon\DB^{\q_1}\to\DB^{\q_2}
\end{align*}
which sends a tuple $\q'\colon E'\to M'$ to the tuple $((f,g)^\*q',(f,g)^\*z_q',(f,g)^\*s_q',(f,g)^\*l_q')$, where:
\begin{align*}
(f,g)^\*q'&\colon g^\*E'\xrightarrow{g^\*q'}g^\*q_2^\*M'\xrightarrow{q_2\o g=f\o q_1}q_1^\*f^\*M'\\
(f,g)^\*z_q'&\colon f^\*M'\xrightarrow{f^\*z_q'}f^\*z_{q_2}^\*E'\xrightarrow{z_{q_2}\o f=g\o z_{q_1}}z_{q_1}^\*g^\*E'\\
(f,g)^\*s_q'&\colon g^\*E_2'\xrightarrow{g^\*s'_q}g^\*s_{q_2}^\*E'\xrightarrow{s_{q_2}\o g=g_2\o s_{q_1}}s_{q_1}^\*g_2^\*E'\cong s_{q_1}^\*(g^\*E')_2\\
(f,g)^\*l_q'&\colon g^\*E'\xrightarrow{g^\*l_q'}g^\*l_{q_2}^\*\T'E'\xrightarrow{l_{q_2}\o g=\T g\o l_{q_1}}l_{q_1}^\*(\T g)^\*\T'E'\xrightarrow{{\xi^g}^{-1}}l_{q_1}^\*\T'(g^\*E')
\end{align*}

\item[Indexed tangent bundle functor] The indexed tangent bundle functor $\DB(\T')$ consists of the list of functors
\begin{align*}
&\T'^\q\colon\DB^\q\to\DB^{\T^\DB\q}
\end{align*}
which send a tuple $\q'\colon E'\to M'$ to the tuple $\T'^\q\q'$ so defined:
\begin{align*}
\T'q'&\colon\T'^EE'\xrightarrow{\T'^Eq'}\T'^Eq^\*M'\xrightarrow{\xi_q}(\T q)^\*\T^MM'\\
\T'z_q'&\colon\T^MM'\xrightarrow{\T^Mz_q'}\T^Mz_q^\*E'\xrightarrow{\xi_{z_q}}(\T z_q)^\*\T^EE'\\
\T's_q'&\colon\T^{E_2}E_2'\xrightarrow{\T^{E_2}s_q'}\T^{E_2}s_q^\*E'\xrightarrow{\xi_{s_q}}(\T s_q)\T^EE'\\
l_{\T q}'&\colon\T^EE'\xrightarrow{\T^El_q'}\T^El_q^\*\T^EE'\xrightarrow{\xi_{l_q}}(\T l_q)^\*\T^{\T E}\T^EE'\xrightarrow{(\T l_q)^\*c'}(\T l_q)^\*c^\*\T^{\T E}\T^EE'\xrightarrow{\IND_2}l_{\T q}^\*\T^{\T E}\T^EE'
\end{align*}
Moreover, $\T'^\q$ sends a morphism $(\varphi,\psi)\colon(\q'_1\colon E'_1\to M'_1)\to(\q'_2\colon E_2'\to M_2')$ to $(\T'^M\varphi,\T'^E\psi)$. The distributors of $\T'^\q$ are pairs $(\xi^f,\xi^g)$ of distributors of $\T'^M$ and $\T'^E$;

\item[Indexed natural transformations] The structural indexed natural transformations of $\DB(\IND)$ are the same as for $\IND$.
\end{description}

\begin{theorem}
\label{theorem:differential-bundles-tangent-indexed-categories}
The $2$-category $\Indx$ of indexed categories admits the construction of differential bundles. Moreover, given a tangent indexed category $(\X,\TT;\IND,\TT')$, the tangent indexed category of differential bundles of $(\X,\TT;\IND,\TT')$ is the tangent indexed category $\DB(\X,\TT;\IND,\TT')=(\DB(\X,\TT),\DB(\IND),\DB(\TT'))$.
\end{theorem}

%__________________________________________________________________________
\subsubsection*{Tangent split restriction categories}
\label{subsubsection:differential-bundles-tangent-split-restriction-categories}
In this section, we consider the construction of differential bundles for tangent split restriction categories.

\begin{lemma}
\label{lemma:differential-bundles-tangent-split-restriction-categories}
Consider a tangent split restriction category $(\X,\TT)$ and let $\DB(\X,\TT)$ be the category whose objects are differential bundles $\q\colon E\to M$ in the tangent subcategory $\Tot(\X,\TT)$ and morphisms are morphisms $(f,g)$ of differential bundles such that also their restriction idempotents $(\bar f,\bar g)$ are also morphisms of differential bundles. The tangent split restriction structure on $(\X,\TT)$ lifts to $\DB(\X,\TT)$. In particular, the tangent bundle functor sends each differential bundle $\q\colon E\to M$ to $\T^\DB\q$, where
\begin{align*}
&l_{\T q}\colon\T E\xrightarrow{l_q}\T^2E\xrightarrow{c}\T^2E
\end{align*}
\end{lemma}

\begin{definition}
\label{definition:restriction-differential-bundle}
A \textbf{restriction differential bundle} in a Cartesian tangent split restriction category $(\X,\TT)$ is an object of $\DB(\X,\TT)$.
\end{definition}

Concretely, a restriction differential bundle consists of two objects $M$ and $E$ and four total maps
\begin{align*}
q&\colon E\to M                     &z_q&\colon M\to E\\
s_q&\colon E_2\to E                 &l_q&\colon E\to\T E
\end{align*}
where $E_n$ is the restriction $n$-fold pullback of $q$ along itself:
\begin{equation*}
% https://q.uiver.app/#q=WzAsNCxbMCwwLCJFX24iXSxbMCwxLCJFIl0sWzEsMCwiRSJdLFsxLDEsIk0iXSxbMSwzLCJxIiwyXSxbMiwzLCJxIl0sWzAsMSwiXFxwaV8xIiwyXSxbMCwyLCJcXHBpX24iXSxbMCwzLCIiLDEseyJzdHlsZSI6eyJuYW1lIjoiY29ybmVyIn19XSxbMiwxLCJcXGRvdHMiLDMseyJvZmZzZXQiOjMsInN0eWxlIjp7ImJvZHkiOnsibmFtZSI6Im5vbmUifSwiaGVhZCI6eyJuYW1lIjoibm9uZSJ9fX1dXQ==
\begin{tikzcd}
{E_n} & E \\
E & M
\arrow["{\pi_n}", from=1-1, to=1-2]
\arrow["{\pi_1}"', from=1-1, to=2-1]
\arrow["\lrcorner"{anchor=center, pos=0.125}, draw=none, from=1-1, to=2-2]
\arrow["\dots"{marking, allow upside down}, shift right=3, draw=none, from=1-2, to=2-1]
\arrow["q", from=1-2, to=2-2]
\arrow["q"', from=2-1, to=2-2]
\end{tikzcd}
\end{equation*}
The total maps $q$, $z_q$, $s_q$, $l_q$ satisfy the same equational axioms of a differential bundle in ordinary tangent category theory. However, the universality of the vertical lift is replaced with the following axiom: the following is a restriction pullback diagram
\begin{equation*}
% https://q.uiver.app/#q=WzAsNCxbMCwwLCJFXzIiXSxbMCwxLCJNIl0sWzEsMSwiXFxUIE0iXSxbMSwwLCJcXFQgRSJdLFsxLDIsInoiLDJdLFswLDEsIlxccGlfMXEiLDJdLFszLDIsIlxcVCBxIl0sWzAsMywiXFx4aV9xIl0sWzAsMiwiIiwxLHsic3R5bGUiOnsibmFtZSI6ImNvcm5lciJ9fV1d
\begin{tikzcd}
{E_2} & {\T E} \\
M & {\T M}
\arrow["{\xi_q}", from=1-1, to=1-2]
\arrow["{\pi_1q}"', from=1-1, to=2-1]
\arrow["\lrcorner"{anchor=center, pos=0.125}, draw=none, from=1-1, to=2-2]
\arrow["{\T q}", from=1-2, to=2-2]
\arrow["z"', from=2-1, to=2-2]
\end{tikzcd}
\end{equation*}
whose universal property is preserved by the tangent bundle functor, and where $\xi_q$ is the morphism:
\begin{align*}
\xi_q&\colon E_2\xrightarrow{\<\pi_1l_q,\pi_2z\>}\T E_2\xrightarrow{\T s_q}\T E
\end{align*}

\begin{theorem}
\label{theorem:differential-bundles-tangent-split-restriction-categories}
The $2$-category $\sRestrCat$ of split restriction categories admits the construction of differential bundles. Moreover, the tangentad of differential bundles of a Cartesian tangent split restriction category $(\X,\TT)$ is the Cartesian tangent split restriction category $\DB(\X,\TT)$.
\end{theorem}

%__________________________________________________________________________
\subsubsection*{Tangent restriction categories: a general approach}
\label{subsubsection:differential-bundles-tangent-restriction-categories}
To extend the construction of differential bundles to tangent restriction categories, we adopt a similar technique to the one used to extend the constructions of vector fields in~\cite[Section~5]{lanfranchi:tangentads-II} and of differential objects in Section~\ref{subsection:examples-differential-objects}, that is, via a suitable $2$-pullback. The main difference with Section~\ref{subsection:examples-differential-objects} is the ambient $2$-category where we need to consider the $2$-pullback.
\par The issue is that the construction of differential bundles involves not one but two forgetful $1$-morphisms $\Base,\Tot\colon\DB(\X,\TT)\to(\X,\TT)$ which project a differential bundle in $(\X,\TT)$ to its base and its total spaces, respectively. The purpose of the $2$-pullbacks of Equation~\eqref{equation:pullback-extension-differential-objects} was to restrict to those restriction vector fields and restriction differential objects in $\Split_R(\X,\TT)$ whose underlying object was coming directly from $(\X,\TT)$ via the unit $\eta$. For differential bundles, we need to restrict to those restriction differential bundles $\q\colon  E\to M$ of $\Split_R(\X,\TT)$ whose base and total objects $M$ and $E$ come directly from objects of $(\X,\TT)$ via $\eta$.
\par To solve this problem, we consider the $2$-category of parallel $1$-morphisms.

\begin{definition}
\label{definition:parallel-1-morphisms}
The \textbf{$2$-category of $n$ parallel morphisms} of a $2$-category $\DD$ is the $2$-category $\Parall_n(\DD)$ whose objects are the objects of $\DD$, a $1$-morphism from an object $\X$ to another object $\X'$ consists of a $n$-tuple $(f_1\,f_n)$ of parallel $1$-morphisms $f_1\,f_n\colon\X\to\X'$ of $\DD$, and a $2$-morphism $(\varphi_1\,\varphi_n)$ from a $1$-morphism $(f_1\,f_n)$ to a morphism $(g_1\,g_n)$ consists of a $n$-tuple $\varphi_1\colon f_1\Rightarrow g_1\,\varphi_n\colon f_n\Rightarrow g_n$ of $2$-morphisms of $\DD$.
\end{definition}

\par Consider a pullback-extension context, that is, two $2$-categories $\CC$ and $\DD$ together with two $2$-functors
\begin{align*}
&\Xi\colon\DD\leftrightarrows\Tng(\CC)\colon\Inc
\end{align*}
together with a natural $2$-transformation
\begin{align*}
\eta_\X&\colon\X\to\Inc(\Xi(\X))
\end{align*}
natural in $\X\in\DD$. Let us also consider an object $\X$ of $\DD$ such that the tangentad $\Xi(\X)$ of $\CC$ admits the construction of differential bundles. Finally, let us assume the existence of the following $2$-pullback in $\Parall_2(\DD)$:
\begin{equation}
\label{equation:pullback-extension-differential-bundles}
% https://q.uiver.app/#q=WzAsNCxbMSwxLCJcXEluYyhcXFhpKFxcWCkpIl0sWzAsMSwiXFxYIl0sWzAsMCwiXFxEQihcXFgpIl0sWzEsMCwiXFxJbmMoXFxEQihcXFhpKFxcWCkpIl0sWzEsMCwiKFxcZXRhLFxcZXRhKSIsMl0sWzIsMSwiKFxcQmFzZV9cXFgsXFxUb3RfXFxYKSIsMl0sWzIsMCwiIiwxLHsic3R5bGUiOnsibmFtZSI6ImNvcm5lciJ9fV0sWzMsMCwiKFxcSW5jKFxcQmFzZV97XFxYaShcXFgpfSksXFxJbmMoXFxUb3Rfe1xcWGkoXFxYKX0pKSJdLFsyLDMsIlxcREIoXFxldGEsXFxldGEpIl1d
\begin{tikzcd}
{\DB(\X)} & {\Inc(\DB(\Xi(\X))} \\
\X & {\Inc(\Xi(\X))}
\arrow["{\DB(\eta,\eta)}", from=1-1, to=1-2]
\arrow["{(\Base_\X,\Tot_\X)}"', from=1-1, to=2-1]
\arrow["\lrcorner"{anchor=center, pos=0.125}, draw=none, from=1-1, to=2-2]
\arrow["{(\Inc(\Base_{\Xi(\X)}),\Inc(\Tot_{\Xi(\X)}))}", from=1-2, to=2-2]
\arrow["{(\eta,\eta)}"', from=2-1, to=2-2]
\end{tikzcd}
\end{equation}

\begin{definition}
\label{definition:differential-bundles-extension}
Let $(\CC,\DD;\Inc,\Xi;\eta)$ be a pullback-extension context. An object $\X$ of $\DD$ admits the \textbf{extended construction of differential bundles} (w.r.t. to the pullback-extension context) if $\Xi(\X)$ admits the construction of differential bundles in $\CC$, and the $2$-pullback diagram of Equation~\eqref{equation:pullback-extension-differential-bundles} exists in $\Parall_2(\DD)$. In this scenario, the \textbf{object of differential bundles} (w.r.t. to the pullback-extension context) of $\X$ is the object $\DB(\X)$ of $\DD$, that is, the $2$-pullback of $(\Inc(\Base_{\Xi(\X)}),\Inc(\Tot_{\Xi(\X)}))$ along $(\eta,\eta)$.
\end{definition}

Consider a generic tangent restriction category $(\X,\TT)$ and let us unwrap the definition of $\DB(\Split_R(\X,\TT))$. The objects of $\DB(\Split_R(\X,\TT))$ are tuples $(E,M,q,z_q,s_q,l_q)$ formed by two objects $M$ and $E$ of $(\X,\TT)$ together with four morphisms
\begin{align*}
q&\colon E\to M         &z_q&\colon M\to E\\
s_q&\colon E_2\to E     &l_q&\colon E\to\T E
\end{align*}
such that:
\begin{align*}
&\bar z_q\o q=q                     &&\bar q\o z_q=z_q\\
&\bar{s_q}=\bar q\times_M\bar q     &&\bar q\o s_q=s_q\\
&\bar{l_q}=\bar q                   &&\T\bar q\o l_q=l_q
\end{align*}
Moreover, $\zeta_A$, $\sigma_A$, $\hat p_A$ satisfy the axioms of a restriction differential bundle.
\par A morphism $(f,g)\colon(\q_1\colon E_1\to M_1)\to(\q_2\colon E_2\to M_2)$ of $\DB(\Split_R(\X,\TT))$ consists of a pair of morphisms $f\colon M_1\to M_2$ and $g\colon E_1\to E_2$ of $(\X,\TT)$ which commutes with the structural morphism of the differential bundles and with their restriction idempotents. The tangent bundle functor sends each $\q\colon E\to M$ to $\T^\DB\q\colon\T E\to\T M$.
\par Now, consider a tangent (non-necessarily split) restriction category $(\X,\TT)$ and define $\DB(\X,\TT)$ to be the full subcategory of $\DB(\Split_R(\X,\TT))$ spanned by the objects $\q\colon E\to M$ where $q$, $z_q$, $s_q$, and $l_q$ are total maps in $(\X,\TT)$, which is precisely when $q$ and $z_q$ are total.

\begin{lemma}
\label{lemma:differential-bundles-tangent-restriction-categories}
Let $(\X,\TT)$ be a tangent restriction category. The subcategory $\DB(\X,\TT)$ of $\DB(\Split_R(\X,\TT))$ spanned by the objects $\q\colon E\to M$ where $q$ and $z_q$ are total in $(\X,\TT)$ is a tangent restriction category.
\end{lemma}

We can finally prove the main theorem of this section.

\begin{theorem}
\label{theorem:differential-bundles-tangent-restriction-categories}
Consider the pullback-extension context $(\RestrCat,\TngRestrCat;\Inc,\Split_R;\eta)$ of tangent restriction categories. The $2$-category $\TngRestrCat$ admits the extended construction of differential bundles with respect to this pullback-extension context. Moreover, the object of differential bundles of a tangent restriction category is the tangent restriction category $\DB(\X,\TT)$.
\end{theorem}

%__________________________________________________________________________
%__________________________________________________________________________

\section{The formal theory of connections}
\label{section:connections}
The notion of connections on differential bundles, introduced in~\cite{cockett:connections}, extends the well-studied theory of connections on vector bundles, which is a key geometric concept of differential geometry and Riemannian geometry. Notions like covariant derivative, curvature, torsion, and parallel transport of a connection have been successfully extended to tangent category theory. In this section, we formalize this notion in the context of tangentads.
\par For starters, we recall the definition of a connection on a differential bundle in a tangent category and recall the construction of the tangent categories of connections. In Sections~\ref{subsection:universal-property-linear-connections} and~\ref{subsection:universal-property-affine-connections}, we identify the correct universal property enjoyed by the tangent categories of connections. In order to extend the notion of a covariant derivative of a connection, in Section~\ref{subsection:universal-property-linear-sections}, we first identify the correct universal property that classifies sections of differential bundles equipped with a connection. Finally, in Section~\ref{subsection:structures-connections}, we discuss the functoriality of the constructions and formally construct the covariant derivative, the curvature tensor, and the torsion tensor of connections.

%__________________________________________________________________________
\subsection{Connections in tangent category theory}
\label{subsection:tangent-category-connections}
A connection on a differential bundle is composed of two parts: a vertical connection and a horizontal connection. Let us first recall the notion of a vertical connection

\begin{definition}[{\cite[Definition~3.2]{cockett:connections}}]
\label{definition:linear-vertical-connection}
Given a display differential bundle $\q\colon E\to M$ of a tangent category $(\X,\TT)$, a \textbf{vertical linear connection} on a differential bundle consists of a morphism $k\colon\T E\to E$, satisfying the following conditions:
\begin{enumerate}
\item $k$ is a retract of the vertical lift $l_q\colon E\to\T E$, that is:
\begin{align*}
&l_qk=\id_E
\end{align*}

\item Linearity (1):
\begin{align*}
&(p,k)\colon\T^\LC\q\to\q
\end{align*}
is a linear morphism of differential bundles;

\item Linearity (2):
\begin{align*}
(q,k)\colon\p_E\to\q
\end{align*}
is a linear morphism of differential bundles.
\end{enumerate}
When the differential bundle $\q$ is the tangent bundle $\p_M$, $k$ is called an \textbf{vertical affine connection} on $M$.
\end{definition}

To introduce horizontal connections, recall that the horizontal bundle of a tangent display map $q\colon E\to M$ is the pullback of $q$ along the tangent bundle projection $p\colon\T M\to M$, that is:
\begin{equation*}
% https://q.uiver.app/#q=WzAsNCxbMCwwLCJcXEggcSJdLFsxLDAsIkUiXSxbMSwxLCJNIl0sWzAsMSwiXFxUIE0iXSxbMywyLCJwIiwyXSxbMSwyLCJxIl0sWzAsM10sWzAsMV0sWzAsMiwiIiwxLHsic3R5bGUiOnsibmFtZSI6ImNvcm5lciJ9fV1d
\begin{tikzcd}
{\H q} & E \\
{\T M} & M
\arrow[from=1-1, to=1-2]
\arrow[from=1-1, to=2-1]
\arrow["\lrcorner"{anchor=center, pos=0.125}, draw=none, from=1-1, to=2-2]
\arrow["q", from=1-2, to=2-2]
\arrow["p"', from=2-1, to=2-2]
\end{tikzcd}
\end{equation*}
Since $\pi_2\colon\H q\to E$ is the pullback of a differential bundle (the tangent bundle), $\pi_2$ becomes a differential bundle, which is called the \textbf{horizontal bundle} of $q$.

\begin{definition}[{\cite[Definition~4.5]{cockett:connections}}]
\label{definition:linear-horizontal-connection}
Given a display differential bundle $\q\colon E\to M$ of a tangent category $(\X,\TT)$, a \textbf{horizontal linear connection} on a differential bundle consists of a morphism $h\colon\T E\to\H q$, satisfying the following conditions:
\begin{enumerate}
\item $h$ is a section of $\<\T q,p\>\colon\T E\to \T M\times_M E=\H q$:
\begin{align*}
&h\<\T q,p\>=\id_{\H q}
\end{align*}

\item Linearity (1):
\begin{align*}
&(\id_E,h)\colon\pi_2\to\p_E
\end{align*}
is a linear morphism of differential bundles, where $\pi_2\colon\H q\to E$ denotes the horizontal bundle of $q$;

\item Linearity (2):
\begin{align*}
(\id_{\T M},h)\colon\pi_2\to\T^\LC\q
\end{align*}
is a linear morphism of differential bundles, where $\pi_2\colon\H q\to E$ denotes the horizontal bundle of $q$.
\end{enumerate}
When the differential bundle $\q$ is the tangent bundle $\p_M$, $h$ is called an \textbf{horizontal affine connection} on $M$.
\end{definition}

\begin{definition}[{\cite[Definition~5.2]{cockett:connections}}]
\label{definition:connection}
A \textbf{linear connection} on a display differential bundle $\q\colon E\to M$ consists of a vertical linear connection $k\colon E\to\T E$ together with a horizontal linear connection $h\colon\T E\to \H q$ satisfying the following two conditions:
\begin{itemize}
\item The vertical and the horizontal components are orthogonal:
\begin{equation*}
% https://q.uiver.app/#q=WzAsNSxbMiwwLCJcXFQgRSJdLFswLDAsIkhFIl0sWzIsMSwiRSJdLFswLDEsIkUiXSxbMSwxLCJNIl0sWzAsMiwiayJdLFsxLDAsImgiXSxbMSwzLCJcXHBpXzIiLDJdLFszLDQsInEiLDJdLFs0LDIsInpfcSIsMl1d
\begin{tikzcd}
\H q && {\T E} \\
E & M & E
\arrow["h", from=1-1, to=1-3]
\arrow["{\pi_2}"', from=1-1, to=2-1]
\arrow["k", from=1-3, to=2-3]
\arrow["q"', from=2-1, to=2-2]
\arrow["{z_q}"', from=2-2, to=2-3]
\end{tikzcd}
\end{equation*}

\item The vertical and the horizontal components are in direct sum to each other:
\begin{align*}
&\<k,p\>(z\times l)\T s_q+\<\T q,p\>h=\id_{\T E}
\end{align*}
\end{itemize}
When the differential bundle $\q$ is the tangent bundle $\p_M$, a linear connection $(k,h)$ is called an \textbf{affine connection}.
\end{definition}

As shown by Cockett and Cruttwell, a (vertical) linear connection on a differential bundle comes with a notion of a covariant derivative.

\begin{definition}[{\cite[Definition~3.15]{cockett:connections}}]
\label{definition:covariant-derivative}
The \textbf{covariant derivative} of a vertical linear connection $k$ of a differential bundle $\q\colon E\to M$ consists of a combinator $\nabla^k$ which sends a vector field $v\colon M\to\T M$ of $M$ and a section $s\colon M\to E$ of $q$ to a section $\nabla^k_vs$ of $q$ so defined:
\begin{align*}
&\nabla^k_vs\colon M\xrightarrow{v}\T M\xrightarrow{\T s}\T E\xrightarrow{k}E
\end{align*}
\end{definition}

Furthermore, when the tangent category has negatives, a (vertical) linear connection on a differential bundle comes with a notion of curvature.

\begin{definition}[{\cite[Definitions~3.20,~3.21]{cockett:connections}}]
\label{definition:curvature}
The \textbf{curvature} $\Curv^k$ of a vertical linear connection $k$ of a differential bundle $\q\colon E\to M$ is the morphism so defined:
\begin{align*}
&\Curv^k\=k\o\T k-k\o\T k\o c\colon\T^2E\to E
\end{align*}
Furthermore, the \textbf{curvature tensor} of $k$ is the combinator $\Riem^k$ which sends two vector fields $u,v\colon M\to\T M$ of $M$ and a section $s\colon M\to E$ of $q$ to the section $\Riem^k_{u,v}s$ of $q$ so defined:
\begin{align*}
&\Riem^k_{u,v}s\colon M\xrightarrow{v}\T M\xrightarrow{\T u}\T^2M\xrightarrow{\T^2s}\T^2E\xrightarrow{\Curv^k}E
\end{align*}
\end{definition}

When the ambient tangent category has negatives, (vertical) affine connections carry a notion of torsion.

\begin{definition}[{\cite[Definitions~3.29,~3.30]{cockett:connections}}]
\label{definition:torsion}
The \textbf{torsion} of a vertical affine connection $k$ on an object $M$ is the morphism so defined:
\begin{align*}
&\Tors^k\=ck-k\colon\T^2M\to\T M
\end{align*}
Furthermore, the \textbf{torsion tensor} of $k$ is the combinator $\TorsTens^k$ which sends two vector fields $u,v\colon M\to\T M$ of $M$ to the vector field $\TorsTens^k_uv$ of $M$ so defined:
\begin{align*}
&\TorsTens^k_uv\colon M\xrightarrow{v}\T M\xrightarrow{\T u}\T^2M\xrightarrow{\Tors^k}\T M
\end{align*}
\end{definition}

For each tangent category $(\X,\TT)$, there is a tangent category denoted by $\VLC(\X,\TT)$ whose objects are pairs $(\q,k)$ formed by a display differential bundle $\q\colon E\to M$ together with a vertical linear connection $k\colon\T E\to E$ on $\q$ and morphisms $(f,g)\colon(\q,k)\to(\q',k')$ are linear morphisms $(f,g)\colon\q\to\q'$ of differential bundles that commute with the vertical connections, $g\o k=k'\o\T g$. The tangent bundle functor of $\VLC(\X,\TT)$ sends a pair $(\q,k)$ to $\T^\LC\q,k_\T)$, where
\begin{align*}
&k_\T\colon\T^2E\xrightarrow{c}\T^2E\xrightarrow{\T k}\T E
\end{align*}
and a morphism $(f,g)\colon(\q,k)\to(\q',k')$ to $(\T f,\T g)$. Finally, the structural natural transformations of $\VLC(\X,\TT)$ are the same as the ones of $(\X,\TT)$.
\par For each tangent category $(\X,\TT)$, there is also a tangent category denoted by $\HLC(\X,\TT)$, whose objects are pairs $(\q,h)$ formed by a display differential bundle $\q\colon E\to M$ together with a horizontal linear connection $h\colon\H\q\to\T E$ on $\q$ and morphisms $(f,g)\colon(\q,h)\to(\q',h')$ are linear morphisms $(f,g)\colon\q\to\q'$ of differential bundles, that commute with the horizontal connections, $\H g\o h=h'\o\T g$, where $\H g$ is the obvious morphism $\H g\colon\H q\to\H q'$ defined by the universal property of the horizontal bundle. The tangent bundle functor of $\HLC(\X,\TT)$ sends a pair $(\q,h)$ to $\T^\LC\q,h_\T)$ where:
\begin{align*}
&h_\T\colon H\T E=\T^2M\times_{\T M}\T E\xrightarrow{\<\pi_1c,\pi_2\>}\T(\T M\times_ME)=\T\H q\xrightarrow{\T h}\T^2E
\end{align*}
and a morphism $(f,g)\colon(\q,k)\to(\q',k')$ to $(\T f,\T g)$. Finally, the structural natural transformations of $\HLC(\X,\TT)$ are the same as the ones of $(\X,\TT)$.

\par Linear connections form also a tangent category denoted by $\LC(\X,\TT)$, whose objects are triples $(\q,k,h)$ formed by a display differential bundle $\q$ together with a linear connection $(k,h)$ on $\q$, and whose tangent structure acts on the vertical and the horizontal parts as in $\VLC(\X,\TT)$ and $\HLC(\X,\TT)$, respectively.

\par Similarly, vertical/horizontal affine connections form tangent categories, denoted by $\VAC(\X,\TT)$, $\HAC(\X,\TT)$, and $\AC(\X,\TT)$, respectively~\cite[Theorem~5.16]{blute:affine-spaces}.

\begin{lemma}
\label{lemma:LC-functoriality}
There are six $2$-functors
\begin{align*}
\VLC&\colon\TngCat_\Dsply\to\TngCat_\cong &\VAC&\colon\TngCat_\cong\to\TngCat_\cong\\
\HLC&\colon\TngCat_\Dsply\to\TngCat_\cong &\HAC&\colon\TngCat_\cong\to\TngCat_\cong\\
\LC&\colon\TngCat_\Dsply\to\TngCat_\cong  &\AC&\colon\TngCat_\cong\to\TngCat_\cong
\end{align*}
on the $2$-category of tangent categories and display-preserving strong tangent morphisms, which send a tangent category $(\X,\TT)$ to the tangent categories of (vertical/horizontal) linear/affine connections of $(\X,\TT)$.
\end{lemma}
\begin{proof}
In~\cite[Sections~5 and~6]{blute:affine-spaces}, the $2$-functoriality of $\AC$ is discussed; the proof of the $2$-functoriality of $\VAC$ and $\HAC$ is similar. To show the $2$-functoriality of $\VLC$, $\HLC$, and $\LC$, first, recall that, thanks to Lemma~\ref{lemma:LC-functoriality}, $\LC\colon\TngCat_\Dsply\to\TngCat_\Dsply$ is $2$-functorial. Consider now a display-preserving strong tangent morphism $(F,\alpha)\colon(\X,\TT)\to(\X',\TT')$. Thus, $\VLC(F,\alpha)$ sends a vertical linear connection $(\q\colon E\to M,k)$ to the vertical connection on the display differential bundle $\LC(F,\alpha)(\q)$ so defined:
\begin{align*}
&k_{(F,\alpha)}\T'FE\xrightarrow{\alpha^{-1}}F\T E\xrightarrow{Fk}FE
\end{align*}
Moreover, given a tangent natural transformation $\varphi\colon(F,\alpha)\to(F',\alpha')$, $\VLC(\varphi)$ is the tangent natural transformation:
\begin{align*}
&\VLC(\varphi)_{(\q,k)}\colon(\LC(F,\alpha),k_{(F,\alpha)})\xrightarrow{(\varphi_M\colon FM\to F'M,\varphi_E\colon FE\to F'E)}(\LC(F',\alpha'),k_{(F',\alpha')})\qedhere
\end{align*}
\end{proof}

%__________________________________________________________________________
\subsection{The universal property of linear connections}
\label{subsection:universal-property-linear-connections}
This section aims to characterize the correct universal property of the tangent categories $\VLC(\X,\TT)$, $\HLC(\X,\TT)$, and $\LC(\X,\TT)$ of the various flavours of linear connections on differential bundles of a tangent category $(\X,\TT)$. For starters, notice that, for each of these three tangent categories, there is a strict tangent morphism which sends each object to its underlying display differential bundle:
\begin{align*}
\U&\colon\VLC(\X,\TT)\to\DB(\X,\TT)\\
\U&\colon\HLC(\X,\TT)\to\DB(\X,\TT)\\
\U&\colon\LC(\X,\TT)\to\DB(\X,\TT)
\end{align*}
Thus, each of these strict tangent morphisms induces a functor from $\DB_\pt[\LC(\X,\TT)\|\X,\TT]$ to the respective tangent categories of pointwise display differential bundles $\DB_\pt[\VLC(\X,\TT)\|\X,\TT]$, $\DB_\pt[\HLC(\X,\TT)\|\X,\TT]$, and $\DB_\pt[\LC(\X,\TT)\|\X,\TT]$. In particular, the universal pointwise display differential bundle $\UnivQ\colon\Tot\to\Base$ of $(\X,\TT)$ is sent to a display differential bundle $\UnivQ_\U$ on each of the Hom-tangent categories $[\VLC(\X,\TT)\|\X,\TT]$, $[\HLC(\X,\TT)\|\X,\TT]$, and $[\LC(\X,\TT)\|\X,\TT]$, denoted by
\begin{align*}
&\UnivQ^\VLC\colon\Tot^\VLC\to\Base^\VLC\\
&\UnivQ^\HLC\colon\Tot^\HLC\to\Base^\HLC\\
&\UnivQ^\LC\colon\Tot^\LC\to\Base^\LC
\end{align*}
respectively.
\par Now, define the following natural transformations
\begin{align}
\label{equation:universal-linear-connections}
\begin{split}
\Univ{k^\VLC}_{(\q,k)}&\colon\bar\T\Tot^\VLC(\q,k)=\T E\xrightarrow{k}E=\Tot^\VLC(\q,k)\\
\Univ{h^\HLC}_{(\q,h)}&\colon(\bar\H\UnivQ^\HLC)(\q,h)=\H q\xrightarrow{h}\T E=\Tot^\HLC(\q,h)\\
\Univ{k^\LC}_{(\q;k,h)}&\colon\bar\T\Tot^\LC(\q;k,h)=\T E\xrightarrow{k}E=\Tot^\LC(\q;k,h)\\
\Univ{h^\LC}_{(\q;k,h)}&\colon(\bar\H\UnivQ^\LC)(\q;k,h)=\H q\xrightarrow{h}\T E=\Tot^\LC(\q;k,h)
\end{split}
\end{align}
where we denoted by $\bar\H\Univ{q^\HLC}$ and $\bar\H\Univ{q^\LC}$ the pointwise horizontal bundles of $\Univ{q^\HLC}$ and $\Univ{q^\LC}$, respectively, in the corresponding Hom-tangent categories, that is, the pointwise pullbacks of $\Univ{q^\HLC}$ and $\Univ{q^\LC}$ along the projection.

\begin{proposition}
\label{proposition:universality-linear-connections}
The natural transformations of Equation~\eqref{equation:universal-linear-connections} define a vertical linear connection $\Univ{k^\VLC}$ on $\UnivQ^\VLC$, a horizontal linear connection $\Univ{h^\HLC}$ on $\UnivQ^\HLC$, and a linear connection $(\Univ{k^\LC},\Univ{h^\LC}$ on $\UnivQ^\LC$, respectively.
\end{proposition}
\begin{proof}
We already mentioned that $\UnivQ^\VLC$, $\UnivQ^\HLC$, and $\UnivQ^\LC$ are pointwise display differential bundles since they are constructed from the universal pointwise display differential bundle. The naturality of $\Univ{k^\VLC}$, $\Univ{h^\HLC}$, $\Univ{k^\LC}$, and $\Univ{h^\LC}$ is equivalent to the compatibilities between the morphisms of $\VLC(\X,\TT)$, $\HLC(\X,\TT)$, and $\LC(\X,\TT)$ with the connections. Furthermore, these natural transformations are compatible with the tangent structures by definition of the tangent bundle functors in each respective tangent category. Finally, the axioms required for these transformations to define connections are a direct consequence of each $(\q,k)\in\VLC(\X,\TT)$, $(\q,h)\in\HLC(\X,\TT)$, and $(\q;k,h)\in\LC(\X,\TT)$ being vertical, horizontal, and linear connections in $(\X,\TT)$.
\end{proof}

To characterize the correct universal property enjoyed by the connections of Proposition~\ref{proposition:universality-linear-connections}, let us first unwrap the definition of a (vertical/horizontal) linear connection in each Hom-tangent category $[\X',\TT'\|\X'',\TT'']$, where $(\X',\TT')$ and $(\X'',\TT'')$ are tangent categories. Let us start with a vertical linear connection.
\begin{description}
\item[Differential bundle] A pointwise display differential bundle $\q\colon(K,\theta)\to(G,\beta)$;

\item[Vertical connection] The vertical linear connection consists of a tangent natural transformation:
\begin{align*}
&k\colon\bar\T(K,\theta)\to(K,\theta)
\end{align*}
\end{description}
satisfying the axioms of a vertical linear connection. A morphism of vertical linear connections in the Hom-tangent category $[\X,\TT\|\X',\TT']$ from a vertical linear connection $(\q\colon(K,\theta)\to(G,\beta),k)$ to a vertical linear connection $(\q'\colon(K',\theta')\to(G',\beta'),k')$ consists of a linear morphism of differential bundles $(\varphi,\psi)\colon\q\to\q'$ which commutes with the vertical linear connections.
\par Similarly, a horizontal linear connection consists of:
\begin{description}
\item[Differential bundle] A pointwise display differential bundle $\q\colon(K,\theta)\to(G,\beta)$;

\item[Horizontal connection] The horizontal linear connection consists of a tangent natural transformation:
\begin{align*}
&h\colon\bar\H\q\to\bar\T(K,\theta)
\end{align*}
\end{description}
satisfying the axioms of a horizontal linear connection. Notice that $\bar\H\q\colon\HLC(\X,\TT)\to(\X,\TT)$ is the strict tangent morphism which sends a horizontal linear connection $(\q,h)$ to the total space of the horizontal bundle $\H\q$ of $\q$.
\par A morphism of horizontal linear connections in the Hom-tangent category $[\X,\TT\|\X',\TT']$ from a horizontal linear connection $(\q\colon(K,\theta)\to(G,\beta),h)$ to a horizontal linear connection $(\q'\colon(K',\theta')\to(G',\beta'),h')$ consists of a linear morphism of differential bundles $(\varphi,\psi)\colon\q\to\q'$ which commutes with the horizontal linear connections.
\par Finally, a linear connection consists of:
\begin{description}
\item[Differential bundle] A pointwise display differential bundle $\q\colon(K,\theta)\to(G,\beta)$;

\item[Linear connection] The linear connection consists of a pair of tangent natural transformations:
\begin{align*}
&k\colon\bar\T(K,\theta)\to(K,\theta)\\
&h\colon\bar\H\q\to\bar\T(K,\theta)
\end{align*}
such that $(\q,k)$ and $(\q,h)$ are a vertical and a horizontal linear connections, respectively;
\end{description}
satisfying the axioms of a linear connection.
\par A morphism of linear connections in the Hom-tangent category $[\X,\TT\|\X',\TT']$ from a linear connection $(\q\colon(K,\theta)\to(G,\beta);k,h)$ to a linear connection $(\q'\colon(K',\theta')\to(G',\beta');k',h')$ consists of a linear morphism of differential bundles $(\varphi,\psi)\colon\q\to\q'$ which commutes with the vertical and the horizontal linear connections.
\par For each tangent category $(\X',\TT')$, a vertical linear connection $(\q,k)$ in $[\X'',\TT''\|\X''',\TT''']$ induces a functor
\begin{align*}
&\Gamma_{(\q,k)}\colon[\X',\TT'\|\X'',\TT'']\to\VLC[\X',\TT'\|\X''',\TT''']
\end{align*}
which sends a lax tangent morphism $(H,\gamma)\colon(\X',\TT')\to(\X'',\TT'')$ to the vertical linear connection
\begin{align*}
&\Gamma_{(\q,k)}(H,\gamma)\=(\Gamma_{\q}(H,\gamma),k_{(H,\gamma)})
\end{align*}
in the Hom-tangent category $[\X',\TT'\|\X''',\TT''']$. Concretely, the base pointwise display differential bundle is induced by the functor $\Gamma_\q$ of Lemma~\ref{lemma:induced-lax-tangent-morphism-from-differential-bundles} and $k_{(H,\gamma)}$ is the natural transformation:
\begin{equation*}
% https://q.uiver.app/#q=WzAsMyxbMCwwLCIoXFxYJyxcXFRUJykiXSxbMSwwLCIoXFxYJycsXFxUVCcnKSJdLFszLDAsIihcXFgnJycsXFxUVCcnJykiXSxbMCwxLCIoSCxcXGdhbW1hKSJdLFsxLDIsIihLLFxcdGhldGEpIiwyLHsiY3VydmUiOjN9XSxbMSwyLCJcXGJhclxcVChLLFxcdGhldGEpIiwwLHsiY3VydmUiOi0zfV0sWzUsNCwiayIsMCx7InNob3J0ZW4iOnsic291cmNlIjoyMCwidGFyZ2V0IjoyMH19XV0=
\begin{tikzcd}
{(\X',\TT')} & {(\X'',\TT'')} && {(\X''',\TT''')}
\arrow["{(H,\gamma)}", from=1-1, to=1-2]
\arrow[""{name=0, anchor=center, inner sep=0}, "{(K,\theta)}"', curve={height=18pt}, from=1-2, to=1-4]
\arrow[""{name=1, anchor=center, inner sep=0}, "{\bar\T(K,\theta)}", curve={height=-18pt}, from=1-2, to=1-4]
\arrow["k",shorten <=5pt, shorten >=5pt, Rightarrow, from=1, to=0]
\end{tikzcd}
\end{equation*}

Similarly, a horizontal linear connection $(\q,h)$ in the Hom-tangent category $[\X'',\TT''\|\X''',\TT''']$ induces a functor
\begin{align*}
&\Gamma_{(\q,h)}\colon[\X',\TT'\|\X'',\TT'']\to\HLC[\X',\TT'\|\X''',\TT''']
\end{align*}
which sends a lax tangent morphism $(H,\gamma)\colon(\X',\TT')\to(\X'',\TT'')$ to the horizontal linear connection
\begin{align*}
&\Gamma_{(\q,h)}(H,\gamma)\=(\Gamma_{\q}(H,\gamma),h_{(H,\gamma)})
\end{align*}
in the Hom-tangent category $[\X',\TT'\|\X''',\TT''']$. Concretely, the base pointwise display differential bundle is induced by the functor $\Gamma_\q$ of Lemma~\ref{lemma:induced-lax-tangent-morphism-from-differential-bundles} and $h_{(H,\gamma)}$ is the natural transformation:
\begin{equation*}
% https://q.uiver.app/#q=WzAsMyxbMCwwLCIoXFxYJyxcXFRUJykiXSxbMSwwLCIoXFxYJycsXFxUVCcnKSJdLFszLDAsIihcXFgnJycsXFxUVCcnJykiXSxbMCwxLCIoSCxcXGdhbW1hKSJdLFsxLDIsIlxcYmFyXFxIXFxxIiwyLHsiY3VydmUiOjN9XSxbMSwyLCJcXGJhclxcVChLLFxcdGhldGEpIiwwLHsiY3VydmUiOi0zfV0sWzQsNSwiaCIsMix7InNob3J0ZW4iOnsic291cmNlIjoyMCwidGFyZ2V0IjoyMH19XV0=
\begin{tikzcd}
{(\X',\TT')} & {(\X'',\TT'')} && {(\X''',\TT''')}
\arrow["{(H,\gamma)}", from=1-1, to=1-2]
\arrow[""{name=0, anchor=center, inner sep=0}, "{\bar\H\q}"', curve={height=18pt}, from=1-2, to=1-4]
\arrow[""{name=1, anchor=center, inner sep=0}, "{\bar\T(K,\theta)}", curve={height=-18pt}, from=1-2, to=1-4]
\arrow["h"',shorten <=5pt, shorten >=5pt, Rightarrow, from=0, to=1]
\end{tikzcd}
\end{equation*}
Furthermore, a  linear connection $(\q;k,h)$ in the Hom-tangent category $[\X'',\TT''\|\X''',\TT''']$ induces a functor
\begin{align*}
&\Gamma_{(\q;k,h)}\colon[\X',\TT'\|\X'',\TT'']\to\LC[\X',\TT'\|\X''',\TT''']
\end{align*}
which sends a lax tangent morphism $(H,\gamma)\colon(\X',\TT')\to(\X'',\TT'')$ to the horizontal linear connection
\begin{align*}
&\Gamma_{(\q;k,h)}(H,\gamma)\=(\Gamma_{\q}(H,\gamma);k_{H,\gamma},h_{(H,\gamma)})
\end{align*}
in the Hom-tangent category $[\X',\TT'\|\X''',\TT''']$. These constructions are functorial, as shown by the next lemma.

\begin{lemma}
\label{lemma:induced-lax-tangent-morphism-from-linear-connections}
Given a pointwise display differential bundle $\q\colon(K,\theta)\to(G,\beta)$ in the Hom-tangent category $[\X'',\TT''\|\X''',\TT''']$, a vertical linear connection $k$ on $\q$, a horizontal linear connection $h$ on $\q$, and a linear connection $(k,h)$ on $\q$ induce lax tangent morphisms
\begin{align*}
\Gamma_{(\q,k)}&\colon[\X',\TT'\|\X'',\TT'']\to\VLC[\X',\TT'\|\X''',\TT''']\\
\Gamma_{(\q,h)}&\colon[\X',\TT'\|\X'',\TT'']\to\HLC[\X',\TT'\|\X''',\TT''']\\
\Gamma_{(\q;k,h)}&\colon[\X',\TT'\|\X'',\TT'']\to\LC[\X',\TT'\|\X''',\TT''']
\end{align*}
for each tangentad $(\X',\TT')$, which are natural in $(\X',\TT')$ and strong (strict) when both $(G,\beta)$ and $(K,\theta)$ are strong (strict).
\end{lemma}
\begin{proof}
The proof of this lemma is fairly similar to that of Lemma~\ref{lemma:induced-lax-tangent-morphism-from-differential-objects}. We leave it to the reader to complete the details.
\end{proof}

By Lemma~\ref{lemma:induced-lax-tangent-morphism-from-linear-connections}, the vertical linear connection $\Univ{k^\VLC}$ on $\UnivQ^\VLC$, the horizontal linear connection $\Univ{h^\HLC}$ on $\UnivQ^\HLC$, and the linear connection $(\Univ{k^\LC},\Univ{h^\LC})$ on $\UnivQ^\LC$, of Proposition~\ref{proposition:universality-linear-connections} induce three strict tangent natural transformations
\begin{align*}
\Gamma_{\Univ{\UnivQ^\VLC,k^\VLC}}&\colon[\X',\TT'\|\VLC(\X,\TT)]\to\VLC[\X',\TT'\|\X,\TT]\\
\Gamma_{\Univ{\q^\HLC,h^\HLC}}&\colon[\X',\TT'\|\HLC(\X,\TT)]\to\HLC[\X',\TT'\|\X,\TT]\\
\Gamma_{\Univ{\q^\LC,k^\LC,h^\LC}}&\colon[\X',\TT'\|\LC(\X,\TT)]\to\LC[\X',\TT'\|\X,\TT]
\end{align*}
natural in $(\X',\TT')$.
\par The goal is to prove that these three transformations are invertible. Consider another tangent category $(\X',\TT')$ together with a pointwise display differential bundle $\q\colon(K,\theta)\to(G,\beta)$ in the Hom-tangent category $[\X',\TT'\|\X,\TT]$ equipped with a vertical connection $k$, with a horizontal connection $h$, and a linear connection $(k,h)$ (each case to be taken separately). For every $A\in\X'$, the tuples
\begin{align*}
&(\q_A\colon KA\to GA,k_A\colon\T KA\to KA)\\
&(\q_A\colon KA\to GA,h\colon\H{\q_A}\to\T KA)\\
&(\q_A\colon KA\to GA;k_A\colon\T KA\to KA,h\colon\H{\q_A}\to\T KA)\\
\end{align*}
are a vertical linear connection, a horizontal linear connection, and a linear connection in $(\X,\TT)$. Therefore, we can define the functors
\begin{align*}
\Lambda[\q,k]&\colon(\X',\TT')\to\VLC(\X,\TT)\\
\Lambda[\q,h]&\colon(\X',\TT')\to\HLC(\X,\TT)\\
\Lambda[\q;k,h]&\colon(\X',\TT')\to\LC(\X,\TT)\\
\end{align*}
which send an object $A$ of $\X'$ to the connections $(\q_A,k_A)$, $(\q_A,h_A)$, and $(\q_A;k_A,h_A)$, respectively, and a morphism $f\colon A\to B$ to $(Gf,Kf)$. By the naturality of $k$, $h$, and the structure morphisms of $\q$, $(Gf,Kf)$ becomes a morphism of (vertical/horizontal) connections of $(\X,\TT)$. Furthermore, $\Lambda[\q,k]$, $\Lambda[\q,h]$, and $\Lambda[\q;k,h]$ come with a distributive law induced by $\beta$ and $\theta$.

\begin{lemma}
\label{lemma:universality-linear-connections}
Given a pointwise display differential bundle structure $\q\colon(K,\theta)\to(G,\beta)$ in the Hom-tangent category $[\X',\TT'\|\X,\TT]$, equipped with a vertical linear connection $k$, a horizontal linear connection $h$, or a linear connection $(k,h)$, the functors $\Lambda[\q,k]$, $\Lambda[\q,h]$, and $\Lambda[\q;k,h]$ define lax tangent morphisms:
\begin{align*}
\Lambda[\q,k]&\colon(\X',\TT')\to\VLC(\X,\TT)\\
\Lambda[\q,h]&\colon(\X',\TT')\to\HLC(\X,\TT)\\
\Lambda[\q;k,h]&\colon(\X',\TT')\to\LC(\X,\TT)\\
\end{align*}
\end{lemma}
\begin{proof}
In the previous discussion, we defined $\Lambda[\q,k]$, $\Lambda[\q,h]$, and $\Lambda[\q;k,h]$ as the lax tangent morphisms which pick out the connections $(\q_A,k_A)$, $(\q_A,h_A)$, and $(\q_A;k_A,h_A)$ for each $A\in\X'$, respectively, and whose distributive law is induced by $\beta$ and $\theta$. In particular, the naturality of the structure morphisms of $\q$ and of the connections makes these functors well-defined, the compatibility of the structure morphisms, $\beta$ and $\theta$, makes them into morphisms of connections, and the compatibility between $\beta$ and $\theta$ and the tangent structures make them into lax tangent morphisms.
\end{proof}

We can now prove the universal property of vertical, horizontal, and linear connections, which is the main result of this section.

\begin{theorem}
\label{theorem:universality-linear-connections}
The vertical linear connection $\Univ{\UnivQ^\VLC,k^\VLC}$, the horizontal linear connection $\Univ{\q^\HLC,h^\HLC}$, and the linear connection $\Univ{\q^\LC;k^\LC,h^\LC}$ of Proposition~\ref{proposition:universality-linear-connections} are universal. Concretely, the induced strict tangent natural transformations
\begin{align*}
\Gamma_{\Univ{\UnivQ^\VLC,k^\VLC}}&\colon[\X',\TT'\|\VLC(\X,\TT)]\to\VLC[\X',\TT'\|\X,\TT]\\
\Gamma_{\Univ{\q^\HLC,h^\HLC}}&\colon[\X',\TT'\|\HLC(\X,\TT)]\to\HLC[\X',\TT'\|\X,\TT]\\
\Gamma_{\Univ{\q^\LC,k^\LC,h^\LC}}&\colon[\X',\TT'\|\LC(\X,\TT)]\to\LC[\X',\TT'\|\X,\TT]
\end{align*}
make the functors
\begin{align*}
&\TngCat^\op\xrightarrow{[-\|\X,\TT]}\TngCat^\op\xrightarrow{\VLC^\op}\TngCat^\op\\
&\TngCat^\op\xrightarrow{[-\|\X,\TT]}\TngCat^\op\xrightarrow{\HLC^\op}\TngCat^\op\\
&\TngCat^\op\xrightarrow{[-\|\X,\TT]}\TngCat^\op\xrightarrow{\LC^\op}\TngCat^\op
\end{align*}
which send a tangent category $(\X',\TT')$ to the tangent categories $\VLC[\X',\TT'\|\X,\TT]$, $\HLC[\X',\TT'\|\X,\TT]$, and $\LC[\X',\TT'\|\X,\TT]$, respectively, into corepresentable functors. In particular, $\Gamma_{\Univ{\UnivQ^\VLC,k^\VLC}}$, $\Gamma_{\Univ{\q^\HLC,h^\HLC}}$, and $\Gamma_{\Univ{\q^\LC,k^\LC,h^\LC}}$ are invertible.
\end{theorem}
\begin{proof}
The proof of this theorem is fairly similar to that of Theorem~\ref{theorem:universality-differential-objects}.
\end{proof}

Theorem~\ref{theorem:universality-linear-connections} establishes the correct universal property of vertical, horizontal, and linear connections. Thanks to this result, we can finally introduce the notion of vertical, horizontal, and linear connections in the formal context of tangentads. Let us start with vertical connections.

\begin{definition}
\label{definition:construction-vertical-linear-connections}
A tangentad $(\X,\TT)$ in a $2$-category $\CC$ \textbf{admits the construction of vertical linear connections} if there exists a tangentad $\VLC(\X,\TT)$ of $\CC$ together with a pointwise display tangent bundle $\UnivQ^\VLC\colon\Tot^\VLC\to\Base^\VLC$ in the Hom-tangent category $[\VLC(\X,\TT)\|\X,\TT]$ and a vertical linear connection $\Univ{k^\VLC}$ on $\UnivQ^\VLC$ such that the induced tangent natural transformation $\Gamma_{\UnivQ^\VLC,\Univ{k^\VLC}}$ of Lemma~\ref{lemma:induced-lax-tangent-morphism-from-linear-connections} is invertible. The vertical linear connection $\Univ{\UnivQ^\VLC,k^\VLC}\=(\UnivQ^\VLC,\Univ{k^\VLC})$ is called the \textbf{universal vertical linear connection} of $(\X,\TT)$ and $\VLC(\X,\TT)$ is called the \textbf{tangentad of vertical linear connections} of $(\X,\TT)$.
\end{definition}

\begin{definition}
\label{definition:construction-vertical-linear-connections-2-category}
A $2$-category $\CC$ \textbf{admits the construction of vertical linear connections} provided that every tangentad of $\CC$ admits the construction of vertical linear connections.
\end{definition}

The next step is to introduce horizontal linear connections.

\begin{definition}
\label{definition:construction-horizontal-linear-connections}
A tangentad $(\X,\TT)$ in a $2$-category $\CC$ \textbf{admits the construction of horizontal linear connections} if there exists a tangentad $\HLC(\X,\TT)$ of $\CC$ together with a pointwise display tangent bundle $\UnivQ^\HLC\colon\Tot^\HLC\to\Base^\HLC$ in the Hom-tangent category $[\HLC(\X,\TT)\|\X,\TT]$ and a horizontal linear connection $\Univ{h^\HLC}$ on $\UnivQ^\HLC$ such that the induced tangent natural transformation $\Gamma_{\UnivQ^\HLC,\Univ{h^\HLC}}$ of Lemma~\ref{lemma:induced-lax-tangent-morphism-from-linear-connections} is invertible. The horizontal linear connection $\Univ{\q^\HLC,h^\HLC}\=(\UnivQ^\HLC,\Univ{h^\HLC})$ is called the \textbf{universal horizontal linear connection} of $(\X,\TT)$ and $\HLC(\X,\TT)$ is called the \textbf{tangentad of horizontal linear connections} of $(\X,\TT)$.
\end{definition}

\begin{definition}
\label{definition:construction-horizontal-linear-connections-2-category}
A $2$-category $\CC$ \textbf{admits the construction of horizontal linear connections} provided that every tangentad of $\CC$ admits the construction of horizontal linear connections.
\end{definition}

Finally, we can introduce linear connections.

\begin{definition}
\label{definition:construction-linear-connections}
A tangentad $(\X,\TT)$ in a $2$-category $\CC$ \textbf{admits the construction of linear connections} if there exists a tangentad $\LC(\X,\TT)$ of $\CC$ together with a pointwise display tangent bundle $\UnivQ^\LC\colon\Tot^\LC\to\Base^\LC$ in the Hom-tangent category $[\LC(\X,\TT)\|\X,\TT]$ and a linear connection $(\Univ{k^\LC},\Univ{h^\LC})$ on $\UnivQ^\LC$ such that the induced tangent natural transformation $\Gamma_{\UnivQ^\LC;\Univ{k^\LC},\Univ{h^\LC}}$ of Lemma~\ref{lemma:induced-lax-tangent-morphism-from-linear-connections} is invertible. The linear connection $\Univ{\q^\LC;k^\LC,h^\LC}\=(\UnivQ^\LC;\Univ{k^\LC},\Univ{h^\LC})$ is called the \textbf{universal linear connection} of $(\X,\TT)$ and $\LC(\X,\TT)$ is called the \textbf{tangentad of linear connections} of $(\X,\TT)$.
\end{definition}

\begin{definition}
\label{definition:construction-linear-connections-2-category}
A $2$-category $\CC$ \textbf{admits the construction of linear connections} provided that every tangentad of $\CC$ admits the construction of linear connections.
\end{definition}

We can now rephrase Theorem~\ref{theorem:universality-linear-connections} as follows.

\begin{corollary}
\label{corollary:universality-linear-connections}
The $2$-category $\Cat$ of categories admits the constructions of vertical, horizontal, and linear connections and the tangentads of vertical, horizontal, and linear connections of a tangentad $(\X,\TT)$ of $\Cat$ are the tangent categories $\VLC(\X,\TT)$, $\HLC(\X,\TT)$, and $\LC(\X,\TT)$ of vertical, horizontal, and linear connections of $(\X,\TT)$, respectively.
\end{corollary}

To ensure that Definitions~\ref{definition:construction-vertical-linear-connections},~\ref{definition:construction-horizontal-linear-connections}, and~\ref{definition:construction-linear-connections} are well-posed, one requires the tangentads of vertical, horizontal, and linear connections of a given tangentad to be unique. The next proposition establishes that such constructions are defined uniquely up to a unique isomorphism.

\begin{proposition}
\label{proposition:uniqueness-linear-connections}
If a tangentad $(\X,\TT)$ admits the construction of (vertical/horizontal) linear connections, the tangentad of (vertical/horizontal) linear connections of $(\X,\TT)$ is unique up to a unique isomorphism which extends to an isomorphism of the corresponding universal (vertical/horizontal) linear connections of $(\X,\TT)$.
\end{proposition}
\begin{proof}
The proof of this proposition is fairly similar to that of Proposition~\ref{proposition:uniqueness-differential-objects}. We leave it to the reader to complete the details.
\end{proof}

%__________________________________________________________________________
\subsection{The universal property of affine connections}
\label{subsection:universal-property-affine-connections}
This section aims to characterize the correct universal property of the tangent categories $\VAC(\X,\TT)$, $\HAC(\X,\TT)$, and $\AC(\X,\TT)$ of the various flavours of affine connections on differential bundles of a tangent category $(\X,\TT)$. For starters, notice that, for each of these three tangent categories, there is a strict tangent morphism which forgets the connection:
\begin{align*}
\U^\VAC&\colon\VAC(\X,\TT)\to(\X,\TT)\\
\U^\HAC&\colon\HAC(\X,\TT)\to(\X,\TT)\\
\U^\AC&\colon\AC(\X,\TT)\to(\X,\TT)
\end{align*}
Define the following natural transformations:
\begin{align}
\label{equation:universal-affine-connections}
\begin{split}
\Univ{k^\VAC}_{(M,k)}&\colon\bar\T^2\U^\VAC(M,k)=\T^2M\xrightarrow{k}\T M=\bar\T\U^\VAC(M,k)\\
\Univ{h^\HAC}_{(M,h)}&\colon\bar\T_2\U^\HAC(M,h)=\T_2M\xrightarrow{h}\T^2M=\bar\T^2\U^\HAC(M,h)\\
\Univ{k^\AC}_{(M;k,h)}&\colon\bar\T^2\U^\AC(M;k,h)=\T^2M\xrightarrow{k}\T M=\bar\T\U^\AC(M;k,h)\\
\Univ{h^\AC}_{(M;k,h)}&\colon\bar\T_2\U^\AC(M;k,h)=\T_2M\xrightarrow{h}\T^2M=\bar\T^2\U^\AC(M;k,h)
\end{split}
\end{align}

\begin{proposition}
\label{proposition:universality-affine-connections}
The natural transformations of Equation~\eqref{equation:universal-affine-connections} define a vertical affine connection $\Univ{k^\VAC}$ on $\U^\VAC$, a horizontal affine connection $\Univ{h^\HAC}$ on $\U^\HAC$, and an affine connection $(\Univ{k^\AC},\Univ{h^\AC}$ on $\U^\AC$, respectively.
\end{proposition}
\begin{proof}
We already mentioned that $\U^\VAC$, $\U^\HAC$, and $\U^\AC$ are pointwise display differential bundles since they are constructed from the universal pointwise display differential bundle. The naturality of $\Univ{k^\VAC}$, $\Univ{h^\HAC}$, $\Univ{k^\AC}$, and $\Univ{h^\AC}$ is equivalent to the compatibilities between the morphisms of $\VAC(\X,\TT)$, $\HAC(\X,\TT)$, and $\AC(\X,\TT)$ with the connections. Furthermore, these natural transformations are compatible with the tangent structures by definition of the tangent bundle functors in each respective tangent category. Finally, the axioms required for these transformations to define connections are a direct consequence of each $(M,k)\in\VAC(\X,\TT)$, $(M,h)\in\HAC(\X,\TT)$, and $(M;k,h)\in\AC(\X,\TT)$ being vertical, horizontal, and affine connections in $(\X,\TT)$.
\end{proof}

To characterize the correct universal property enjoyed by the connections of Proposition~\ref{proposition:universality-affine-connections}, let us first unwrap the definition of a (vertical/horizontal) affine connection in each Hom-tangent category $[\X',\TT'\|\X'',\TT'']$, where $(\X',\TT')$ and $(\X'',\TT'')$ are tangent categories. Let us start with a vertical affine connection.
\begin{description}
\item[Base object] A lax tangent morphism $(G,\beta)\colon(\X',\TT')\to(\X'',\TT'')$;

\item[Vertical connection] The vertical affine connection consists of a tangent natural transformation:
\begin{align*}
&k\colon\bar\T^2(G,\beta)\to\bar\T(G,\beta)
\end{align*}
\end{description}
satisfying the axioms of a vertical affine connection. A morphism of vertical affine connections in the Hom-tangent category $[\X,\TT\|\X',\TT']$ from a vertical affine connection $(G,\beta;k)$ to a vertical affine connection $(G',\beta';k')$ consists of a tangent natural transformation $\varphi\colon(G,\beta)\Rightarrow(G',\beta')$ which commutes with the vertical affine connections.
\par Similarly, a horizontal affine connection consists of:
\begin{description}
\item[Base object] A lax tangent morphism $(G,\beta)\colon(\X',\TT')\to(\X'',\TT'')$;

\item[Horizontal connection] The horizontal affine connection consists of a tangent natural transformation:
\begin{align*}
&h\colon\bar\T_2(G,\beta)\to\bar\T(G,\beta)
\end{align*}
\end{description}
satisfying the axioms of a horizontal affine connection.
\par A morphism of horizontal affine connections in the Hom-tangent category $[\X,\TT\|\X',\TT']$ from a horizontal affine connection $(G,\beta;h)$ to a horizontal affine connection $(G',\beta';h')$ consists of a tangent natural transformation $\varphi\colon(G,\beta)\Rightarrow(G',\beta')$ which commutes with the horizontal affine connections.
\par Finally, an affine connection consists of:
\begin{description}
\item[Base object] A lax tangent morphism $(G,\beta)\colon(\X',\TT')\to(\X'',\TT'')$;

\item[Affine connection] The affine connection consists of a pair of tangent natural transformations:
\begin{align*}
&k\colon\bar\T^2(G,\beta)\to\bar\T(G,\beta)\\
&h\colon\bar\T_2(G,\beta)\to\bar\T^2(G,\beta)
\end{align*}
such that $(G,\beta;k)$ and $(G,\beta;h)$ are a vertical and a horizontal affine connections, respectively;
\end{description}
satisfying the axioms of an affine connection.
\par A morphism of affine connections in the Hom-tangent category $[\X,\TT\|\X',\TT']$ from an affine connection $(G,\beta;k,h)$ to an affine connection $(G',\beta';k',h')$ consists of a tangent natural transformation $\varphi\colon(G,\beta)\Rightarrow(G',\beta')$ which commutes with the vertical and the horizontal affine connections.
\par For each tangent category $(\X',\TT')$, a vertical affine connection $(G,\beta;k)$ in $[\X'',\TT''\|\X''',\TT''']$ induces a functor
\begin{align*}
&\Gamma_{(G,\beta;k)}\colon[\X',\TT'\|\X'',\TT'']\to\VAC[\X',\TT'\|\X''',\TT''']
\end{align*}
which sends a lax tangent morphism $(H,\gamma)\colon(\X',\TT')\to(\X'',\TT'')$ to the vertical affine connection
\begin{align*}
&\Gamma_{(G,\beta;k)}(H,\gamma)\=((H,\gamma)\o\U^\VAC,k_{(H,\gamma)})
\end{align*}
in the Hom-tangent category $[\X',\TT'\|\X''',\TT''']$. Concretely, $k_{(H,\gamma)}$ is the natural transformation:
\begin{equation*}
% https://q.uiver.app/#q=WzAsMyxbMCwwLCIoXFxYJyxcXFRUJykiXSxbMSwwLCIoXFxYJycsXFxUVCcnKSJdLFszLDAsIihcXFgnJycsXFxUVCcnJykiXSxbMCwxLCIoSCxcXGdhbW1hKSJdLFsxLDIsIihLLFxcdGhldGEpIiwyLHsiY3VydmUiOjN9XSxbMSwyLCJcXGJhclxcVChLLFxcdGhldGEpIiwwLHsiY3VydmUiOi0zfV0sWzUsNCwiayIsMCx7InNob3J0ZW4iOnsic291cmNlIjoyMCwidGFyZ2V0IjoyMH19XV0=
\begin{tikzcd}
{(\X',\TT')} & {(\X'',\TT'')} && {(\X''',\TT''')}
\arrow["{(H,\gamma)}", from=1-1, to=1-2]
\arrow[""{name=0, anchor=center, inner sep=0}, "{(G,\beta)}"', curve={height=18pt}, from=1-2, to=1-4]
\arrow[""{name=1, anchor=center, inner sep=0}, "{\bar\T(G,\beta)}", curve={height=-18pt}, from=1-2, to=1-4]
\arrow["k",shorten <=5pt, shorten >=5pt, Rightarrow, from=1, to=0]
\end{tikzcd}
\end{equation*}

Similarly, a horizontal affine connection $(G,\beta;h)$ in the Hom-tangent category $[\X'',\TT''\|\X''',\TT''']$ induces a functor
\begin{align*}
&\Gamma_{(G,\beta;h)}\colon[\X',\TT'\|\X'',\TT'']\to\HAC[\X',\TT'\|\X''',\TT''']
\end{align*}
which sends a lax tangent morphism $(H,\gamma)\colon(\X',\TT')\to(\X'',\TT'')$ to the horizontal affine connection
\begin{align*}
&\Gamma_{(G,\beta;h)}(H,\gamma)\=((H,\gamma)\o\U^\HAC,h_{(H,\gamma)})
\end{align*}
in the Hom-tangent category $[\X',\TT'\|\X''',\TT''']$. Concretely, $h_{(H,\gamma)}$ is the natural transformation:
\begin{equation*}
% https://q.uiver.app/#q=WzAsMyxbMCwwLCIoXFxYJyxcXFRUJykiXSxbMSwwLCIoXFxYJycsXFxUVCcnKSJdLFszLDAsIihcXFgnJycsXFxUVCcnJykiXSxbMCwxLCIoSCxcXGdhbW1hKSJdLFsxLDIsIlxcYmFyXFxIXFxxIiwyLHsiY3VydmUiOjN9XSxbMSwyLCJcXGJhclxcVChLLFxcdGhldGEpIiwwLHsiY3VydmUiOi0zfV0sWzQsNSwiaCIsMix7InNob3J0ZW4iOnsic291cmNlIjoyMCwidGFyZ2V0IjoyMH19XV0=
\begin{tikzcd}
{(\X',\TT')} & {(\X'',\TT'')} && {(\X''',\TT''')}
\arrow["{(H,\gamma)}", from=1-1, to=1-2]
\arrow[""{name=0, anchor=center, inner sep=0}, "{\bar\T_2(G,\beta)}"', curve={height=18pt}, from=1-2, to=1-4]
\arrow[""{name=1, anchor=center, inner sep=0}, "{\bar\T(G,\beta)}", curve={height=-18pt}, from=1-2, to=1-4]
\arrow["h"',shorten <=5pt, shorten >=5pt, Rightarrow, from=0, to=1]
\end{tikzcd}
\end{equation*}
Furthermore,a affine connection $(\q;k,h)$ in the Hom-tangent category $[\X'',\TT''\|\X''',\TT''']$ induces a functor
\begin{align*}
&\Gamma_{(\q;k,h)}\colon[\X',\TT'\|\X'',\TT'']\to\AC[\X',\TT'\|\X''',\TT''']
\end{align*}
which sends a lax tangent morphism $(H,\gamma)\colon(\X',\TT')\to(\X'',\TT'')$ to the horizontal affine connection
\begin{align*}
&\Gamma_{(\q;k,h)}(H,\gamma)\=((H,\gamma)\o\U^\AC;k_{H,\gamma},h_{(H,\gamma)})
\end{align*}
in the Hom-tangent category $[\X',\TT'\|\X''',\TT''']$. These constructions are functorial, as shown by the next lemma.

\begin{lemma}
\label{lemma:induced-lax-tangent-morphism-from-affine-connections}
Given an object $(G,\beta)$ in the Hom-tangent category $[\X'',\TT''\|\X''',\TT''']$, a vertical affine connection $k$ on $\q$, a horizontal affine connection $h$ on $\q$, and an affine connection $(k,h)$ on $\q$ induce lax tangent morphisms
\begin{align*}
\Gamma_{(G,\beta;k)}&\colon[\X',\TT'\|\X'',\TT'']\to\VAC[\X',\TT'\|\X''',\TT''']\\
\Gamma_{(G,\beta;h)}&\colon[\X',\TT'\|\X'',\TT'']\to\HAC[\X',\TT'\|\X''',\TT''']\\
\Gamma_{(\q;k,h)}&\colon[\X',\TT'\|\X'',\TT'']\to\AC[\X',\TT'\|\X''',\TT''']
\end{align*}
for each tangentad $(\X',\TT')$, which are natural in $(\X',\TT')$ and strong (strict) when both $(G,\beta)$ and $(G,\beta)$ are strong (strict).
\end{lemma}
\begin{proof}
The proof of this lemma is fairly similar to that of Lemma~\ref{lemma:induced-lax-tangent-morphism-from-differential-objects}. We leave it to the reader to complete the details.
\end{proof}

By Lemma~\ref{lemma:induced-lax-tangent-morphism-from-affine-connections}, the vertical affine connection $\Univ{k^\VAC}$ on $\U^\VAC$, the horizontal affine connection $\Univ{h^\HAC}$ on $\U^\HAC$, and the affine connection $(\Univ{k^\AC},\Univ{h^\AC})$ on $\U^\AC$, of Proposition~\ref{proposition:universality-affine-connections} induce three strict tangent natural transformations
\begin{align*}
\Gamma_{\Univ{\U^\VAC,k^\VAC}}&\colon[\X',\TT'\|\VAC(\X,\TT)]\to\VAC[\X',\TT'\|\X,\TT]\\
\Gamma_{\Univ{\U^\HAC,h^\HAC}}&\colon[\X',\TT'\|\HAC(\X,\TT)]\to\HAC[\X',\TT'\|\X,\TT]\\
\Gamma_{\Univ{\U^\AC,k^\AC,h^\AC}}&\colon[\X',\TT'\|\AC(\X,\TT)]\to\AC[\X',\TT'\|\X,\TT]
\end{align*}
natural in $(\X',\TT')$.
\par The goal is to prove that these three transformations are invertible. Consider another tangent category $(\X',\TT')$ together with an object $(G,\beta)$ in the Hom-tangent category $[\X',\TT'\|\X,\TT]$ equipped with a vertical connection $k$, with a horizontal connection $h$, and an affine connection $(k,h)$ (each case to be taken separately). For every $A\in\X'$, the tuples
\begin{align*}
&(GA,k_A\colon\T^2A\to\T GA)\\
&(GA,h\colon\H{ GA}\to\T^2A)\\
&(GA;k_A\colon\T^2A\to\T GA,h\colon\H{ GA}\to\T^2A)\\
\end{align*}
are a vertical affine connection, a horizontal affine connection, and an affine connection in $(\X,\TT)$. Therefore, we can define the functors
\begin{align*}
\Lambda[G,\beta;k]&\colon(\X',\TT')\to\VAC(\X,\TT)\\
\Lambda[G,\beta;h]&\colon(\X',\TT')\to\HAC(\X,\TT)\\
\Lambda[\q;k,h]&\colon(\X',\TT')\to\AC(\X,\TT)\\
\end{align*}
which send an object $A$ of $\X'$ to the connections $(GA,k_A)$, $(GA,h_A)$, and $(GA;k_A,h_A)$, respectively, and a morphism $f\colon A\to B$ to $Gf$. By the naturality of $k$, $h$, and the structure morphisms of $\q$, $Gf$ becomes a morphism of (vertical/horizontal) connections of $(\X,\TT)$. Furthermore, $\Lambda[G,\beta;k]$, $\Lambda[G,\beta;h]$, and $\Lambda[\q;k,h]$ come with a distributive law induced by $\beta$ and $\theta$.

\begin{lemma}
\label{lemma:universality-affine-connections}
Given an object $(G,\beta)$ in the Hom-tangent category $[\X',\TT'\|\X,\TT]$, equipped with a vertical affine connection $k$, a horizontal affine connection $h$, or an affine connection $(k,h)$, the functors $\Lambda[G,\beta;k]$, $\Lambda[G,\beta;h]$, and $\Lambda[G,\beta;k,h]$ define lax tangent morphisms:
\begin{align*}
\Lambda[G,\beta;k]&\colon(\X',\TT')\to\VAC(\X,\TT)\\
\Lambda[G,\beta;h]&\colon(\X',\TT')\to\HAC(\X,\TT)\\
\Lambda[G,\beta;k,h]&\colon(\X',\TT')\to\AC(\X,\TT)\\
\end{align*}
\end{lemma}
\begin{proof}
In the previous discussion, we defined $\Lambda[G,\beta;k]$, $\Lambda[G,\beta;h]$, and $\Lambda[G,\beta;k,h]$ as the lax tangent morphisms which pick out the connections $(GA,k_A)$, $(GA,h_A)$, and $(GA;k_A,h_A)$ for each $A\in\X'$, respectively, and whose distributive law is induced by $\beta$. In particular, the naturality of the structure morphisms of $\q$ and of the connections makes these functors well-defined, the compatibility of the structure morphisms, $\beta$, makes them into morphisms of connections, and the compatibility between $\beta$ and the tangent structures makes them into lax tangent morphisms.
\end{proof}

We can now prove the universal property of vertical, horizontal, and affine connections, which is the main result of this section.

\begin{theorem}
\label{theorem:universality-affine-connections}
The vertical affine connection $\Univ{\U^\VAC,k^\VAC}$, the horizontal affine connection $\Univ{\U^\HAC,h^\HAC}$, and the affine connection $\Univ{\U^\AC;k^\AC,h^\AC}$ of Proposition~\ref{proposition:universality-affine-connections} are universal. Concretely, the induced strict tangent natural transformations
\begin{align*}
\Gamma_{\Univ{\U^\VAC,k^\VAC}}&\colon[\X',\TT'\|\VAC(\X,\TT)]\to\VAC[\X',\TT'\|\X,\TT]\\
\Gamma_{\Univ{\U^\HAC,h^\HAC}}&\colon[\X',\TT'\|\HAC(\X,\TT)]\to\HAC[\X',\TT'\|\X,\TT]\\
\Gamma_{\Univ{\U^\AC,k^\AC,h^\AC}}&\colon[\X',\TT'\|\AC(\X,\TT)]\to\AC[\X',\TT'\|\X,\TT]
\end{align*}
make the functors
\begin{align*}
&\TngCat^\op\xrightarrow{[-\|\X,\TT]}\TngCat^\op\xrightarrow{\VAC^\op}\TngCat^\op\\
&\TngCat^\op\xrightarrow{[-\|\X,\TT]}\TngCat^\op\xrightarrow{\HAC^\op}\TngCat^\op\\
&\TngCat^\op\xrightarrow{[-\|\X,\TT]}\TngCat^\op\xrightarrow{\AC^\op}\TngCat^\op
\end{align*}
which send a tangent category $(\X',\TT')$ to the tangent categories $\VAC[\X',\TT'\|\X,\TT]$, $\HAC[\X',\TT'\|\X,\TT]$, and $\AC[\X',\TT'\|\X,\TT]$, respectively, into corepresentable functors. In particular, $\Gamma_{\Univ{\U^\VAC,k^\VAC}}$, $\Gamma_{\Univ{\U^\HAC,h^\HAC}}$, and $\Gamma_{\Univ{\U^\AC,k^\AC,h^\AC}}$ are invertible.
\end{theorem}
\begin{proof}
The proof of this theorem is fairly similar to that of Theorem~\ref{theorem:universality-differential-objects}.
\end{proof}

Theorem~\ref{theorem:universality-affine-connections} establishes the correct universal property of vertical, horizontal, and affine connections. Thanks to this result, we can finally introduce the notion of vertical, horizontal, and affine connections in the formal context of tangentads. Let us start with vertical connections.

\begin{definition}
\label{definition:construction-vertical-affine-connections}
A tangentad $(\X,\TT)$ in a $2$-category $\CC$ \textbf{admits the construction of vertical affine connections} if there exists a tangentad $\VAC(\X,\TT)$ of $\CC$ together with a tangent morphism $\U^\VAC\colon\VAC(\X,\TT)\to(\X,\TT)$ and a vertical affine connection $\Univ{k^\VAC}$ on $\U^\VAC$ such that the induced tangent natural transformation $\Gamma_{\U^\VAC,\Univ{k^\VAC}}$ of Lemma~\ref{lemma:induced-lax-tangent-morphism-from-affine-connections} is invertible. The vertical affine connection $\Univ{k^\VAC}\=(\U^\VAC,\Univ{k^\VAC})$ is called the \textbf{universal vertical affine connection} of $(\X,\TT)$ and $\VAC(\X,\TT)$ is called the \textbf{tangentad of vertical affine connections} of $(\X,\TT)$.
\end{definition}

\begin{definition}
\label{definition:construction-vertical-affine-connections-2-category}
A $2$-category $\CC$ \textbf{admits the construction of vertical affine connections} provided that every tangentad of $\CC$ admits the construction of vertical affine connections.
\end{definition}

The next step is to introduce horizontal affine connections.

\begin{definition}
\label{definition:construction-horizontal-affine-connections}
A tangentad $(\X,\TT)$ in a $2$-category $\CC$ \textbf{admits the construction of horizontal affine connections} if there exists a tangentad $\HAC(\X,\TT)$ of $\CC$ together with a tangent morphism $\U^\HAC\colon\HAC(\X,\TT)\to(\X,\TT)$ and a horizontal affine connection $\Univ{h^\HAC}$ on $\U^\HAC$ such that the induced tangent natural transformation $\Gamma_{\U^\HAC,\Univ{h^\HAC}}$ of Lemma~\ref{lemma:induced-lax-tangent-morphism-from-affine-connections} is invertible. The horizontal affine connection $\Univ{\U^\HAC,h^\HAC}\=(\U^\HAC,\Univ{h^\HAC})$ is called the \textbf{universal horizontal affine connection} of $(\X,\TT)$ and $\HAC(\X,\TT)$ is called the \textbf{tangentad of horizontal affine connections} of $(\X,\TT)$.
\end{definition}

\begin{definition}
\label{definition:construction-horizontal-affine-connections-2-category}
A $2$-category $\CC$ \textbf{admits the construction of horizontal affine connections} provided that every tangentad of $\CC$ admits the construction of horizontal affine connections.
\end{definition}

Finally, we can introduce affine connections.

\begin{definition}
\label{definition:construction-affine-connections}
A tangentad $(\X,\TT)$ in a $2$-category $\CC$ \textbf{admits the construction of affine connections} if there exists a tangentad $\AC(\X,\TT)$ of $\CC$ together with a tangent morphism $\U^\AC\colon\AC(\X,\TT)\to(\X,\TT)$ and an affine connection $(\Univ{k^\AC},\Univ{h^\AC})$ on $\U^\AC$ such that the induced tangent natural transformation $\Gamma_{\U^\AC;\Univ{k^\AC},\Univ{h^\AC}}$ of Lemma~\ref{lemma:induced-lax-tangent-morphism-from-affine-connections} is invertible. The affine connection $\Univ{\U^\AC;k^\AC,h^\AC}\=(\U^\AC;\Univ{k^\AC},\Univ{h^\AC})$ is called the \textbf{universal affine connection} of $(\X,\TT)$ and $\AC(\X,\TT)$ is called the \textbf{tangentad of affine connections} of $(\X,\TT)$.
\end{definition}

\begin{definition}
\label{definition:construction-affine-connections-2-category}
A $2$-category $\CC$ \textbf{admits the construction of affine connections} provided that every tangentad of $\CC$ admits the construction of affine connections.
\end{definition}

We can now rephrase Theorem~\ref{theorem:universality-affine-connections} as follows.

\begin{corollary}
\label{corollary:universality-affine-connections}
The $2$-category $\Cat$ of categories admits the constructions of vertical, horizontal, and affine connections and the tangentads of vertical, horizontal, and affine connections of a tangentad $(\X,\TT)$ of $\Cat$ are the tangent categories $\VAC(\X,\TT)$, $\HAC(\X,\TT)$, and $\AC(\X,\TT)$ of vertical, horizontal, and affine connections of $(\X,\TT)$, respectively.
\end{corollary}

To ensure that Definitions~\ref{definition:construction-vertical-affine-connections},~\ref{definition:construction-horizontal-affine-connections}, and~\ref{definition:construction-affine-connections} are well-posed, one requires the tangentads of vertical, horizontal, and affine connections of a given tangentad to be unique. The next proposition establishes that such constructions are defined uniquely up to a unique isomorphism.

\begin{proposition}
\label{proposition:uniqueness-affine-connections}
If a tangentad $(\X,\TT)$ admits the construction of (vertical/horizontal) affine connections, the tangentad of (vertical/horizontal) affine connections of $(\X,\TT)$ is unique up to a unique isomorphism which extends to an isomorphism of the corresponding universal (vertical/horizontal) affine connections of $(\X,\TT)$.
\end{proposition}
\begin{proof}
The proof of this proposition is fairly similar to that of Proposition~\ref{proposition:uniqueness-differential-objects}. We leave it to the reader to complete the details.
\end{proof}

%__________________________________________________________________________
\subsection{The universal property of linear sections}
\label{subsection:universal-property-linear-sections}
One of the goals of the formal theory of connections is to extend to the formal context of tangentads notions like the covariant derivative, the curvature, and the torsion of a connection. These are operations which involve vector fields and sections of the linear connection on which the connection is defined. Thus, in order to introduce these concepts, we first need to capture the sections of differential bundles with a connection.
\par In~\cite[Section~3.2]{lanfranchi:tangentads-II}, we introduced the tangentad of vector fields of a given tangentad via a universal property. We adopt a similar strategy and introduce the tangentad of \textbf{linear sections} of differential bundles with a connection.
\par For starters, consider a pointwise display differential bundle $\q\colon(K,\theta)\to(G,\beta)$ of the Hom-tangent category $[\X',\TT'\|\X,\TT]$; assume also that the underlying tangent morphisms $(G,\beta)$ and $(K,\theta)$ are strong. Via the universal property of linear connections, when $(\X,\TT)$ and $(\X',\TT')$ are tangent categories, the data of a pointwise display differential bundle $\q$ is equivalent to a strong tangent morphism
\begin{align*}
&\Lambda[\q]\colon(\X',\TT')\to\LC(\X,\TT)
\end{align*}
which sends each $A\in\X'$ to the display differential bundle $\q_A\colon KA\to GA$. Define the tangent category $\LS(\q)$ of linear sections of $\q$ as follows:
\begin{description}
\item[Objects] An object of $\LS(\q)$ is a pair $(A,s)$ formed by an object $A$ of $\X'$ together with a morphism $s\colon GA\to KA$, which is a section of $q_A\colon KA\to GA$, that is, the following diagram commutes:
\begin{equation*}
% https://q.uiver.app/#q=WzAsMyxbMCwxLCJHQSJdLFsxLDAsIktBIl0sWzIsMSwiR0EiXSxbMCwxLCJzIl0sWzEsMiwicV9BIl0sWzAsMiwiIiwyLHsibGV2ZWwiOjIsInN0eWxlIjp7ImhlYWQiOnsibmFtZSI6Im5vbmUifX19XV0=
\begin{tikzcd}
& KA \\
GA && GA
\arrow["{q_A}", from=1-2, to=2-3]
\arrow["s", from=2-1, to=1-2]
\arrow[equals, from=2-1, to=2-3]
\end{tikzcd}
\end{equation*}

\item[Morphisms] A morphism of $\LS(\q)$ from a linear section $(A,s)$ to another linear section $(B,t)$ consists of a morphism $f\colon A\to B$ of $\X'$ which commutes with the two linear sections, that is, such that the following diagram commutes:
\begin{equation*}
% https://q.uiver.app/#q=WzAsNCxbMCwxLCJHQSJdLFswLDAsIktBIl0sWzEsMSwiR0IiXSxbMSwwLCJLQiJdLFswLDEsInMiXSxbMiwzLCJ0IiwyXSxbMCwyLCJHZiIsMl0sWzEsMywiS2YiXV0=
\begin{tikzcd}
KA & KB \\
GA & GB
\arrow["Kf", from=1-1, to=1-2]
\arrow["s", from=2-1, to=1-1]
\arrow["Gf"', from=2-1, to=2-2]
\arrow["t"', from=2-2, to=1-2]
\end{tikzcd}
\end{equation*}

\item[Tangent bundle functor] The tangent bundle functor $\T^\LS\colon\LS(\q)\to\LS(\q)$ sends a pair $(A,s)$ to $(\T'A,s_{\T'})$, where:
\begin{align*}
&s_{\T'}\colon G\T'A\xrightarrow{\beta}\T GA\xrightarrow{\T s}\T KA\xrightarrow{\theta^{-1}}K\T'A
\end{align*}
Moreover, it sends a morphism $f\colon(A,s)\to(B,t)$ to $\T f$;

\item[Structural natural transformations] The structural natural transformations of $\LS(\q)$ are the same as the ones of $(\X,\TT)$.
\end{description}

\begin{proposition}
\label{proposition:linear-sections-tangent-category}
If $\q\colon(K,\theta)\to(G,\beta)$ is a pointwise display differential bundle on the Hom-tangent category $[\X',\TT'\|\X,\TT]$ whose underlying tangent morphisms $(G,\beta)$ and $(K,\theta)$ are strong, the linear sections of $\q$ form a tangent category $\LS(\q)$.
\end{proposition}
\begin{proof}
The proof is fairly similar to proving that vector fields of a tangent category $(\X,\TT)$ form a tangent category $\VF(\X,\TT)$~\cite[Proposition~2.10]{cockett:differential-equations}, thus, we leave it to the reader to spell out the details.
\end{proof}

The tangent category $\LS(\q)$ comes with a strict tangent morphism
\begin{align*}
&\U\colon\LS(\q)\to(\X',\TT')
\end{align*}
which forgets about the linear section, that is, it sends each pair $(A,s)$ to its underlying object $A\in\X'$. By composing the forgetful functor $\U$ with the strong tangent morphism $\Lambda[\q]\colon(\X',\TT')\to\LC(\X,\TT)$ and using the universal property of linear connections, one obtains a new pointwise display differential bundle
\begin{align*}
&\q_\U\=\Gamma_{\q}(\Lambda[\q]\o\U)\colon(K,\theta)\o\U\xrightarrow{\q_{\U}}(G,\beta)\o\U
\end{align*}
in the Hom-tangent category $[\LS(\q)\|\X,\TT]$. Furthermore, $\q_\U$ comes equipped with a natural transformation which picks out each linear section:
\begin{align*}
&\Univ{s}_{(A,s)}\colon((G,\beta)\o\U)(A,s)=GA\xrightarrow{s}KA=((K,\theta)\o\U)(A,s)
\end{align*}
However, for every $A\in\X'$, $\Univ s_{(A,s)}=s\colon GA\to KA$ is a section of ${\q_\U}_{(A,s)}=q_A\colon KA\to GA$. Thus, $\Univ s$ is a section of the projection of the pointwise display differential bundle $\q_\U$.

\begin{proposition}
\label{universality-linear-sections}
If $\q\colon(K,\theta)\to(G,\beta)$ is a pointwise display differential bundle on the Hom-tangent category $[\X',\TT'\|\X,\TT]$ whose underlying tangent morphisms $(G,\beta)$ and $(K,\theta)$ are strong, $\Univ s$ is a section of the projection $\q_\U\colon(K,\theta)\o\U\to(G,\beta)\o\U$ in the Hom-tangent category $[\X',\TT'\|\X,\TT]$.
\end{proposition}
\begin{proof}
As previously discussed, $\Univ s$ composed with the projection $\q_\U$ is equal to the identity. Thus, to conclude, it suffices to show that $\Univ s$ is a morphism in the correct category. However, the naturality of $\Univ s$ is a direct consequence of the definition of the morphisms of $\LS(\q)$, while the compatibility of $\Univ s$ with the tangent structures is a consequence of the definition of the tangent structure on $\LS(\q)$.
\end{proof}

We claim that $\Univ s$ is the \emph{universal section} of $\q_\U$. To formalize this idea, consider a tangent category $(\X'',\TT'')$ and let us define a new tangent category $\LS_\q(\X'',\TT'')$ as follows:
\begin{description}
\item[Objects] An object of $\LS_\q(\X'',\TT'')$ is a tuple $(F,\alpha;\xi)$ formed by a lax tangent morphism $(F,\alpha)\colon(\X'',\TT'')\to(\X',\TT')$ and a natural transformation
\begin{align*}
&\xi\colon(G,\beta)\o(F,\alpha)\to(K,\theta)\o(F,\alpha)
\end{align*}
which is a section of the projection of the pointwise display differential bundle:
\begin{align*}
&\q_{(F,\alpha)}\=\Gamma_{\q}(\Lambda[\q]\o\U)\colon(K,\theta)\o(F,\alpha)\xrightarrow{\q_{(F,\alpha)}}(G,\beta)\o(F,\alpha)
\end{align*}

\item[Morphisms] A morphism of $\LS_\q(\X'',\TT'')$ from $(F,\alpha;\xi)$ to $(F',\alpha';\xi')$ is a tangent natural transformation $\psi\colon(F,\alpha)\Rightarrow(F',\alpha')$ which commutes with the sections, that is, the following diagram commutes for every $X\in\X''$:
\begin{equation*}
% https://q.uiver.app/#q=WzAsNCxbMCwwLCJHRlgiXSxbMSwwLCJLRlgiXSxbMCwxLCJHRidYIl0sWzEsMSwiS0YnWCJdLFswLDEsIlxcdmFycGhpX1giXSxbMCwyLCJHXFxwc2lfWCIsMl0sWzEsMywiS1xccHNpX1giXSxbMiwzLCJcXHZhcnBoaSdfWCIsMl1d
\begin{tikzcd}
GFX & KFX \\
{GF'X} & {KF'X}
\arrow["{\xi_X}", from=1-1, to=1-2]
\arrow["{G\psi_X}"', from=1-1, to=2-1]
\arrow["{K\psi_X}", from=1-2, to=2-2]
\arrow["{\xi'_X}"', from=2-1, to=2-2]
\end{tikzcd}
\end{equation*}

\item[Tangent bundle functor] The tangent bundle functor of $\LS_\q(\X'',\TT'')$ sends $(F,\alpha;\xi)$ to $(\bar\T(F,\alpha);\xi_{\T'})$, where:
\begin{align*}
&\xi_{\T'}\colon G\o\T'\o F\xrightarrow{\beta_F}\T\o G\o F\xrightarrow{\T\xi}\T\o K\o F\xrightarrow{\theta^{1}_F}K\o\T'\o F
\end{align*}

\item[Structural natural transformations] The structural natural transformations of $\LS_\q(\X'',\TT'')$ are as in the Hom-tangent category $[\X'',\TT''\|\X',\TT']$.
\end{description}
There is an equivalent description of $\LS_\q(\X'',\TT'')$ that is worth mentioning. First, notice that, since $[\X',\TT'\|\X,\TT]$ is a tangent category, by Corollary~\ref{corollary:universality-differential-bundles}, it admits the construction of linear connections. In particular, one can define the universal vertical linear connection $(\hat\UnivQ^\VLC,\hat{\Univ k})$ for the tangent category $[\X',\TT'\|\X,\TT]$, where $\hat\UnivQ^\VLC$ is a pointwise display differential bundle in the Hom-tangent category
\begin{align*}
&[\DB_\pt[\X',\TT'\|\X,\TT]\|[\X',\TT'\|\X,\TT]]
\end{align*}
Then, we can construct the tangent category $\LS(\UnivQ^\VLC)$ of linear sections of the pointwise display differential bundle $\UnivQ^\VLC$. Such a tangent category coincides with the tangent category $\LS_\q(\X'',\TT'')$.
\par To characterize the universal property enjoyed by $\LS(\q)$, consider a lax tangent morphism:
\begin{align*}
(H,\gamma)&\colon(\X'',\TT'')\to\LS(\q)
\end{align*}
By precomposing $(\q_\U,\Univ s)$ by $(H,\gamma)$ we define an object $(\U\o(H,\gamma);\Univ s_{(H,\gamma)})$ of $\LS_\q(\X'',\TT'')$. This extends to a functor:
\begin{align*}
&\Gamma_{\Univ s}\colon[\X'',\TT''\|\LS(\q)]\to\LS_\q(\X'',\TT'')
\end{align*}
Conversely, consider an object $(F,\alpha;\xi)$ of $\LS_\q(\X'',\TT'')$. Define the following functor:
\begin{align*}
&\Lambda[F,\alpha;\xi]\colon(\X'',\TT'')\to\LS(\q)
\end{align*}
which sends each object $X\in\X''$ to the pair $(FX,\xi_X\colon GFX\to KFX)$ and each $f\colon X\to Y$ of $\X''$ to $Ff$. Furthermore, the distributive law $\alpha$ of $(F,\alpha)$ lifts to a distributive law for $\Lambda[F,\alpha;\xi]$. We claim that $\Gamma_{\Univ s}$ and $\Lambda\colon\LS_\q(\X'',\TT'')\to[\X'',\TT''\|\LS(\q)]$ are inverse to each other.

\begin{theorem}
\label{theorem:universality-linear-sections}
Consider a pointwise display differential bundle $\q\colon(K,\theta)\to(G,\beta)$ on the Hom-tangent category $[\X',\TT'\|\X,\TT]$ whose underlying tangent morphisms $(G,\beta)$ and $(K,\theta)$ are strong. The section $\Univ s$ of $\q_\U$ is universal. Concretely, the induced strict natural transformation
\begin{align*}
&\Gamma_{\Univ s}\colon[\X'',\TT''\|\LS(\q)]\to\LS_\q(\X'',\TT'')
\end{align*}
makes the functor
\begin{align*}
&\TngCat\xrightarrow{\LS_\q}\TngCat
\end{align*}
which sends each tangent category $(\X'',\TT'')$ to $\LS_\q(\X'',\TT'')$ into a corepresentable functor. In particular, $\Gamma_{\Univ s}$ is invertible.
\end{theorem}
\begin{proof}
Consider a tangent category $(\X'',\TT'')$ and a lax tangent morphism $(H,\gamma)\colon(\X'',\TT'')\to\LS(\q)$, which sends each $X\in\X''$ to the pair:
\begin{align*}
&HX=(FX,\xi_X\colon GFX\to KFX)
\end{align*}
We want to compare $(H,\gamma)$ with $\Lambda[\Gamma_{\Univ s}(H,\gamma)]$. Concretely, $\Lambda[\Gamma_{\Univ s}(H,\gamma)]$ sends an object $X$ of $\X''$ to the section $\Gamma_{\Univ s}(H,\gamma)_X$. However, this corresponds to the section $(FX,\xi_X)$. Given a morphism $\varphi\colon(H,\gamma)\to(H',\gamma')$ of $[\X'',\TT''\|\LS(\q))]$
\begin{align*}
&\varphi_X\colon(FX,\xi_X\colon GFX\to KFX)\to(F'X,\xi'_X\colon GF'X\to KF'X)
\end{align*}
which coincides with the morphism $\Lambda[\Gamma_{\Univ s}(\varphi)]$, since:
\begin{align*}
&\Gamma_{\Univ s}(\varphi)_X=FX\xrightarrow{\varphi_X}F'X
\end{align*}
Conversely, consider an object $(F,\alpha;\xi)$ of $\LS_\q(\X'',\TT'')$. Therefore, $\Lambda[F,\alpha;\xi]$ sends each $X$ to $(FX,\xi_X)$ where $\xi_X\colon GFX\to KFX$. Therefore, $\Gamma_{\Univ s}(\Lambda[F,\alpha;\xi])$ coincides with $(F,\alpha;\xi)$. Similarly, $\Gamma_{\Univ s}(\Lambda[\varphi])$ coincides with $\varphi$, for each $\varphi\colon(F,\alpha;\xi)\to(F',\alpha';\xi'))$.
\par Finally, it is not hard to see that the functor
\begin{align*}
&\Lambda\colon\LS_\q(\X'',\TT'')\to[\X',\TT'\|\LS(\q)]
\end{align*}
extends to a strict tangent morphism, since both $\Lambda[\T^\LS(F,\alpha;\xi)]$ and $\bar\T(\Lambda[F,\alpha;\xi])$ send each $X\in\X''$ to the section
\begin{align*}
&(\T'X,\xi_{\T'})
\end{align*}
Therefore, $\Lambda$ inverts $\Gamma_{\Univ s}$.
\end{proof}

Thanks to Theorem~\ref{theorem:universality-linear-sections}, we can now introduce the construction of linear sections in the context of tangentads.

\begin{definition}
\label{definition:construction-linear-sections}
Consider a pair of tangentads $(\X',\TT')$ and $(\X,\TT)$ in a $2$-category $\CC$ and a pointwise (display) differential bundle $\q\colon(K,\theta)\to(G,\beta)$ in the Hom-tangent category $[\X',\TT'\|\X,\TT]$. We say that $\q$ \textbf{admits the construction of linear sections} of $\q$ if there exists a tangentad $\LS(\q)$ of $\CC$ together with a tangent morphism $\U\colon\LS(\q)\to(\X',\TT')$ and a section $\Univ s$ of $\q_\U$ such that the induced tangent natural transformation $\Gamma_{\Univ s}$ is invertible. The section $(\U,\Univ s)$ is called the \textbf{universal linear section} of $\q$ and $\LS(\q)$ is called the \textbf{tangentad of linear sections} of $\q$.
\end{definition}

\begin{definition}
\label{definition:construction-linear-sections-2-category}
A $2$-category $\CC$ \textbf{admits the construction of linear sections} provided that every pointwise differential bundle $\q\colon(K,\theta)\to(G,\beta)$ in every Hom-tangent category $[\X',\TT'\|\X,\TT]$ admits the construction of linear sections.
\end{definition}

We can now rephrase Theorem~\ref{theorem:universality-linear-sections} as follows.

\begin{corollary}
\label{corollary:universality-linear-sections}
The $2$-category $\Cat$ of categories admits the constructions of linear sections and the tangentad of linear sections of a pointwise (display) differential bundle $\q\colon(K,\theta)\to(G,\beta)$ is the tangent category $\LS(\q)$.
\end{corollary}

Given a tangentad $(\X,\TT)$ (whose projection is tangent display) in the Hom-tangent category $[\X,\TT\|\X,\TT]$ of lax tangent endo-morphisms over $(\X,\TT)$, there is always a canonical choice of pointwise (display) differential bundle, that is, the tangent bundle $\bar\p\colon(\T,c)\to\id_{(\X,\TT)}$ over the identity $\id_{(\X,\TT)}$. Thus, the tangentad $\LS(\bar\p)$ of linear sections of $\bar\p$ is precisely the tangentad of vector fields of $(\X,\TT)$.

\begin{theorem}
\label{theorem:vector-fields-as-linear-sections}
Consider a tangentad $(\X,\TT)$ of a $2$-category $\CC$. $(\X,\TT)$ admits the construction of vector fields if and only if the tangent bundle $\bar\p\colon(\T,c)\to\id_{(\X,\TT)}$ in the Hom-tangent category $[\X,\TT\|\X,\TT]$ admits the construction of linear sections.
\end{theorem}

We finally introduce the type of linear sections that we will use in the following section.

\begin{definition}
\label{definition:formal-linear-sections-connections}
A tangentad $(\X,\TT)$ in a $2$-category $\CC$ \textbf{admits the construction of linear sections on differential bundles with a connection} when it admits the construction of linear connections and the pointwise display differential bundle $\UnivQ^\LC\colon\Tot^\LC\to\Base^\LC$ of Proposition~\ref{proposition:universality-linear-connections} admits the construction of linear sections. In the following, we denote by $\LSC(\X,\TT)$ the tangentad $\LS(\UnivQ^\LC)$ of linear sections of $\UnivQ^\LC$.
\end{definition}

%__________________________________________________________________________
\subsection{The formal structures of connections}
\label{subsection:structures-connections}
In this section, we extend some constructions related to linear and affine connections. We start by showing that $\VLC$, $\HLC$, $\LC$, $\VAC$, $\HAC$, and $\AC$ are $2$-functorial operations on the $2$-category of tangentads and strong tangent morphisms. Next, using the construction of linear sections on differential bundles with a connection, we introduce the constructions of the covariant derivative and the curvature tensor of linear connections. Finally, we discuss how to formalize the torsion tensor for affine connections.

%__________________________________________________________________________
\subsubsection*{The functoriality of the construction of connections}
\label{subsubsection:functoriality-connections}
In this section, we show that the constructions of (vertical/horizontal) linear and affine connections are $2$-functorial in the $2$-category of tangentads of $\CC$.

\begin{proposition}
\label{proposition:LC-functoriality}
There are three $2$-functors
\begin{align*}
\VLC&\colon\Tng_\Dsply(\CC)\to\Tng_\cong(\CC)\\
\HLC&\colon\Tng_\Dsply(\CC)\to\Tng_\cong(\CC)\\
\LC&\colon\Tng_\Dsply(\CC)\to\Tng_\cong(\CC)
\end{align*}
on the $2$-category of tangentads and display tangent morphisms, which send a tangentad $(\X,\TT)$ to the tangent tangentads of (vertical/horizontal) linear connections of $(\X,\TT)$.
\end{proposition}
\begin{proof}
The proof is fairly similar to the one of PropositionsThe proof is fairly similar to the proofs of Proposition~\ref{proposition:DO-functoriality}, and~\ref{proposition:DB-functoriality}. Thus, we leave it to the reader to spell out the details.
\end{proof}

We obtain a similar result for affine connections.

\begin{proposition}
\label{proposition:AC-functoriality}
There are three $2$-endofunctors
\begin{align*}
\VAC&\colon\Tng_\cong(\CC)\to\Tng_\cong(\CC)\\
\HAC&\colon\Tng_\cong(\CC)\to\Tng_\cong(\CC)\\
\AC&\colon\Tng_\cong(\CC)\to\Tng_\cong(\CC)
\end{align*}
on the $2$-category of tangentads and strong tangent morphisms, which send a tangentad $(\X,\TT)$ to the tangent tangentads of (vertical/horizontal) affine connections of $(\X,\TT)$.
\end{proposition}

%__________________________________________________________________________
\subsubsection*{The covariant derivative of connections}
\label{subsubsection:covariant-derivative}
Linear and affine connections in a tangent category come with a notion of covariant derivative. In this section, we show how to formalize this concept in the context of tangentads. Recall that the covariant derivative $\nabla^k$ associated to a (vertical) linear connection $k$ on a differential bundle $\q\colon E\to M$ is an operation which takes a vector field $v\colon M\to\T M$ of the base object $M$ of $\q$ and a section $s\colon M\to E$ of $\q$ and yields a new section $\nabla^k_vs$ of $\q$. 
\par Therefore, while we consider only tangentads $(\X,\TT)$ which admit the construction of (vertical) linear connections, we must also assume that $(\X,\TT)$ admits the construction of vector fields and the one of sections of differential bundles with a connection. In particular, let $\Univ k$ denotes the universal vertical connection of $(\X,\TT)$, $\UnivQ^\VLC$ the underlying pointwise display differential bundle of the universal vertical connection, $\Univ v$ the universal vector field of $(\X,\TT)$, and $\Univ s$ the universal section of $\UnivQ^\VLC$.
\par We must also choose a vector field of the base of a linear connection with a connection, together with a section of the differential bundle. For this, we need to require the existence of the following $2$-pullback in $\CC$:
\begin{equation}
\label{equation:covariant-derivative-pullback}
% https://q.uiver.app/#q=WzAsNSxbMCwwLCJcXFZGKFxcWCxcXFRUKVxcdGltZXNfeyhcXFgsXFxUVCl9XFxMUyhcXFVuaXZ7XFxxfSkiXSxbMCwyLCJcXFZGKFxcWCxcXFRUKSJdLFsyLDIsIihcXFgsXFxUVCkiXSxbMiwwLCJcXExTKFxcVW5pdntcXHF9KSJdLFsyLDEsIlxcVkxDKFxcWCxcXFRUKSJdLFszLDQsIlxcVV5cXExTIl0sWzQsMiwiXFxCYXNlXlxcVkxDIl0sWzEsMiwiXFxVXlxcVkYiLDJdLFswLDEsIlxccGlfe1xcVkZ9IiwyXSxbMCwzLCJcXHBpX1xcTFMiXSxbMCwyLCIiLDEseyJzdHlsZSI6eyJuYW1lIjoiY29ybmVyIn19XV0=
\begin{tikzcd}
{\VF(\X,\TT)\times_{(\X,\TT)}\LSC(\X,\TT)} && {\LSC(\X,\TT)} \\
&& {\VLC(\X,\TT)} \\
{\VF(\X,\TT)} && {(\X,\TT)}
\arrow["{\pi_\LS}", from=1-1, to=1-3]
\arrow["{\pi_{\VF}}"', from=1-1, to=3-1]
\arrow["\lrcorner"{anchor=center, pos=0.125}, draw=none, from=1-1, to=3-3]
\arrow["{\U^\LS}", from=1-3, to=2-3]
\arrow["{\Base^\VLC}", from=2-3, to=3-3]
\arrow["{\U^\VF}"', from=3-1, to=3-3]
\end{tikzcd}
\end{equation}
Notice that, by~\cite[Proposition~4.7]{lanfranchi:tangentads-II}, the diagram of Equation~\eqref{equation:covariant-derivative-pullback}, becomes a $2$-pullback in $\Tng(\CC)$.
\par In the context of tangent categories, the objects of $\VF(\X,\TT)\times_{(\X,\TT)}\LSC(\X,\TT)$ are fourth-tuples $(\q,k;v,s)$ formed by a display differential bundle $\q\colon E\to M$, a vertical connection $k$ on $\q$, a vector field $v\colon M\to\T M$ on the base object of $\q$ and a section $s\colon M\to E$ of $\q$. A morphism
\begin{align*}
(f,g)&\colon(\q,k;v,s)\to(\q',k';v',s')
\end{align*}
of $\VF(\X,\TT)\times_{(\X,\TT)}\LSC(\X,\TT)$ is a linear morphism of differential bundles $(f,g)\colon\q\to\q'$ which commutes with the vertical connections, whose base morphism $f$ commutes with the vector fields, that is, $v'\o f=\T f\o v$, and such that $(f,g)$ commutes with the sections, that is, $s'\o f=g\o s$. Finally, the tangent structure of $\VF(\X,\TT)\times_{(\X,\TT)}\LSC(\X,\TT)$ is computed as in $\VF(\X,\TT)$ and $\LSC(\X,\TT)$.
\par We can use the tangent morphisms
\begin{align*}
&\VF(\X,\TT)\times_{(\X,\TT)}\LSC(\X,\TT)\xrightarrow{\pi_\VF}\VF(\X,\TT)\\
&\VF(\X,\TT)\times_{(\X,\TT)}\LSC(\X,\TT)\xrightarrow{\pi_\LS}\LSC(\X,\TT)\xrightarrow{\U^\LS}\VLC(\X,\TT)
\end{align*}
to pull back the universal vector field $\Univ v$, universal vertical linear connection $\Univ k$, the underlying linear connection $\UnivQ^\VLC$, and the universal section $\Univ s$ to $\VF(\X,\TT)\times_{(\X,\TT)}\LSC(\X,\TT)$ and obtain a vector field $\tilde{\Univ v}\=\Gamma_{\Univ v}(\pi_\VF)$, a vertical linear connection $\tilde{\Univ k}\=\Gamma_{\Univ k}(\U^\LS\o\pi_\LS)$, a linear connection ${\tilde{\UnivQ}}^\VLC\=\Gamma_{\UnivQ^\VLC}(\U^\LS\o\pi_\LS)$, and a section $\tilde{\Univ s}\=\Gamma_{\Univ s}(\pi_\LS)$ of ${\tilde{\UnivQ}}^\VLC$. Since $\tilde{\Univ k}$ is a vertical linear connection on ${\tilde{\UnivQ}}^\VLC$, it defines a covariant derivative $\nabla^{\tilde{\Univ k}}$. Thus, by applying $\nabla^{\tilde{\Univ k}}$ to the vector field $\tilde{\Univ v}$ and to the section $\tilde{\Univ s}$ of ${\tilde{\UnivQ}}^\VLC$ we obtain a new section $\nabla^{\tilde{\Univ k}}_{\tilde{\Univ v}}\tilde{\Univ s}$ to ${\tilde{\UnivQ}}^\VLC$.
\par However, since
\begin{align*}
&{\tilde{\UnivQ}}^\VLC=\UnivQ^\VLC_{\U^\LS\o\pi_\LS}
\end{align*}
is obtained by pre-composing $\UnivQ^\VLC$ by $\U^\LS\o\pi_\LS$, it follows that
\begin{align*}
&(\U_\LS\o\pi_\LS,\nabla^{\tilde{\Univ k}}_{\tilde{\Univ v}}\tilde{\Univ s})
\end{align*}
defines an object of $\LS_{\UnivQ^\VLC}(\VF(\X,\TT)\times_{(\X,\TT)}\LSC(\X,\TT))$. By the universal property of linear sections, the section $\nabla^{\tilde{\Univ k}}_{\tilde{\Univ v}}\tilde{\Univ s}$ corresponds to a tangent morphism:
\begin{align*}
&\nabla\colon\VF(\X,\TT)\times_{(\X,\TT)}\LSC(\X,\TT)\to\LSC(\X,\TT)
\end{align*}

\begin{definition}
\label{definition:construction-covariant-derivative}
A tangentad $(\X,\TT)$ of a $2$-category $\CC$ \textbf{admits the construction of the covariant derivative} when it admits the constructions of vertical linear connections, the construction of vector fields, the construction of linear sections of differential bundles with a connection, and the $2$-pullback of Equation~\eqref{equation:covariant-derivative-pullback} exists in $\CC$ of tangentads of $\CC$. Under these conditions, the \textbf{covariant derivative operator} of $(\X,\TT)$ is the tangent morphism
\begin{align*}
&\nabla\colon\VF(\X,\TT)\times_{(\X,\TT)}\LSC(\X,\TT)\to\LSC(\X,\TT)
\end{align*}
which corresponds to the linear section $\nabla^{\tilde{\Univ k}}_{\tilde{\Univ v}}\tilde{\Univ s}$.
\end{definition}

\begin{definition}
\label{definition:construction-covariant-derivative-2-category}
A $2$-category $\CC$ \textbf{admits the construction of the covariant derivative} provided that every tangentad of $\CC$ admits such a construction.
\end{definition}

\begin{theorem}
\label{theorem:covariant-derivative}
The $2$-category $\Cat$ of categories admits the construction of the covariant derivative. Moreover, given a tangent category $(\X,\TT)$, the corresponding covariant derivative operator is a functor which sends a fourth-tuple $(\q,k;v,s)$ formed by a display differential bundle $\q\colon E\to M$, a vertical linear connection $k$ of $\q$, a vector field $v$ of $M$ and a section $s$ of $\q$, to the section $\nabla^k_vs$ obtained by applying the covariant derivative of the vertical connection $k$ to $v$ and $s$.
\end{theorem}

%__________________________________________________________________________
\subsubsection*{The curvature of connections}
\label{subsubsection:curvature}
In this section, we extend the notion of curvature tensor of linear connections to the formal context of tangentads. For starters, recall that the curvature tensor $\Riem$ of a (vertical) linear connection $k$ on a display differential bundle $\q\colon E\to M$ of a tangent category $(\X,\TT)$ is an operation which sends two vector fields $u,v\colon M\to\T M$ of the base object $M$ of $\q$ and a section $s\colon M\to E$ of $\q$ to a new section $\Riem^k_{u,v}s\colon M\to E$ of $\q$.
\par This suggests considering a tangentad $(\X,\TT)$ of a $2$-category $\CC$ which admits the constructions of (vertical) linear connections, the construction of vector fields, and the construction of linear sections of differential bundles with a connection. Moreover, we shall assume the existence of the following trifold $2$-pullback in $\CC$:
\begin{equation}
\label{equation:curvature-pullback}
% https://q.uiver.app/#q=WzAsNixbMiwyLCIoXFxYLFxcVFQpIl0sWzAsMiwiXFxWRihcXFgsXFxUVCkiXSxbMSwxLCJcXFZGKFxcWCxcXFRUKSJdLFsyLDEsIlxcVkxDKFxcWCxcXFRUKSJdLFsyLDAsIlxcTFNDKFxcWCxcXFRUKSJdLFswLDAsIlxcVkZfMihcXFgsXFxUVClcXHRpbWVzX3soXFxYLFxcVFQpfVxcTFNDKFxcWCxcXFRUKSJdLFsxLDAsIlxcVV5cXFZGIiwyXSxbMiwwLCJcXFVeXFxWRiIsMl0sWzMsMCwiXFxCYXNlIl0sWzQsMywiXFxVXlxcTFMiXSxbNSwxLCJcXHBpXzEiLDJdLFs1LDIsIlxccGlfMiIsMl0sWzUsNCwiXFxwaV8zIl0sWzUsNiwiIiwwLHsibGV2ZWwiOjEsInN0eWxlIjp7Im5hbWUiOiJjb3JuZXIifX1dXQ==
\begin{tikzcd}
{\VF_2(\X,\TT)\times_{(\X,\TT)}\LSC(\X,\TT)} && {\LSC(\X,\TT)} \\
& {\VF(\X,\TT)} & {\VLC(\X,\TT)} \\
{\VF(\X,\TT)} && {(\X,\TT)}
\arrow["{\pi_3}", from=1-1, to=1-3]
\arrow["{\pi_2}"', from=1-1, to=2-2]
\arrow["{\pi_1}"', from=1-1, to=3-1]
\arrow["{\U^\LS}", from=1-3, to=2-3]
\arrow["{\U^\VF}"', from=2-2, to=3-3]
\arrow["\Base", from=2-3, to=3-3]
\arrow[""{name=0, anchor=center, inner sep=0}, "{\U^\VF}"', from=3-1, to=3-3]
\arrow["\lrcorner"{anchor=center, pos=0.125}, draw=none, from=1-1, to=0]
\end{tikzcd}
\end{equation}
Notice that, by~\cite[Proposition~4.7]{lanfranchi:tangentads-II}, the diagram of Equation~\eqref{equation:curvature-pullback}, becomes a $2$-pullback in $\Tng(\CC)$.
\par In the context of tangent categories, the objects of $\VF_2(\X,\TT)\times_{(\X,\TT)}\LSC(\X,\TT)$ are quintuples $(\q,k;u,v,s)$ formed by a display differential bundle $\q\colon E\to M$ of $(\X,\TT)$, a vertical linear connection $k$ on $\q$, a pair of vector fields $u,v\colon M\to\T M$ on the base object $M$ of $\q$, and a section $s\colon M\to E$ of $\q$. A morphism from $(\q,k;u,v,s)$ to $(\q',k';u',v',s')$ is a linear morphism $(f,g)\colon\q\to\q'$ of differential bundles which commute with the vertical connections $k$ and $k'$, whose base morphism $f\colon M\to M'$ commutes with the pair of vector fields $u$, $u'$ and $v$, $v'$, that is, $v'\o f=\T f\o v$ and $u'\o f=\T f\o u$, and such that $(f,g)$ commutes with the sections $s$ and $s'$, that is, $s'\o f=g\o s$. The tangent structure on $\VF_2(\X,\TT)\times_{(\X,\TT)}\LSC(\X,\TT)$ acts on the linear connection with connection $(\q,k)$ as in $\VLC(\X,\TT)$, on the vector fields $u$ and $v$ as in $\VF(\X,\TT)$, and on the section $s$ as in $\LSC(\X,\TT)$.
\par We can use the tangent morphisms
\begin{align*}
&\VF_2(\X,\TT)\times_{(\X,\TT)}\LSC(\X,\TT)\xrightarrow{\pi_1}\VF(\X,\TT)\\
&\VF_2(\X,\TT)\times_{(\X,\TT)}\LSC(\X,\TT)\xrightarrow{\pi_2}\VF(\X,\TT)\\
&\VF_2(\X,\TT)\times_{(\X,\TT)}\LSC(\X,\TT)\xrightarrow{\pi_3}\LSC(\X,\TT)\xrightarrow{\U^\LS}\VLC(\X,\TT)\\
\end{align*}
to pull back the universal vector field $\Univ v$, universal vertical linear connection $\Univ k$, the underlying linear connection $\UnivQ^\VLC$, and the universal section $\Univ s$ to $\VF_2(\X,\TT)\times_{(\X,\TT)}\LSC(\X,\TT)$ and obtain two vector fields $\tilde{\Univ v}_1\=\Gamma_{\Univ v}(\pi_1)$ and $\tilde{\Univ v}_2\=\Gamma_{\Univ v}(\pi_2)$, a vertical linear connection $\tilde{\Univ k}\=\Gamma_{\Univ k}(\U^\LS\o\pi_3)$, a linear connection ${\tilde{\UnivQ}}^\VLC\=\Gamma_{\UnivQ^\VLC}(\U^\LS\o\pi_3)$, and a section $\tilde{\Univ s}\=\Gamma_{\Univ s}(\pi_3)$ of ${\tilde{\UnivQ}}^\VLC$.

\par Since $\tilde{\Univ k}$ is a vertical linear connection on ${\tilde{\UnivQ}}^\VLC$, it defines a curvature tensor $\Riem^{\tilde{\Univ k}}$. Thus, by applying $\Riem^{\tilde{\Univ k}}$ to the vector fields $\tilde{\Univ v}_1$ and $\tilde{\Univ v}_2$ and to the section $\tilde{\Univ s}$ of ${\tilde{\UnivQ}}^\VLC$ we obtain a new section $\Riem^{\tilde{\Univ k}}_{\tilde{\Univ v}_1,\tilde{\Univ v}_2}\tilde{\Univ s}$ of ${\tilde{\UnivQ}}^\VLC$.
\par However, since
\begin{align*}
&{\tilde{\UnivQ}}^\VLC=\UnivQ^\VLC_{\U^\LS\o\pi_3}
\end{align*}
is obtained by pre-composing $\UnivQ^\VLC$ by $\U^\LS\o\pi_3$, it follows that
\begin{align*}
&(\U_\LS\o\pi_3,\Riem^{\tilde{\Univ k}}_{\tilde{\Univ v}_1,\tilde{\Univ v}_2}\tilde{\Univ s})
\end{align*}
defines an object of $\LS_{\UnivQ^\VLC}(\VF_2(\X,\TT)\times_{(\X,\TT)}\LSC(\X,\TT))$. By the universal property of linear sections, the section $\Riem^{\tilde{\Univ k}}_{\tilde{\Univ v}_1,\tilde{\Univ v}_2}\tilde{\Univ s}$ corresponds to a tangent morphism:
\begin{align*}
&\Riem\colon\VF_2(\X,\TT)\times_{(\X,\TT)}\LSC(\X,\TT)\to\LSC(\X,\TT)
\end{align*}

\begin{definition}
\label{definition:curvature-construction}
A tangentad $(\X,\TT)$ of a $2$-category $\CC$ \textbf{admits the construction of the curvature tensor} when it admits the constructions of vertical linear connections, the construction of vector fields, the construction of linear sections of differential bundles with a connection, and the trifold $2$-pullback of Equation~\eqref{equation:curvature-pullback} exists in $\CC$. Under these conditions, the \textbf{curvature tensor operator} of $(\X,\TT)$ is the tangent morphism
\begin{align*}
&\Riem\colon\VF_2(\X,\TT)\times_{(\X,\TT)}\LSC(\X,\TT)\to\LSC(\X,\TT)
\end{align*}
which corresponds to the linear section $\Riem^{\tilde{\Univ k}}_{\tilde{\Univ v}_1,\tilde{\Univ v}_2}\tilde{\Univ s}$.
\end{definition}

\begin{definition}
\label{definition:construction-curvature-2-category}
A $2$-category $\CC$ \textbf{admits the construction of the curvature tensor} provided that every tangentad of $\CC$ admits such a construction.
\end{definition}

\begin{theorem}
\label{theorem:curvature-construction}
The $2$-category $\Cat$ of categories admits the construction of the curvature tensor. Moreover, given a tangent category $(\X,\TT)$, the corresponding curvature tensor operator is a functor which sends a tuple $(\q,k;u,v,s)$ formed by a display differential bundle $\q\colon E\to M$, a vertical linear connection $k$ of $\q$, a pair of vector fields $u$ and $v$ of $M$ and a section $s$ of $\q$, to the section $\Riem^k_{u,v}s$ obtained by applying the curvature tensor of the vertical connection $k$ to $u$, $v$, and $s$.
\end{theorem}

%__________________________________________________________________________
\subsubsection*{The torsion of connections}
\label{subsubsection:torsion}
In this section, we extend the notion of torsion tensor of affine connections to the formal context of tangentads. For starters, recall that the torsion tensor $\TorsTens$ of a (vertical) affine connection $k$ on an object $M$ of a tangent category $(\X,\TT)$ is an operation which sends two vector fields $u,v\colon M\to\T M$ of $M$ to a new vector field $\TorsTens^k_uv\colon M\to\T M$ of $M$. This suggests considering a tangentad $(\X,\TT)$ of a $2$-category $\CC$ which admits the constructions of (vertical) linear connections, and the construction of vector fields. Moreover, we shall assume the existence of the following $2$-pullback in the $2$-category $\Tng(\CC)$ of tangentads of $\CC$:
\begin{equation}
\label{equation:torsion-pullback}
% https://q.uiver.app/#q=WzAsNSxbMiwyLCIoXFxYLFxcVFQpIl0sWzAsMiwiXFxWRihcXFgsXFxUVCkiXSxbMSwxLCJcXFZGKFxcWCxcXFRUKSJdLFsyLDAsIlxcVkFDKFxcWCxcXFRUKSJdLFswLDAsIlxcVkZfMihcXFgsXFxUVClcXHRpbWVzX3soXFxYLFxcVFQpfVxcVkFDKFxcWCxcXFRUKSJdLFsxLDAsIlxcVV5cXFZGIiwyXSxbMiwwLCJcXFVeXFxWRiIsMl0sWzQsMSwiXFxwaV8xIiwyXSxbNCwyLCJcXHBpXzIiLDJdLFs0LDMsIlxccGlfMyJdLFszLDAsIlxcVV5cXFZBQyJdLFs0LDUsIiIsMCx7ImxldmVsIjoxLCJzdHlsZSI6eyJuYW1lIjoiY29ybmVyIn19XV0=
\begin{tikzcd}
{\VF_2(\X,\TT)\times_{(\X,\TT)}\VAC(\X,\TT)} && {\VAC(\X,\TT)} \\
& {\VF(\X,\TT)} \\
{\VF(\X,\TT)} && {(\X,\TT)}
\arrow["{\pi_3}", from=1-1, to=1-3]
\arrow["{\pi_2}"', from=1-1, to=2-2]
\arrow["{\pi_1}"', from=1-1, to=3-1]
\arrow["{\U^\VAC}", from=1-3, to=3-3]
\arrow["{\U^\VF}"', from=2-2, to=3-3]
\arrow[""{name=0, anchor=center, inner sep=0}, "{\U^\VF}"', from=3-1, to=3-3]
\arrow["\lrcorner"{anchor=center, pos=0.125}, draw=none, from=1-1, to=0]
\end{tikzcd}
\end{equation}
Notice that, by~\cite[Proposition~4.7]{lanfranchi:tangentads-II}, the diagram of Equation~\eqref{equation:torsion-pullback}, becomes a $2$-pullback in $\Tng(\CC)$.
\par In the context of tangent categories, the objects of $\VF_2(\X,\TT)\times_{(\X,\TT)}\VAC(\X,\TT)$ are fourth-tuples $(M,k;u,v)$ formed by an object $M$ of $(\X,\TT)$, a vertical affine connection $k$ on $M$, and a pair of vector fields $u,v\colon M\to\T M$ on $M$. A morphism from $(\q,k;u,v)$ to $\q',k';u',v')$ is a morphism $f\colon M\to M'$ which commutes with the vertical connections $k$ and $k'$ and which commutes with the pair of vector fields $u$, $u'$ and $v$, $v'$, that is, $v'\o f=\T f\o v$ and $u'\o f=\T f\o u$. The tangent structure on $\VF_2(\X,\TT)\times_{(\X,\TT)}\VAC(\X,\TT)$ acts on the vertical affine connection $(M,k)$ as in $\VAC(\X,\TT)$, and on the vector fields $u$ and $v$ as in $\VF(\X,\TT)$.
\par We can use the tangent morphisms
\begin{align*}
&\VF_2(\X,\TT)\times_{(\X,\TT)}\VAC(\X,\TT)\xrightarrow{\pi_1}\VF(\X,\TT)\\
&\VF_2(\X,\TT)\times_{(\X,\TT)}\VAC(\X,\TT)\xrightarrow{\pi_2}\VF(\X,\TT)\\
&\VF_2(\X,\TT)\times_{(\X,\TT)}\VAC(\X,\TT)\xrightarrow{\pi_3}\VAC(\X,\TT)
\end{align*}
to pull back the universal vector field $\Univ v$, and universal vertical affine connection $\Univ k$ to $\VF_2(\X,\TT)\times_{(\X,\TT)}\VAC(\X,\TT)$ and obtain two vector fields $\tilde{\Univ v}_1\=\Gamma_{\Univ v}(\pi_1)$ and $\tilde{\Univ v}_2\=\Gamma_{\Univ v}(\pi_2)$, and a vertical affine connection $\tilde{\Univ k}\=\Gamma_{\Univ k}(\pi_3)$.

\par Since $\tilde{\Univ k}$ is a vertical affine connection, it defines a torsion tensor $\TorsTens^{\tilde{\Univ k}}$. Thus, by applying $\TorsTens^{\tilde{\Univ k}}$ to the vector fields $\tilde{\Univ v}_1$ and $\tilde{\Univ v}_2$ we obtain a new vector field $\TorsTens^{\tilde{\Univ k}}_{\tilde{\Univ v}_1}\tilde{\Univ v}_2$.
\par However, by the universal property of vector fields, the vector field $\TorsTens^{\tilde{\Univ k}}_{\tilde{\Univ v}_1}\tilde{\Univ v}_2$ corresponds to a tangent morphism:
\begin{align*}
&\TorsTens\colon\VF_2(\X,\TT)\times_{(\X,\TT)}\VAC(\X,\TT)\to\VF(\X,\TT)
\end{align*}

\begin{definition}
\label{definition:torsion-construction}
A tangentad $(\X,\TT)$ of a $2$-category $\CC$ \textbf{admits the construction of the torsion tensor} when it admits the constructions of vertical linear connections, the construction of vector fields, the construction of linear sections of differential bundles with a connection, and the $2$-pullback of Equation~\eqref{equation:curvature-pullback} exists in $\CC$. Under these conditions, the \textbf{torsion tensor operator} of $(\X,\TT)$ is the tangent morphism
\begin{align*}
&\TorsTens\colon\VF_2(\X,\TT)\times_{(\X,\TT)}\VAC(\X,\TT)\to\VF(\X,\TT)
\end{align*}
which corresponds to the vector field $\TorsTens^{\tilde{\Univ k}}_{\tilde{\Univ v}_1}\tilde{\Univ v}_2$.
\end{definition}

\begin{definition}
\label{definition:construction-torsion-2-category}
A $2$-category $\CC$ \textbf{admits the construction of the torsion tensor} provided that every tangentad of $\CC$ admits such a construction.
\end{definition}

\begin{theorem}
\label{theorem:torsion-construction}
The $2$-category $\Cat$ of categories admits the construction of the torsion tensor. Moreover, given a tangent category $(\X,\TT)$, the corresponding torsion tensor operator is a functor which sends a tuple $(M,k;u,v)$ formed by an object $M$ of $(\X,\TT)$, a vertical affine connection $k$ on $M$, and a pair of vector fields $u$ and $v$ of $M$ to the vector field $\TorsTens^k_uv$ obtained by applying the torsion tensor of the vertical connection $k$ to $u$ and $v$.
\end{theorem}

%__________________________________________________________________________
\subsection{The construction of connections via (P)IE limits}
\label{subsection:PIE-limits-connections}
PIE limits can be used to formally construct algebraic theories, such as the construction of algebras of a monad internal to a $2$-category. In~\cite[Section~4.5]{lanfranchi:tangentads-II}, we used this property of PIE limits to construct the tangentads of vector fields of the tangentads of a $2$-category, by means of inserters and equifiers. The same cannot be done for the construction of differential objects and the constructions of differential bundles, since non-equational axioms (the universality of the differential projection of a differential object and the universality of the vertical lift of a display differential bundle, together with the existence of all pullbacks along the projection) are involved in the definitions of these two concepts.
\par The theory of affine connections is entirely algebraic, that is, it only involves equational axioms, and the only non-equational axioms demanded in the theory of linear connections are related to the differential bundles on which the connections are defined. In this section, we harness this property of connections and show how to construct the tangentads of (vertical/horizontal) affine and linear connections of the tangentads of a $2$-category via inserters and equifiers.
\par For starters, consider a tangent category $(\X,\TT)$ and let us construct the tangent categories of vertical, horizontal and linear connections of $(\X,\TT)$. Recall that a vertical connection on a differential bundle $\q\colon E\to M$ is a morphism $k\colon\T E\to E$ and a horizontal connection on $\q$ is a morphism $h\colon\H q=E\times_M\T M\to\T E$. Consider the following two inserters:
\begin{align*}
(\I^k_0,\TT^k_0)&\colon\bar\T\Tot\Rightarrow\Tot        &&(\I^h_0,\TT^h_0)\colon\bar\H\UnivQ\Rightarrow\bar\T\Tot
\end{align*}
where $\UnivQ\colon\Tot\to\Base$ is the universal pointwise display differential bundle of $(\X,\TT)$ and $\bar\H\UnivQ=\Tot\times_\Base\bar\T\Base$ denotes the pointwise horizontal bundle of $\UnivQ$, which is the tangent morphism which sends each differential bundle $\q\colon E\to M$ to $\H q=E\times_M\T M\in(\X,\TT)$. Concretely, $(\I^k_0,\TT^k_0)$ is a tangent category equipped with a strict tangent morphism $V^k_0\colon(\I^k_0,\TT^k_0)\to\DB(\X,\TT)$ and a tangent $2$-morphism:
\begin{equation*}
% https://q.uiver.app/#q=WzAsNCxbMSwwLCJcXERCKFxcWCxcXFRUKSJdLFswLDAsIihcXElea18wLFxcVFRea18wKSJdLFsxLDEsIihcXFgsXFxUVCkiXSxbMCwxLCJcXERCKFxcWCxcXFRUKSJdLFsxLDAsIlYiXSxbMSwzLCJWIiwyXSxbMCwyLCJcXGJhclxcVFxcVG90Il0sWzMsMiwiXFxUb3QiLDJdLFswLDMsIktfMCIsMix7ImxldmVsIjoyfV1d
\begin{tikzcd}
{(\I^k_0,\TT^k_0)} & {\DB(\X,\TT)} \\
{\DB(\X,\TT)} & {(\X,\TT)}
\arrow["V^k_0", from=1-1, to=1-2]
\arrow["V^k_0"', from=1-1, to=2-1]
\arrow["{K_0}"', Rightarrow, from=1-2, to=2-1]
\arrow["{\bar\T\Tot}", from=1-2, to=2-2]
\arrow["\Tot"', from=2-1, to=2-2]
\end{tikzcd}
\end{equation*}
and $(\I^h_0,\TT^h_0)$ is a tangent category equipped with a strict tangent morphism $V^h_0\colon(\I^h_0,\TT^h_0)\to\DB(\X,\TT)$ and a $2$-morphism:
\begin{equation*}
% https://q.uiver.app/#q=WzAsNCxbMSwwLCJcXERCKFxcWCxcXFRUKSJdLFswLDAsIihcXEleaF8wLFxcVFReaF8wKSJdLFsxLDEsIihcXFgsXFxUVCkiXSxbMCwxLCJcXERCKFxcWCxcXFRUKSJdLFsxLDAsIlZeaCJdLFsxLDMsIlZeaCIsMl0sWzAsMiwiXFxiYXJcXEhcXFVuaXZRIl0sWzMsMiwiXFxiYXJcXFRcXFRvdCIsMl0sWzAsMywiSF8wIiwyLHsibGV2ZWwiOjJ9XV0=
\begin{tikzcd}
{(\I^h_0,\TT^h_0)} & {\DB(\X,\TT)} \\
{\DB(\X,\TT)} & {(\X,\TT)}
\arrow["{V^h_0}", from=1-1, to=1-2]
\arrow["{V^h_0}"', from=1-1, to=2-1]
\arrow["{H_0}"', Rightarrow, from=1-2, to=2-1]
\arrow["{\bar\H\UnivQ}", from=1-2, to=2-2]
\arrow["{\bar\T\Tot}"', from=2-1, to=2-2]
\end{tikzcd}
\end{equation*}
The objects of $(\I^k_0,\TT^k_0)$ are pairs $(\q,k_0)$ formed by a display differential bundle $\q\colon E\to M$ together with a morphism $k_0\colon\T E\to E$ and morphisms $(f,g)\colon(\q,k_0)\to(\q',k_0')$ are morphisms $(f,g)\colon\q\to\q'$ of differential bundles which commute with $k_0$ and $k_0'$, that is, $k_0'\o\T g=k_0\o g$.
\par The objects of $(\I^h_0,\TT^h_0)$ are pairs $(\q,h_0)$ formed by a display differential bundle $\q\colon E\to M$ together with a morphism $h_0\colon\H q\to\T E$ and morphisms $(f,g)\colon(\q,h_0)\to(\q',h_0')$ are morphisms $(f,g)\colon\q\to\q'$ of differential bundles which commute with $h_0$ and $h_0'$, that is, $h_0'\o\T g=h_0\o(g\times_f\T f)$.
\par So far, the morphisms $k_0$ and $h_0$ not need to satisfy any equational axioms. To impose the correct constraints, we use equifiers. For the sake of simplicity, we focus on vertical connections and leave it to the reader to replicate the same procedure for horizontal connections.
\par According to~\cite[Lemma~3.3]{cockett:connections}, a morphism $k\colon\T E\to E$ need to satisfy the following equations in order to define a vertical connection on $\q\colon E\to M$:
\begin{align}
\label{equation:vertical-connections-equifiers}
\begin{split}
&k\o l_q=\id_E\\
&q\o k=q\o p_E\\
&l_q\o k=\T k\o l_E\\
&l_q\o k=\T k\o c_E\T l_q
\end{split}
\end{align}
Let us start by imposing the equation $k\o l_q=\id_E$. Consider the equifier:
\begin{align*}
&(\E^k_1,\TT^k_1)\colon K_0\o\Univ{l_q}_{V^k_0}\to\id_{\Tot\o V^k_0}
\end{align*}
Concretely, $(\E^k_1,\TT^k_1)$ is a tangent category equipped with a strict tangent morphism $W_1^k\colon(\E^k_1,\TT^k_1)\to(\I^h_0,\TT^h_0)$ satisfying the following equation:
\begin{align*}
&(K_0)_{W_1^k}\o\Univ{l_q}_{V^k_0\o W_1^k}=\id_{\Tot\o V^k_0\o W_1^k}
\end{align*}
Thus, $(\E^k_1,\TT^k_1)$ is the tangent subcategory of $(\I^k_0,\TT^k_0)$ spanned by the objects $(\q,k_1)$, where $l_q\o k_1=\id_E$. In the following, let us denote by $V_1^k\=W_1^k\o V_0^k$ and $K_1\=(K_0)_{W_1^k}\colon\bar\T\Tot\o V_1^k\to\Tot\o V_1^k$. Next, we want to impose the equation $q\o k=q\o p$. Consider the equifier:
\begin{align*}
(\E^k_2,\TT^k_2)\colon\UnivQ_{V_1^k}\o K_1\to\o\bar p_{\Tot\o V_1^k}
\end{align*}
$(\E^k_2,\TT^k_2)$ is a tangent category equipped with a strict tangent morphism $W_2^k\colon(\E^k_2,\TT^k_2)\to(\E^k_1,\TT^k_1)$ satisfying the following equation:
\begin{align*}
&\UnivQ_{V^k_1\o W_2^k}\o(K_1)_{W_2^k}=\UnivQ_{V^k_1\o W_2^k}\o\bar p_{\Tot\o V_1^k\o W_2^k}
\end{align*}
Let denote by $V_2^k\=W_2^k\o V_1^k$ and $K_2\=(K_1)_{W_2^k}\colon\bar\T\Tot\o V_2^k\to\Tot\o V_2^k$. To impose the equation $l_q\o k=\T k\o l$, consider the equifier:
\begin{align*}
(\E^k_3,\TT^k_3)\colon\Univ{l_q}_{V_2^k}\o K_2\to\bar\T K_2\o\bar l_{\Tot\o V_2^k}
\end{align*}
$(\E^k_3,\TT^k_3)$ is a tangent category equipped with a strict tangent morphism $W_3^k\colon(\E^k_3,\TT^k_3)\to(\E^k_2,\TT^k_2)$ satisfying the following equation:
\begin{align*}
&\Univ{l_q}_{V_2^k\o W_3^k}\o(K_2)_{W_3^k}=\bar\T(K_2)_{W_3^k}\o\bar l_{\Tot\o V_2^k\o W_3^k}
\end{align*}
Let denote by $V_3^k\=W_3^k\o V_2^k$ and $K_3\=(K_2)_{W_3^k}\colon\bar\T\Tot\o V_3^k\to\Tot\o V_3^k$. Finally, to impose the equation $l_q\o k=\T k\o c\o\T l_q$, consider the equifier:
\begin{align*}
(\E^k_4,\TT^k_4)\colon\Univ{l_q}_{V_3^k}\o K_3\to\bar\T K_3\o\bar c_{\Tot\o V_3^k}\o\bar \T\Univ{l_q}_{V_3^k}
\end{align*}
$(\E^k_3,\TT^k_3)$ is a tangent category equipped with a strict tangent morphism $W_3^k\colon(\E^k_3,\TT^k_3)\to(\E^k_2,\TT^k_2)$ satisfying the following equation:
\begin{align*}
&\Univ{l_q}_{V_3^k\o W_4^k}\o(K_3)_{W_4^k}=\bar\T(K_3)_{W_4^k}\o\bar c_{\Tot\o V_3^k\o W_4^k}\o\bar \T\Univ{l_q}_{V_3^k\o W_4^k}
\end{align*}
$(\E^k_4,\TT^k_4)$ is precisely the tangent subcategory of $(\I^k_0,\TT^k_0)$ spanned by those objects $(\q,k)$ whose morphisms $k\colon\T E\to E$ satisfy all the Equations~\eqref{equation:vertical-connections-equifiers}.

\begin{theorem}
\label{theorem:PIE-limits-vertical-linear-connections}
Consider a tangentad $(\X,\TT)$ in a $2$-category $\CC$ which admits the construction of display differential bundles and suppose that the inserter $\I^k_0$ and the equifiers $\E^k_1\,\E^k_4$ exist in $\CC$. Thus, $(\X,\TT)$ admits the construction of vertical linear connections in $\CC$ and the tangentad $\VLC(\X,\TT)$ of vertical linear connections of $(\X,\TT)$ is the equifier $(\E^k_4,\TT^k_4)$.
\end{theorem}
\begin{proof}
The proof is fairly similar to the proof of~\cite[Theorem~4.8]{lanfranchi:tangentads-II}. The idea is to prove that for every tangentad $(\X',\TT')$ of $\CC$, the tangent category $\Equif(\X',\TT';\varphi_4^k,\psi_4^k)\cong[\X',\TT'\|\E_4^k,\TT^k_4]$, is isomorphic to the tangent category $\VLC[\X',\TT'\|\X,\TT]$ of vertical linear connections on the Hom-tangent category $[\X',\TT'\|\X,\TT]$, where $\varphi_4^k$ and $\psi_4^k$ denote the two $2$-morphisms equified by $(\E_4^k,\TT^k_4)$. One starts by observing that an object of $\Equif(\X',\TT';\varphi_4^k,\psi_4^k)$ consists of a lax tangent morphism $(F,\alpha)\colon(\X',\TT')\to(\E_4^k,\TT^k_4)$ such that $F_{\varphi_4^k}=F_{\psi_4^k}$. By postcomposing $(F,\alpha)$ with
\begin{align*}
&V_4^k\=W_4^k\o V_3^k=W_4^k\o{\dots}\o W_1^k\o V_0^k\colon(\E_4^k,\TT^k_4)\to(\I_0^k,\TT^k_0)
\end{align*}
we can equip $(F,\alpha)\o V_4^k$ with the tangent $2$-morphism $K_4\=(K_0)_{V_4^k}\colon\bar\T\Tot\o V_4^k\to\Tot\o V_4^k$. Using the properties of the equifiers, we can show that $K_4$ is a vertical linear connection on $(F,\alpha)\o V_4^k$ in the Hom-tangent category $[\X',\TT'\|\X,\TT]$. Moreover, by the universal properties of the inserter $\I_0^k$ and the equifiers $\E_1^k,\E_2^k,\E_3^k$, one can show that every vertical connection in the Hom-tangent category $[\X',\TT'\|\X,\TT]$ corresponds to an object of $\Equif(\X',\TT';\varphi_4^k,\psi_4^k)$. This correspondence extends to an isomorphism of tangent categories:
\begin{align*}
&\Equif(\X',\TT';\varphi_4^k,\psi_4^k)\cong\VLC[\X',\TT'\|\X,\TT]
\end{align*}
Finally, by using the universal property of the equifier $\E_4^k$, we obtain an isomorphism
\begin{align*}
&[\X',\TT'\|\E_4^k,\TT_4^k]\cong\Equif(\X',\TT';\varphi_4^k,\psi_4^k)\cong\VLC[\X',\TT'\|\X,\TT]
\end{align*}
Thus, $(\E_4^k,\TT_4^k)$ is the tangentad of vertical linear connections of $(\X,\TT)$.
\end{proof}

With a similar strategy, one can also construct horizontal linear connections with inserters and equifiers. Let denote by $(\E^h_1,\TT^h_1)\,(\E^h_4,\TT^h_4)$ the equifiers corresponding to the equations
\begin{align*}
&\T q\o h=\pi_1\\
&p_E\o h=\pi_2\\
&l_M\o h=\T h\o(z_E\times_{z_M}l_M)\\
&c_E\o\T l_q\o h=\T h\o(l_q\times_{z_M}z_{\T M})
\end{align*}
respectively.

\begin{theorem}
\label{theorem:PIE-limits-horizontal-linear-connections}
Consider a tangentad $(\X,\TT)$ in a $2$-category $\CC$ which admits the construction of display differential bundles and suppose that the inserter $\I^h_0$ and the equifiers $\E^h_1\,\E^h_4$ exist in $\CC$. Thus, $(\X,\TT)$ admits the construction of horizontal linear connections in $\CC$ and the tangentad $\HLC(\X,\TT)$ of horizontal linear connections of $(\X,\TT)$ is the equifier $(\E^h_4,\TT^h_4)$.
\end{theorem}
\begin{proof}
The proof is substantially the same as the proof of Theorem~\ref{theorem:PIE-limits-vertical-linear-connections}.
\end{proof}

To construct (full) linear connections, let us already assume that a tangentad $(\X,\TT)$ admits the constructions of vertical and horizontal linear connections and let us assume that the following $2$-pullback diagram exists in $\CC$:
\begin{equation}
\label{equation:linear-connections-PIE}
% https://q.uiver.app/#q=WzAsNCxbMCwxLCJcXFZMQyhcXFgsXFxUVCkiXSxbMSwxLCJcXERCKFxcWCxcXFRUKSJdLFsxLDAsIlxcSExDKFxcWCxcXFRUKSJdLFswLDAsIlxcTENfMChcXFgsXFxUVCkiXSxbMCwxLCJcXFVeXFxWTEMiLDJdLFsyLDEsIlxcVV5cXEhMQyJdLFszLDAsIlxccGlfXFxWTEMiLDJdLFszLDIsIlxccGlfXFxITEMiXSxbMywxLCIiLDEseyJzdHlsZSI6eyJuYW1lIjoiY29ybmVyIn19XV0=
\begin{tikzcd}
{\LC_0(\X,\TT)} & {\HLC(\X,\TT)} \\
{\VLC(\X,\TT)} & {\DB(\X,\TT)}
\arrow["{\pi_\HLC}", from=1-1, to=1-2]
\arrow["{\pi_\VLC}"', from=1-1, to=2-1]
\arrow["\lrcorner"{anchor=center, pos=0.125}, draw=none, from=1-1, to=2-2]
\arrow["{\U^\HLC}", from=1-2, to=2-2]
\arrow["{\U^\VLC}"', from=2-1, to=2-2]
\end{tikzcd}
\end{equation}
In the context of tangent categories, $\LC_0(\X,\TT)$ is the tangent category of triples $(\q;k,h)$ formed by a display differential bundle $\q$, a vertical linear connection $k$ of $\q$, and a horizontal linear connection $h$ of $\q$. However, there is no required compatibility between $k$ and $h$. To impose the right equations, once again, we turn to equifiers. We need to impose the equations
\begin{align*}
&k\o h=z_q\o q\o\pi_2\\
&s_E\o\left\<\xi_q\o\<k,p_M\>,h\o\<p_E,\T q\>\right\>=\id_{\T E}
\end{align*}
by means of two equifiers that we denote by $(\E^{(k,h)}_1,\TT^{(k,h)}_1)$ and $(\E^{(k,h)}_2,\TT^{(k,h)}_2)$, respectively.

\begin{theorem}
\label{theorem:PIE-limits-linear-connections}
Consider a tangentad $(\X,\TT)$ in a $2$-category $\CC$. Under the hypotheses of Theorems~\ref{theorem:PIE-limits-vertical-linear-connections} and~\ref{theorem:PIE-limits-horizontal-linear-connections} and assuming the existence of the $2$-pullback of Equation~\eqref{equation:linear-connections-PIE} and of the equifiers $\E^{(k,h)}_1$ and $\E^{(k,h)}_2$ in $\CC$, $(\X,\TT)$ admits the construction of linear connections. Moreover, the tangentad $\LC(\X,\TT)$ of linear connections of $(\X,\TT)$ is the equifier $(\E^{(k,h)}_2,\TT^{(k,h)}_2)$.
\end{theorem}

\begin{corollary}
\label{corollary:PIE-limits-linear-connections}
If $\CC$ admits the construction of differential bundles and it admits inserters and equifiers, then $\CC$ admits the constructions of vertical and horizontal linear connections. Moreover, if the $2$-pullback of Equation~\eqref{equation:linear-connections-PIE} is defined for each tangentad of $\CC$, then $\CC$ also admits the construction of linear connections.
\end{corollary}

To construct (vertical/horizontal) affine connections via inserters and equifiers, one needs to follow the same procedures of Theorems~\ref{theorem:PIE-limits-vertical-linear-connections},~\ref{theorem:PIE-limits-horizontal-linear-connections}, and~\ref{theorem:PIE-limits-linear-connections}, but replacing the universal pointwise display differential bundle $\UnivQ\colon\Tot\to\Base$ with the tangent bundle $\bar p\colon\bar\T\to\id_{(\X,\TT)}$ in the Hom-tangent category $[\X,\TT\|\X,\TT]$. In particular, the pullback of Equation~\eqref{equation:linear-connections-PIE} is replaced by the $2$-pullback diagram:
\begin{equation}
\label{equation:affine-connections-PIE}
% https://q.uiver.app/#q=WzAsNCxbMCwxLCJcXFZBQyhcXFgsXFxUVCkiXSxbMSwxLCIoXFxYLFxcVFQpIl0sWzEsMCwiXFxIQUMoXFxYLFxcVFQpIl0sWzAsMCwiXFxBQ18wKFxcWCxcXFRUKSJdLFswLDEsIlxcVV5cXFZBQyIsMl0sWzIsMSwiXFxVXlxcSEFDIl0sWzMsMCwiXFxwaV9cXFZBQyIsMl0sWzMsMiwiXFxwaV9cXEhBQyJdLFszLDEsIiIsMSx7InN0eWxlIjp7Im5hbWUiOiJjb3JuZXIifX1dXQ==
\begin{tikzcd}
{\AC_0(\X,\TT)} & {\HAC(\X,\TT)} \\
{\VAC(\X,\TT)} & {(\X,\TT)}
\arrow["{\pi_\HAC}", from=1-1, to=1-2]
\arrow["{\pi_\VAC}"', from=1-1, to=2-1]
\arrow["\lrcorner"{anchor=center, pos=0.125}, draw=none, from=1-1, to=2-2]
\arrow["{\U^\HAC}", from=1-2, to=2-2]
\arrow["{\U^\VAC}"', from=2-1, to=2-2]
\end{tikzcd}
\end{equation}
We denote by $({\I_0'}^k,{\TT'_0}^k)$, $({\E_1'}^k,{\TT'_1}^k)\,({\E_4'}^k,{\TT'_4}^k)$, and $({\I_0'}^h,{\TT'_0}^h)$, $({\E_1'}^h,{\TT'_1}^h)\,({\E_4'}^h,{\TT'_4}^h)$ the corresponding inserters and equifiers to construct $\VAC(\X,\TT)$ and $\HAC(\X,\TT)$ and $({\E_1'}^{(k,h)},{\TT'_1}^{(k,h)})$, $({\E_2'}^{(k,h)},{\TT'_2}^{(k,h)})$ the two remaining equifiers to construct $\AC(\X,\TT)$.

\begin{theorem}
\label{theorem:PIE-limits-vertical-affine-connections}
Consider a tangentad $(\X,\TT)$ in a $2$-category $\CC$ and assume that the inserter $\I_0'^k$ and the equifiers $\E_1'^k\,E_4'^k$ exists in $\CC$. Thus, $(\X,\TT)$ admits the construction of vertical affine connections in $\CC$. Moreover, the tangentad $\VAC(\X,\TT)$ of vertical affine connections of $(\X,\TT)$ is the equifier $(\E_4'^k,\TT_4'^k)$.
\end{theorem}

\begin{theorem}
\label{theorem:PIE-limits-horizontal-affine-connections}
Consider a tangentad $(\X,\TT)$ in a $2$-category $\CC$ and assume that the inserter $\I_0'^h$ and the equifiers $\E_1'^h\,\E_4'^h$ exists in $\CC$. Thus, $(\X,\TT)$ admits the construction of horizontal affine connections in $\CC$. Moreover, the tangentad $\HAC(\X,\TT)$ of horizontal affine connections of $(\X,\TT)$ is the equifier $(\E_4'^h,\TT_4'^h)$.
\end{theorem}

\begin{theorem}
\label{theorem:PIE-limits-affine-connections}
Consider a tangentad $(\X,\TT)$ in a $2$-category $\CC$. Under the assumptions of Theorems~\ref{theorem:PIE-limits-vertical-affine-connections} and~\ref{theorem:PIE-limits-horizontal-affine-connections} and assuming the existence of the $2$-pullback of Equation~\eqref{equation:affine-connections-PIE} and of the equifiers $\E'^{(k,h)}_1$ and $\E'^{(k,h)}_2$ in $\CC$, $(\X,\TT)$ admits the construction of affine connections. Moreover, the tangentad $\AC(\X,\TT)$ of affine connections of $(\X,\TT)$ is the equifier $(\E'^{(k,h)}_2,\TT'^{(k,h)}_2)$.
\end{theorem}

\begin{corollary}
\label{corollary:PIE-affine-linear-connections}
If $\CC$ admits inserters and equifiers, then $\CC$ admits the constructions of vertical and horizontal affine connections. Moreover, if the $2$-pullback of Equation~\eqref{equation:affine-connections-PIE} is defined for each tangentad of $\CC$, then $\CC$ also admits the construction of affine connections.
\end{corollary}

%__________________________________________________________________________
\subsection{Applications}
\label{subsection:examples-connections}
In Section~\ref{subsection:definition-tangentad}, we listed some examples of tangentads and in Sections~\ref{subsection:examples-differential-objects},~\ref{subsection:examples-differential-bundles}, and~\ref{subsection:examples-differential-bundles} we computed the constructions of vector fields, differential objects, and linear connections for tangent monads, tangent fibrations, tangent indexed categories, tangent split restriction categories, and showed how to extend these constructions to tangent restriction categories. In this section, we compute the constructions of linear and affine connections for these examples and show how to extend them to tangent restriction categories. Since the proofs are fairly similar to those for vector fields, we provide the final constructions without giving the details of the proofs. We also focus on (full) linear connections, since the construction of affine connections is fairly similar.

%__________________________________________________________________________
\subsubsection*{Tangent monads}
\label{subsubsection:connections-tangent-monads}
Let us start by considering tangent monads. We previously discussed that strong tangent morphisms lift to the tangent categories of linear connections. It is not hard to convince ourselves that every strong tangent monad $(S,\alpha)$ on a tangent category $(\X,\TT)$, that is, a tangent monad whose underlying tangent morphism is strong, lifts to $\LC(\X,\TT)$ as the tangent monad $\LC(S,\alpha)$.

\begin{theorem}
\label{theorem:connections-tangent-monads}
In the $2$-category $\TngMnd$ of tangent monads, each strong tangent monad $(S,\alpha)$ on a tangent category admits the construction of linear connections. Moreover, the tangentads of linear connections of $(S,\alpha)$ is the tangent monad $\LC(S,\alpha)$.
\end{theorem}

%__________________________________________________________________________
\subsubsection*{Tangent fibrations}
\label{subsubsection:connections-tangent-fibrations}
Since the underlying functor $\Pi\colon(\X',\TT')\to(\X,\TT)$ of a tangent fibration is a strict tangent morphism, hence strong, it lifts to a strict tangent morphism $\LC(\Pi)\colon\LC(\X',\TT')\to\LC(\X,\TT)$. However, since $\Pi$ is a tangent fibration, one can show that $\LC(\Pi)$ is also a tangent fibration.

\begin{lemma}
\label{lemma:connections-tangent-fibration}
Consider a (cloven) tangent fibration $\Pi\colon(\X',\TT')\to(\X,\TT)$ between two tangent categories. The strict tangent morphism
\begin{align*}
\LC(\Pi)&\colon\LC(\X',\TT')\to\LC(\X,\TT)
\end{align*}
which sends a linear connection $(\q';k',h')$ to linear connection $(\Pi(\q');\Pi(k'),\Pi(h'))$, is a (cloven) tangent fibration.
\end{lemma}

\begin{theorem}
\label{theorem:connections-tangent-fibrations}
The $2$-category $\Fib$ of (cloven) fibrations admits the constructions of linear connections. In particular, the tangentad of linear connections of a (cloven) tangent fibration $\Pi$ is the tangent fibration $\LC(\Pi)$.
\end{theorem}

%__________________________________________________________________________
\subsubsection*{Tangent indexed categories}
\label{subsubsection:connections-tangent-categories}
In this section, we employ the Grothendieck $2$-equivalence $\TngFib\simeq\TngIndx$ to compute the tangentad of linear connections for tangent indexed categories.

\begin{proposition}
\label{proposition:equivalence-connections}
Let $\CC$ and $\CC'$ be two $2$-categories and suppose there is a $2$-equivalence $\Xi\colon\Tng(\CC)\simeq\Tng(\CC')$ of the $2$-categories of tangentads of $\CC$ and $\CC'$. If a tangentad $(\X,\TT)$ of $\CC$ admits the construction of linear connections and $\LC(\X,\TT)$ denotes the tangentad of linear connections of $(\X,\TT)$, so does $\Xi(\X,\TT)$ and the tangentad of linear connections of $\Xi(\X,\TT)$ is $\Xi(\LC(\X,\TT))$.
\end{proposition}

Proposition~\ref{proposition:equivalence-connections} together with the Grothendieck $2$-equivalence $\TngFib\cong\TngIndx$ allows us to compute the constructions of linear connections of tangent indexed categories. Consider a tangent indexed category $(\X,\TT;\IND,\TT')$ whose underlying indexed category $\IND\colon\X^\op\to\Cat$ sends an object $M$ of $\X$ to the category $\X^M$ and a morphism $f\colon M\to N$ to the functor $f^\*\colon\X^N\to\X^M$. Define $\LC(\X,\TT;\IND,\TT')$ as follows:
\begin{description}
\item[Base tangent category] The base tangent category is $\LC(\X,\TT)$;

\item[Indexed category] The indexed category $\LC(\IND)\colon\LC(\X,\TT)^\op\to\Cat$ sends a linear connection $(\q;k,h)$ of $(\X,\TT)$ to the category $\LC^{(\q;k,h)}$ whose objects are tuples $(\q';k',h')$ formed by a differential bundle $\q'\in\DB^\q$ and by two morphisms:
\begin{align*}
k'&\colon\T'^EE'\to k^\*E'       &h'&\colon{\pi_1^E}^\*E'\times_{(\pi_1\o q)^\*M'}{\pi_2^{\T M}}^\*\T'^MM'\to h^\*\T'^EE'
\end{align*}
such that $(\q,\q';k,k';h,h')$ form a linear connection in the tangent category of elements of $\IND$. Morphisms of $\LC^{(\q;k,h)}$ are morphisms $(\varphi,\psi)\colon\q'\to\q''$ of $\DB^\q$ which commute with the connections.
\par Furthermore, $\LC(\IND)$ sends a morphism $(f,g)\colon(\q_1;k_1,h_1)\to(\q_2;k_2,h_2)$ of linear connections to the functor
\begin{align*}
(f,g)^\*&\colon\LC^{(\q_1;k_1,h_1)}\to\LC^{(\q_2;k_2,h_2)}
\end{align*}
which sends a tuple $(\q';k',h')$ to $((f,g)^\*\q';(f,g)^\*k',(f,g)^\*h')$, where $(f,g)^\*\q'=(\DB(\IND)(f,g))(\q)$ and:
\begin{align*}
&(f,g)^\*k'\colon\T^{E_1}g^\*E'\xrightarrow{\xi^g}(\T g)^\*\T^{E_2}E'\xrightarrow{(\T g)^\*k'}(\T g)^\*k_2^\*E'\xrightarrow{k_2\o\T g=g\o k_1}k_1^\*g^\*E'\\
&(f,g)^\*h'\colon{\pi_1^{E_1}}^\*g^\*E'\times_{(\pi_1^{E_1}\o q_1)^\*f^\*M'}{\pi_2^{\T M_1}}^\*\T^{M_1}f^\*M'\cong\\
&\qquad\cong(g\times_f\T f)^\*{\pi_1^{E_2}}^\*E'\times_{(g\times_f\T f)^\*(\pi_1^{E_2}\o q_2)M'}(g\times_f\T f)^\*{\pi_2^{\T M_2}}^\*\T^{M_2}M'\xrightarrow{(g\times_f\T f)^\*h'}(g\times_f\T f)^\*h^\*\T^{E_2}E'\cong\\
&\qquad\cong h^\*(\T g)^\*\T^{E_2}E'\xrightarrow{\xi^g}h^\*\T^{E_1}g^\*E'
\end{align*}
where we used the following equalities:
\begin{align*}
&g\o\pi_1^{E_1}=\pi_1^{E_2}\o(g\times_f\T f)       &&g\o\pi_2^{E_1}=\pi_2^{E_2}\o(g\times_f\T f)
&&f\o q_1\o\pi_1^{E_1}=q_2\o\pi_1^{E_2}\o(g\times_f\T f)
\end{align*}

\item[Indexed tangent bundle functor] The indexed tangent bundle functor $\LC(\T')$ consists of the list of functors
\begin{align*}
&\T'^{(\q;k,h)}\colon\LC^\q\to\LC^{\T^\LC(\q;k,h)}
\end{align*}
which send a tuple $(\q';k',h')$ to the tuple $\T^\DB\q';\T^{(\q;k,h)}k',\T^{(\q;k,h)}h')$ where:
\begin{align*}
&\T^{(\q;k,h)}k'\colon\T^{\T E}\T'^EE'\xrightarrow{c'}c^\*\T^{\T E}\T'^EE'\xrightarrow{\T'^Ek'}c^\*\T'^Ek^\*E'\xrightarrow{\xi_k}c^\*(\T k)^\*\T'^EE'\xrightarrow{\IND_2}k_\T^\*\T'^EE'\\
&\T^{(\q;k,h)}h'\colon\pi_1^\*\T'^EE'\times_{\T'^MM'}\pi_2^\*\T'^{\T M}\T'^MM'\xrightarrow{\id\times c'}(\T\pi_1)^\*\T'^EE'\times_{\T'^MM'}\pi_2^\*c^\*\T'^{\T M}\T'^MM'\cong\\
&\qquad\cong(\id_E\times_Mc)^\*\T'^{E\times_M\T M}(\pi_1^\*E'\times_{M'}\T'^MM')\xrightarrow{(\id_E\times_Mc)^\*\T'^{E\times_M\T M}h'}(\id_E\times_Mc)^\*\T'^{E\times_M\T M}h^\*\T'^EE'\mathdash\\
&\qquad\xrightarrow{\xi_h}(\id_E\times_Mc)^\*(\T h)^\*\T'^{\T E}\T'^EE'\xrightarrow{(\id_E\times_Mc)^\*(\T h)^\*c'}(\id_E\times_Mc)^\*(\T h)^\*c^\*\T'^{\T E}\T'^EE'\xrightarrow{\IND_2}h_\T^\*\T'^{\T E}\T'^EE'
\end{align*}
Moreover, $\T'^\q$ sends a morphism $(\varphi,\psi)\colon(\q'_1;k'_1,h'_1)\to(\q'_2;k'_2,h'_2)$ to $(\T'^M\varphi,\T'^E\psi)$. The distributors of $\T'^\q$ are pairs $(\xi^f,\xi^g)$ of distributors of $\T'^M$ and $\T'^E$;

\item[Indexed natural transformations] The structural indexed natural transformations of $\LC(\IND)$ are the same as for $\IND$.
\end{description}

\begin{theorem}
\label{theorem:connections-tangent-indexed-categories}
The $2$-category $\Indx$ of indexed categories admits the construction of linear connections. In particular, given a tangent indexed category $(\X,\TT;\IND,\TT')$, the tangent indexed category of linear connections of $(\X,\TT;\IND,\TT')$ is the tangent indexed category $\LC(\X,\TT;\IND,\TT')=(\LC(\X,\TT),\LC(\IND),\LC(\TT'))$.
\end{theorem}

%__________________________________________________________________________
\subsubsection*{Tangent split restriction categories}
\label{subsubsection:connections-tangent-split-restriction-categories}
In this section, we consider the construction of linear connections for tangent split restriction categories.

\begin{lemma}
\label{lemma:connections-tangent-split-restriction-categories}
Consider a tangent split restriction category $(\X,\TT)$ and let $\LC(\X,\TT)$ be the category whose objects are linear connections $(\q;k,h)$ in the tangent subcategory $\Tot(\X,\TT)$ and morphisms are morphisms $(f,g)$ of linear connections such that also their restriction idempotents $(\bar f,\bar g)$ are also morphisms of linear connections. The tangent split restriction structure on $(\X,\TT)$ lifts to $\LC(\X,\TT)$.
\end{lemma}

\begin{definition}
\label{definition:restriction-linear-connection}
A \textbf{restriction differential bundle} in a Cartesian tangent split restriction category $(\X,\TT)$ is an object of $\LC(\X,\TT)$.
\end{definition}

Concretely, a restriction differential bundle consists of a restriction differential bundle $\q\colon E\to M$ together with two total maps
\begin{align*}
k&\colon\T E\to E     &h&\colon E\times_M\T M\to\T E
\end{align*}
where $E\times_M\T M$ is the restriction pullback of $q$ along $p_M$:
\begin{equation*}
% https://q.uiver.app/#q=WzAsNCxbMCwwLCJFXFx0aW1lc19NXFxUIE0iXSxbMCwxLCJFIl0sWzEsMCwiXFxUIE0iXSxbMSwxLCJNIl0sWzEsMywicSIsMl0sWzIsMywicF9NIl0sWzAsMSwiXFxwaV8xIiwyXSxbMCwyLCJcXHBpXzIiXSxbMCwzLCIiLDEseyJzdHlsZSI6eyJuYW1lIjoiY29ybmVyIn19XV0=
\begin{tikzcd}
{E\times_M\T M} & {\T M} \\
E & M
\arrow["{\pi_2}", from=1-1, to=1-2]
\arrow["{\pi_1}"', from=1-1, to=2-1]
\arrow["\lrcorner"{anchor=center, pos=0.125}, draw=none, from=1-1, to=2-2]
\arrow["{p_M}", from=1-2, to=2-2]
\arrow["q"', from=2-1, to=2-2]
\end{tikzcd}
\end{equation*}
The total maps $k$ and $h$ satisfy the same equational axioms of a linear connection in ordinary tangent category theory.

\begin{theorem}
\label{theorem:connections-tangent-split-restriction-categories}
The $2$-category $\sRestrCat$ of split restriction categories admits the construction of linear connections. In particular, the tangentad of linear connections of a Cartesian tangent split restriction category $(\X,\TT)$ is the Cartesian tangent split restriction category $\LC(\X,\TT)$.
\end{theorem}

%__________________________________________________________________________
\subsubsection*{Tangent restriction categories: a general approach}
\label{subsubsection:connections-tangent-restriction-categories}
To extend the construction of linear connections to tangent restriction categories, we adopt a similar technique to the one used to extend the constructions differential bundles.
\par Consider a pullback-extension context, that is, two $2$-categories $\CC$ and $\DD$ together with two $2$-functors
\begin{align*}
&\Xi\colon\DD\leftrightarrows\Tng(\CC)\colon\Inc
\end{align*}
together with a natural $2$-transformation
\begin{align*}
\eta_\X&\colon\X\to\Inc(\Xi(\X))
\end{align*}
natural in $\X\in\DD$. Let us also consider an object $\X$ of $\DD$ such that the tangentad $\Xi(\X)$ of $\CC$ admits the construction of linear connections. Finally, let us assume the existence of the following $2$-pullback in $\Parall_2(\DD)$:
\begin{equation}
\label{equation:pullback-extension-connections}
% https://q.uiver.app/#q=WzAsNCxbMSwxLCJcXEluYyhcXFhpKFxcWCkpIl0sWzAsMSwiXFxYIl0sWzAsMCwiXFxMQyhcXFgpIl0sWzEsMCwiXFxJbmMoXFxMQyhcXFhpKFxcWCkpIl0sWzEsMCwiKFxcZXRhLFxcZXRhKSIsMl0sWzIsMSwiKFxcQmFzZV5cXExDX1xcWCxcXFRvdF5cXExDX1xcWCkiLDJdLFsyLDAsIiIsMSx7InN0eWxlIjp7Im5hbWUiOiJjb3JuZXIifX1dLFszLDAsIihcXEluYyhcXEJhc2VeXFxMQ197XFxYaShcXFgpfSksXFxJbmMoXFxUb3ReXFxMQ197XFxYaShcXFgpfSkpIl0sWzIsMywiXFxMQyhcXGV0YSxcXGV0YSkiXV0=
\begin{tikzcd}
{\LC(\X)} & {\Inc(\LC(\Xi(\X))} \\
\X & {\Inc(\Xi(\X))}
\arrow["{\LC(\eta,\eta)}", from=1-1, to=1-2]
\arrow["{(\Base^\LC_\X,\Tot^\LC_\X)}"', from=1-1, to=2-1]
\arrow["\lrcorner"{anchor=center, pos=0.125}, draw=none, from=1-1, to=2-2]
\arrow["{(\Inc(\Base^\LC_{\Xi(\X)}),\Inc(\Tot^\LC_{\Xi(\X)}))}", from=1-2, to=2-2]
\arrow["{(\eta,\eta)}"', from=2-1, to=2-2]
\end{tikzcd}
\end{equation}

\begin{definition}
\label{definition:connections-extension}
Let $(\CC,\DD;\Inc,\Xi;\eta)$ be a pullback-extension context. An object $\X$ of $\DD$ admits the \textbf{extended construction of linear connections} (w.r.t. to the pullback-extension context) if $\Xi(\X)$ admits the construction of linear connections in $\CC$, and the $2$-pullback diagram of Equation~\eqref{equation:pullback-extension-connections} exists in $\Parall_2(\DD)$. In this scenario, the \textbf{object of linear connections} (w.r.t. to the pullback-extension context) of $\X$ is the object $\LC(\X)$ of $\DD$, that is, the $2$-pullback of $(\Inc(\Base^\LC_{\Xi(\X)}),\Inc(\Tot^\LC_{\Xi(\X)}))$ along $(\eta,\eta)$.
\end{definition}

Consider a generic tangent restriction category $(\X,\TT)$ and let us unwrap the definition of $\LC(\Split_R(\X,\TT))$. The objects of $\LC(\Split_R(\X,\TT))$ are tuples $(\q;k,h)$ formed by a restriction differential bundle $\q\colon E\to M$ of $\Split_R(\X,\TT)$ together with two morphisms
\begin{align*}
k&\colon\T E\to E       &h&\colon E\times_M\T M\to\T E
\end{align*}
such that:
\begin{align*}
&\bar{k}=\bar q                     &&\bar q\o k=k\\
&\bar{h}=\bar q\times_M\T\bar z_q   &&\T\bar q\o h=h
\end{align*}
Moreover, $k$ and $h$ satisfy the axioms of a restriction differential bundle.
\par A morphism $(f,g)\colon(\q_1;k_1,h_1)\to(\q_2;k_2,h_2)$ of $\LC(\Split_R(\X,\TT))$ consists of a morphism $(f,g)\colon\q_1\to\q_2$ of restriction differential bundles which commutes with the structural morphism of the linear connections and with their restriction idempotents.
\par Now, consider a tangent (non-necessarily split) restriction category $(\X,\TT)$ and define $\LC(\X,\TT)$ to be the full subcategory of $\LC(\Split_R(\X,\TT))$ spanned by the objects $\q\colon E\to M$ where $q$, $z_q$, $s_q$, $l_q$, $k$, and $h$ are total maps in $(\X,\TT)$, which is precisely when $q$ and $z_q$ are total.

\begin{lemma}
\label{lemma:connections-tangent-restriction-categories}
Let $(\X,\TT)$ be a tangent restriction category. The subcategory $\LC(\X,\TT)$ of $\LC(\Split_R(\X,\TT))$ spanned by the objects $(\q;k,h)$ where $q$ and $z_q$ are total in $(\X,\TT)$ is a tangent restriction category.
\end{lemma}

We can finally prove the main theorem of this section.

\begin{theorem}
\label{theorem:connections-tangent-restriction-categories}
Consider the pullback-extension context $(\RestrCat,\TngRestrCat;\Inc,\Split_R;\eta)$ of tangent restriction categories. The $2$-category $\TngRestrCat$ admits the extended construction of linear connections with respect to this pullback-extension context. Moreover, the object of linear connections of a tangent restriction category is the tangent restriction category $\LC(\X,\TT)$.
\end{theorem}

%__________________________________________________________________________
%__________________________________________________________________________

%%%%%%%%%%%%%%%%%%%%%%%%%%% BIBLIOGRAPHY %%%%%%%%%%%%%%%%%%%%%%%%%%%%%%%%%%

\end{document}